\documentclass{amsart}
\usepackage{amsmath,amssymb,amsthm,graphicx,epsfig,float,url}
\usepackage[colorlinks=true,citecolor={Plum},linkcolor={Periwinkle}]{hyperref}
\usepackage{mathrsfs}  
\usepackage{pdfsync}
\usepackage[usenames, dvipsnames]{xcolor}
\usepackage{tikz}
\usepackage{subfig}
\usepackage{bbm}
\usepackage{stmaryrd}
\usepackage{amssymb}
\makeatletter
\let\std@footnotetext\@footnotetext
\usepackage{setspace}
\usepackage{caption}
\usepackage{float}
\usepackage{esint}

\usepackage{dsfont}
\usepackage[makeroom]{cancel}
\usetikzlibrary{patterns}

\newcommand{\diag}{\operatorname{diag}}
\newcommand{\bo}[1]{{\bf#1}}\renewcommand{\geq}{\geqslant}
\newcommand{\R}{\textnormal{I\kern-0.21emR}}
\newcommand{\N}{\textnormal{I\kern-0.21emN}}

\renewcommand{\geq}{\geqslant}
\renewcommand{\leq}{\leqslant}

\def\e{{\varepsilon}}

\allowdisplaybreaks

\def\Yint#1{\mathchoice
    {\YYint\displaystyle\textstyle{#1}}%
    {\YYint\textstyle\scriptstyle{#1}}%
    {\YYint\scriptstyle\scriptscriptstyle{#1}}%
    {\YYint\scriptscriptstyle\scriptscriptstyle{#1}}%
      \!\iint}
\def\YYint#1#2#3{{\setbox0=\hbox{$#1{#2#3}{\iint}$}
    \vcenter{\hbox{$#2#3$}}\kern-.51\wd0}}
\def\longdash{{-}\mkern-3.5mu{-}} 
\def\tiltlongdash{\rotatebox[origin=c]{15}{$\longdash$}}
\def\fiint{\Yint\tiltlongdash}

\usepackage[dvipsnames]{xcolor}

\newtheorem*{theorem*}{Theorem}

\newtheorem{theorem}{Theorem}

\newtheorem{material}{material}
\newtheorem{proposition}[material]{Proposition}
\newtheorem{corollary}[material]{Corollary}
\newtheorem{definition}[material]{Definition}
\newtheorem{lemma}[material]{Lemma}

\newtheorem{remark}[material]{Remark}

\def\O{{\Omega}}

\def\n{{\nabla}}

\def\p{{\varphi}}

\def\TT{{(0,T)\times \T}}
\def\TTo{{(0,T)\times \mathbb{T}}}

\def\T{{\mathbb T^d}}

\usepackage{xargs}% use more than one optional parameter in a new command
 \usepackage[colorinlistoftodos,textsize=small]{todonotes}
 \newcommandx{\christian}[2][1=]{\todo[linecolor=red,backgroundcolor=red!25,bordercolor=red,#1]{#2}}
 \newcommandx{\laura}[2][1=]{\todo[linecolor=blue,backgroundcolor=blue!25,bordercolor=blue,#1]{#2}}
 \newcommandx{\info}[2][1=]{\todo[linecolor=green,backgroundcolor=green!25,bordercolor=green,#1]{#2}}
 \newcommandx{\improvement}[2][1=]{\todo[linecolor=yellow,backgroundcolor=yellow!25,bordercolor=yellow,#1]{#2}}
 
  \newcommandx{\biblio}[2][1=]{\todo[linecolor=blue,backgroundcolor=magenta!25,bordercolor=blue,#1]{#2}}

 \numberwithin{equation}{section}

\begin{document}

\title{Optimisation of space-time periodic eigenvalues}

%    Remove any unused author tags.
\author{Beniamin Bogosel}
\address{Beniamin Bogosel, CMAP, CNRS, \'Ecole polytechnique, Institut Polytechnique de Paris, 91120
Palaiseau, France.}
\email{beniamin.bogosel@polytechnique.edu}
\author{ Idriss Mazari-Fouquer}
\address{Idriss Mazari-Fouquer, CEREMADE, UMR CNRS 7534, Universit\'e Paris-Dauphine, Universit\'e PSL, Place du Mar\'echal De Lattre De Tassigny, 75775 Paris cedex 16, France.}
\email{mazari@ceremade.dauphine.fr}
\author{Gr\'egoire Nadin}
\address{Gr\'egoire Nadin, Institut Denis Poisson, Universit\'e d’Orl\'eans, Universit\'e de Tours, CNRS, Orl\'eans, France.}
\email{gregoire.nadin@cnrs.fr}
\begin{abstract}The goal of this paper is to provide a qualitative analysis of the optimisation of space-time periodic principal eigenvalues. Namely, considering a fixed time horizon $T$ and the $d$-dimensional torus  $\mathbb{T}^d$, let, for any  $m\in L^\infty((0,T)\times\T)$, $\lambda(m)$ be the principal eigenvalue of the operator $\partial_t-\Delta-m$ endowed with (time-space) periodic boundary conditions. The main question we set out to answer is the following: how to choose $m$ so as to minimise $\lambda(m)$? This question stems from  population dynamics. We prove that in several cases it is always beneficial to rearrange $m$ with respect to time in a symmetric way, which is the first comparison result for the rearrangement in time of parabolic equations. Furthermore, we investigate the validity (or lack thereof) of Talenti inequalities for the rearrangement in time of parabolic equations. The numerical simulations which illustrate our results were obtained by developing a framework within which it is possible to optimise criteria with respect to functions having a prescribed rearrangement (or distribution function).
\end{abstract}

\maketitle
\textbf{Keywords:}  Non-symmetric operator, Principal eigenvalue, Rearrangements, Talenti inequalities, Spectral optimisation.

\textbf{MSC Classification:} 35P05,	35Q93, 49M41,  65N35.

\textbf{Acknowledgements:} The authors were partially supported by the ANR Project: STOIQUES. I. Mazari-Fouquer was partially supported by a PSL Young Researcher Starting Grant. 

\section{Introduction \& main results}

\subsection{Setting and objectives}

\subsubsection{Scope of the paper} Let $\T$ be the $d$-dimensional torus and $T>0$ be a time horizon. The primary focus of this paper is the optimisation, with respect to a function $m$ in $L^\infty((0,T)\times \T)$ naturally seen as $\T$-periodic in space and $T$-periodic in time, of the principal parabolic eigenvalue $\lambda(m)$ associated with the following system:
 \begin{equation}\label{Eq:Main1}
 \begin{cases}
 \frac{\partial u_m}{\partial t}-\Delta u_m=\lambda(m)u_m+mu_m&\text{ in }\TT\,, 
 \\ u_m(T,\cdot)=u_m(0,\cdot)\,, 
 \\ u_m\geq 0\,, \neq 0\,, \fiint_\TT u_m^2=1.\end{cases}\end{equation} Here we use the notation $\fint_\T=\frac1{|\T|}\int_\T$. This eigenvalue problem is well-studied  \cite{zbMATH03897446,zbMATH01525790}; its well-posedness follows from the Krein-Rutman theorem when $m$ is H\"{o}lder continuous \cite[Chapter 2]{zbMATH07668634}, or from a Perron-Frobenius approach \cite{zbMATH05834183}.
The main question under consideration is the following: \emph{how to choose $m$ so as to minimise the eigenvalue $\lambda$?} In other words, we are interested in the problem 
\[ \min_m\lambda(m)\] where $m$ is constrained to belong to a certain admissible set. Naturally we would need to specify the precise constraints, but, at a qualitative level, there are two main questions. The first one is the \emph{symmetry} of optimisers: is it true that it always better to replace $m$ with another potential (that satisfies the same constraints) but that is also symmetric in time and in space? The second one has to do with the \emph{monotonicity} of the optimisers: provided the answer to the first question is positive, is it true that the optimiser is not only symmetric, but also monotone? Let us emphasise that these questions are answered positively when considering only the symmetry and monotonicity with respect to the space variable; this result is due to Nadin \cite{zbMATH05834183} and relies on a class of elliptic Talenti inequalities (we refer to section \ref{Se:Bib}). On the other hand, to the best of our knowledge, no results exist for the symmetry and monotonicity in time of optimisers. 

In this paper, we do not provide a complete answer to these questions, but we offer a detailed analysis of several cases. To be more specific:
\begin{enumerate}
\item We first prove a general symmetry result: all optimisers of the first eigenvalue are symmetric in time. This is Theorem \ref{Th:Symmetry1}.
\item Second, we investigate the question of monotonicity in a variety of regimes. It is important to note that here we restrict our investigation to the following case: we assume that $m(t,x)=c(t)V(x)$, where $V$ is a fixed potential, and $c$ is the function to be optimised. This case will be referred to as "multiplicative potentials".
\begin{enumerate}
\item A first core result deals with the case where the eigenvalue problem is set in all of $\R^d$ and $V$ is a quadratic function. In this case we prove that it is always better to replace $c$ with its symmetric decreasing rearrangement (see Definition \ref{De:Rearrangement}); this is Theorem \ref{Th:Gaussian}.
\item A second class of results deals with the problem in the torus. In that case we show that in a singular perturbation limit and in a large heat operator limit,  it is always better to replace $c$ with its symmetric decreasing rearrangement. These are Theorems \ref{Th:LargeDiffusivity}-\ref{Th:SmallDiffusivity}.
\item We then show that, contrary to what happens in the elliptic case, these results can not be linked to Talenti inequalities, which goes to highlight the deep difficulties posed by this endeavour. This is Theorem \ref{Th:SymmetryBreaking}.
\item Finally, we develop a numerical framework within which it is possible to optimise eigenvalues (or other criteria) with respect to functions having a fixed rearrangement -- this allows to go beyond the usual thresholding algorithm, and the numerical simulations we obtain back both our results and conjectures. We refer to section \ref{Se:Numerics}.
\end{enumerate}
\end{enumerate}

\subsubsection{State of the art}\label{Se:Bib}
The questions at hand are at the crossroads of multiple active research lines; we briefly summarise the relevant works below.
\paragraph{{\em Time-space periodic principal eigenvalues}}
The theory of time-space periodic principal eigenvalues is well established. To the best of our knowledge, the investigation of this spectral problem was first initiated by Beltramo \& Hess \cite{zbMATH03897446}, who used it to study bifurcation results for non-linear parabolic equations. Observe that  \cite{zbMATH03897446} uses the Krein-Rutman approach to establish the existence of $\lambda(m)$. The standard work here is the monograph of Hess \cite{zbMATH00049232}, which assumes some regularity on $m$. Furthermore, this periodic eigenvalue can be seen as the principal eigenvalue of a degenerate, non-symmetric elliptic operator. Indeed, introduce, for any $\e>0$, the operator $L_\e:=-\e\partial^2_{tt}-\Delta+\partial_t-m$. One can show using elliptic methods that $L_\e$ has a principal eigenvalue $\lambda_\e(m)$, and it is readily seen that one can write $\lambda(m)=\lim_{\e \to 0}\lambda_\e(m)$. With this interpretation in mind  it is important to note that Pinsky \cite{zbMATH00764019} studied the existence and uniqueness of a principal periodic eigenvalue $\lambda_\e(m)$  and, additionally,  investigated the sensitivity of $\lambda_\e(m)$ under perturbations of the coefficients of $L_\e$. In this contribution as well, the coefficient of $L$ (in particular, $m$) are assumed to have some H\"{o}lder continuity. The question of existence of $\lambda(m)$ when $m$ is merely assumed to be $L^\infty$ was later revisited by Godoy, Lami  \& Paczka \cite{zbMATH01060813} and Daners \cite{zbMATH01028210,zbMATH01525790} using, for \cite{zbMATH01060813}, a semi-group approach and, for \cite{zbMATH01525790}, an adaptation of the approach of Beltramo \& Hess \cite{zbMATH03897446}. Regarding applications to population dynamics, Cantrell \& Cosner\cite{CantrellCosner} and Berestycki, Hamel \& Roques \cite{zbMATH02228673} showed that the study of principal eigenvalues was crucial in understanding the invasion of a species in space periodic environments, especially so for characterising the invasion speed. Their approach was later shown by Nolen, Rudd \& Xin \cite{zbMATH05011372} and Nadin  \cite{NadinPTW} to hold in the case of time-space periodic equations, and these methods were adapted to accommodate the study of space-time periodic solutions of reaction-diffusion equations in bounded domains by Nadin \cite{zbMATH05787082}.  Hutson, Mischaikow \& Pol\'a\v cik \cite{zbMATH01773170} and Hutson, Shen \& Vickers \cite{zbMATH01579419} investigated the qualitative dynamics of space-time periodic reaction-diffusion equations by using the interpretation of $\lambda(m)$ as a criterion for extinction or survival of a species; we will go back to a discussion of some of their results in a subsequent paragraph. We refer to \cite[Chapter 2]{zbMATH07668634} for a self-contained introduction to the topic, and we conclude by mentioning that a related approach that recently garnered a lot of interest from the mathematical community is based on the study of the principal Floquet bundle. We refer to \cite[Chapter 5]{zbMATH07668634}.

\paragraph{{\em Optimisation of eigenvalues of (non-)symmetric operators}}
Regarding the main qualitative questions of this paper, it is important to review some key contributions in spectral optimisation of elliptic and parabolic operators. A lot of work has been devoted to the understanding of the following optimisation problem:
\begin{multline*} \text{ Minimise }\mu(D,m) \\\text{ with respect to }m\in M:=\left\{m\in L^\infty(\O)\,,-1\leq m\leq 1\,, \fint_\Omega m=m_0\right\}\end{multline*} where $\mu(D,m)$ is the first eigenvalue of the operator $-D\Delta-m$ endowed with Dirichlet, Neumann or Robin boundary conditions in a smooth bounded domain $\Omega$. Similar to $\lambda(m)$, $\mu(D,m)$ provides a natural survival or extinction criterion for various classes of reaction-diffusion equations, and minimising $\mu$ amounts to answering the following question: what is the best distribution of resources to ensure the optimal survival of a population? This question emerged in the works of Cantrell \& Cosner \cite{CantrellCosner} and subsequently attracted much attention from several points of view. The first one is the question of the geometry of optimisers, while the second one is the influence of the diffusion coefficient  $D$. As regards the first question, \cite{CantrellCosner} provided a first qualitative result: in the one-dimensional case, it is better, in the case of periodic boundary conditions, to place resources at the center of the interval. Later contributions by Berstycki, Hamel \& Roques \cite{zbMATH02194918} and Kao, Lou \& Yanagida  \cite{zbMATH05530397}  showed more general symmetry properties: in $\T$, the optimal resources distribution for $\mu(D,\cdot)$ is symmetric decreasing in every direction. There are several proofs available, the most flexible one relying on rearrangement inequalities. We refer to \cite{zbMATH02194918} for further details, as well as to \cite{zbMATH06690453} for more general symmetry properties. A key aspect of these contributions is that $\mu(D,m)$ admits a variational formulation; this is due to the symmetry of the operator of $-D\Delta-m$. A natural question was then to understand the properties of spectral optimisation problems for non-symmetric operators, typically of the first eigenvalue $\Lambda(X,m)$ of
\[ -\Delta+\langle X,\n \cdot\rangle-m,\] in which case the eigenvalue has to be optimised with respect to both the drift $X$ and the potential $m$.   This problem was analysed by Hamel, Nadirashvili \& Russ \cite{zbMATH02153800,zbMATH05960716} (see also \cite{Alvino1990}). Nevertheless, their techniques can not be applied to parabolic optimisation problem, as it involves modifying the drift term. In other words, $\Lambda(X,m)$ can be compared with $\Lambda(X',m^\#)$ where $m^\#$ and $X'$ are suitably symmetrised versions of $m$ and $X$ respectively. However, for parabolic problems (seen as degenerate elliptic problems), the drift is the vector field $X=(1,0,\dots,0)$, the modified vector field would be $X'=(\mathds 1_{[0,T/2]}-\mathds 1_{[T/2;2]},0,\dots,0)$ and the modified operator is no longer parabolic. Nadin \cite{zbMATH05610626} showed a spatial symmetry result for the optimisation of the time-periodic eigenvalue $\lambda(m)$ (see \eqref{Eq:Main1}) in the one-dimensional case: letting for any $f:=f(t,x)$, $f^*(t,\cdot)$ be the symmetric decreasing rearrangement of $f$ in $x$ (\emph{i.e.} the only symmetric non-increasing function with the same repartition function as $f(t,\cdot)$, see Definition \ref{De:Rearrangement}) then 
\[ \lambda(m^*)\leq \lambda(m).\] His proof relies on a combination of Talenti inequalities and of the Perron-Frobenius construction of $\lambda(m)$. His proof does not extend to symmetry in time results. 

Regarding the second question, that is, the monotonicity property of $\mu(D,m)$ with respect to the diffusivity $D$, this is settled immediately using the variational formulation of $\mu(D,m)$. In the case of non-symmetric elliptic and parabolic operators, the question is more subtle. For the sake of simplicity, let us consider the first eigenvalue $\lambda(\tau,A,D)$ of the operator $\tau\partial_t-D\Delta+A\langle X,\n\rangle-m$. The monotonicity with respect to $D$ is false in general, as established by Nadin \& Carr\`ere \cite{zbMATH07226744}. The monotonicity with respect to $\tau$ is correct when $X\equiv 0$; this result is due to Liu, Lou, Peng \& Zhou \cite{zbMATH07122717}.

 \subsection{Mathematical statement of the problem}
 Throughout this paper, we use the notations 
 \[ \fint_\T f=\frac1{|\T|}\int_\T f\,, \fint_0^Tf=\frac1T\int_0^Tf\,, \fiint_\TT f=\frac1{T\cdot|\T|}\iint_\TT f.\]
 
 We let $(u_m,\lambda(m))$ denote the unique eigenpair associated with \eqref{Eq:Main1}. The potential $m$, unless otherwise stated, belongs to the set of $T$-periodic in time, $\T$-periodic in space, bounded measurable functions. While we will investigate rearrangement inequalities for the map $m\mapsto \lambda(m)$ we will also study optimisation problems.  To properly state the admissible classes and questions we need to introduce basic notions of symmetric rearrangements.

\subsubsection{Symmetric rearrangement} 
We first recall the definition of symmetric rearrangement of functions: for $f\in L^\infty((0,T))$ there exists a unique  $f^\#\in L^\infty((0,T))$ with the following three properties:
\begin{enumerate}
\item $f^\#$ is symmetric with respect to $\frac{T}2$,
\item $f^\#$ is non-increasing in $[T/2,T]$, 
\item $f^\#$ has the same distribution function as $f$: for any $\xi\in \R$, $|\{f\geq \xi\}|=|\{f^\# \geq \xi\}|$.
\end{enumerate}
\begin{definition}[Symmetric decreasing rearrangement]\label{De:Rearrangement}
For a function $f:[0,T]\to \R$, the function $f^\#$ is called the symmetric decreasing rearrangement of $f$. When $f$ is $T$-periodic, $f^\#$ is also understood to be extended by $T$-periodicity.\end{definition} We refer for further details on rearrangements to \cite{zbMATH05042914}.
In the sequel, for a function  $m\in L^\infty((0,T)\times \T)$, we will denote by $m^\#$ the rearrangement of the function $t\mapsto m(t,x)$, with $x\in \TT$ given.

\subsubsection{Admissible classes and the optimisation problems}
We will consider two admissible classes. The first one is 
 \begin{equation}\label{Eq:Adm}\mathcal M:=\left\{m\in L^\infty((0,T)\times \T)\,, -\kappa\leq m\leq 1\text{ a.e., } \fiint_\TT m=m_0\right\},\end{equation}
where $\kappa\geq 0$, and $m_0\in (-\kappa,1)$. The first optimisation problem that will be of importance to us is 

\begin{equation}\label{Pv:1}
 \min_{m\in \mathcal M} \lambda(m).\end{equation} The existence of an optimiser is a consequence of the direct method in the calculus of variations. Regarding this problem we will show (Theorem \ref{Th:Symmetry1}) that any optimiser $m_{\mathrm{min}}$ is symmetric in time and in space. Nevertheless, observe that it is not clear that $m_{\mathrm{min}}^\#=m_{\mathrm{min}}$, and we thus devote some effort to understanding the following question: is it always true that $\lambda(m^\#)\leq \lambda(m)$? 

For this purpose, we need 
to introduce the following comparison relation between functions: for two $T$-periodic functions $c_1\,, c_2:[0,T]\to \R$, we say that $c_1$ is dominated by $c_2$ and we write 
\[ c_1\preceq c_2\] if, and only if,
\begin{equation}\label{Eq:Domination}\forall s\in [0,T]\,, \sup_{|E|=s} \int_E c_1\leq \sup_{|E|=s} \int_E c_2\text{ and }\int_0^T c_1=\int_0^T c_2.\end{equation}

As for any $s\in [0,T]$ and any $c\in L^\infty((0,T))$ we have 
\[\sup_{|E|=s}\int_E c=\int_{\frac{T-s}2}^{\frac{T+s}2} c^\#\] we have the following lemma:
\begin{lemma}
For any $c_1\,, c_2\in L^\infty((0,T))$, $c_1\preceq c_2$ if, and only if, 
\[\forall r\in [0,T]\,, \int_{\frac{T-r}2}^{\frac{T+r}2} c_1^\#\leq \int_{\frac{T-r}2}^{\frac{T+r}2} c_2^\#\] and 
\[ \int_0^T c_1^\#=\int_0^T c_2^\#.\]\end{lemma}

We are now in position to introduce our second admissible class. Let $L^\infty_{\mathrm{per}}((0,T))$ be the set of $T$-periodic bounded measurable functions. Let $\bar{c}\in L^\infty_{\mathrm{per}}((0,T))$ be such that 
\[ \bar{c}=\bar{c}^\#.\] We define
\begin{equation}\label{Eq:Adm2}
\mathcal C:=\left\{ c\in L^\infty_{\mathrm{per}}((0,T))\,, c\preceq \bar c\right\}.\end{equation} 
\begin{remark}\label{Re:Classical}
An important case is $\bar c=a+(b-a)\mathds 1_{\left[\frac{T-\ell}2,\frac{T+\ell}2\right]}$ for $a<b$ two real numbers. In this case the class $\mathcal C$ can be described more explicitly as $\mathcal C=\{c:[0,T]\to \R\,, a\leq c\leq b\,, \int_0^T c=aT+\ell(b-a)\}$. This is not important for the mathematical analysis, but it can be in the numerical schemes. In section \ref{Se:Numerics}, we refer to this case as the ``classical" one, as this is the one most often considered in the literature. The core of section \ref{Se:Numerics} is to develop efficient methods when $\bar c$ is not a characteristic function.
\end{remark}

For a fixed potential $V=V(x)$, we let $\lambda(cV)$ be the first eigenvalue of \eqref{Eq:Main1} with $m(t,x)=c(t)V(x)$. When no confusion is possible,  $\lambda(c)$ denotes $\lambda(cV)$. The second problem under consideration is 
\begin{equation}\label{Pv:2}
\min_{c \in \mathcal C}\lambda(c).\end{equation} Here again, the existence of an optimal $c_{\mathrm{min}}$ is a consequence of the direct method of the calculus of variations. 
\begin{remark}\label{Re:Topological}
The topological properties of $\mathcal C$ are well-known: $\mathcal C$ is convex and compact for the weak $L^\infty-*$ topology; we refer to \cite{Alvino1989}. Furthermore, the extreme points of $\mathcal C$ are exactly the functions $c$ that satisfy $c^\#=\bar c$. This has an important consequence: as it is straightforward to see that $c\mapsto \lambda(c)$ is strictly concave, any solution $c_{\min}$ of \eqref{Pv:2} satisfies $c_{\min}^\#=\bar c$.
\end{remark}
While studying \eqref{Pv:2} we will also be led to proving in some regimes that 
\[ \lambda(c^\#)\leq \lambda(c).\]
It is important to note that a core result of the paper is the study of the Gaussian case, when $V$ is a quadratic potential and the equation is set in all of $\R^d$, where we can prove more precise results.

\subsection{Main results}
\subsubsection{Symmetry in time of optimal potentials for \eqref{Pv:1}}
We are interested in symmetry and monotonicity (in $t$ and $x$) of the solutions $m_{\mathrm{min}}$ of \eqref{Pv:1}. As mentioned earlier, the symmetry in $x$ of the optimiser can be obtained fairly easily using rearrangement arguments \cite[Theorem 3.9]{zbMATH05610626}. 

\begin{remark} As, for any $\tau>0$, $\lambda(m(\cdot+\tau,\cdot))=\lambda(m)$, so that symmetry in time has to be understood as: symmetry up to a translation.\end{remark}

 \begin{theorem}\label{Th:Symmetry1}
 Any solution $m_{\mathrm{min}}$  of \eqref{Pv:1} is symmetric in time: up to a translation, there holds, for any $t\in [0,T/2]$, for any $x\in \O$, $m_{\mathrm{min}}(t,x)=m_{\mathrm{min}}(T-t,x)$. Similarly, for a fixed $V$, any solution $c_{\mathrm{min}}$ of \eqref{Pv:2} is symmetric in time: up to a translation, there holds, for any $t\in [0,T/2]$, $c_{\mathrm{min}}(t)=c_{\mathrm{min}}(T-t)$.
 \end{theorem}
 This theorem will be obtained as a corollary of the following variational formulation of the eigenvalue which in turn relies on a  Heinze-type formula \cite[Proposition 2.6]{zbMATH05834183}.
  \begin{proposition}\label{Pr:Variational}
 For any $m\in L^\infty((0,T)\times \O)$, the following holds:
 \begin{multline}\label{Eq:Variational} \lambda(m)=\min_{\alpha\in L^2(0,T;W^{1,2}(\O))\cap W^{1,2}(0,T;L^2(\O))\,, \iint_\TT \alpha^2=1 } \left(\iint_{(0,T)\times \O}  |\n_x\alpha|^2\right.\\\left.-\iint_\TT m\alpha^2-\min_{\beta \in L^2(0,T;W^{1,2}(\T)) }\iint_\TT \left(-\beta\partial_t\alpha^2+|\n_x\beta|^2\alpha^2\right)\right).\end{multline}
 \end{proposition}
 
 This proposition, coupled with an elementary concavity property shows that the solution $m_{\mathrm{min}}$ of \eqref{Pv:1} satisfies the so-called bang-bang property or, in other words, that there exists $E\subset \TT$ such that 
 \[ m_{\mathrm{min}}(t,x)=(1+\kappa)\mathds 1_E(t,x)-\kappa,\]and the set $E$ is symmetric in time.

\subsubsection{The influence of time rearrangement for multiplicative potentials: Gaussian potentials}
In this section we investigate in more details the question of the influence of time rearrangements: is it always beneficial to replace $c$ with $c^\#$? We first study the Gaussian case, before investigating several limiting regimes in the initial framework of $(0,T)\times \T$ periodic functions.

We first study a related problem set in all of $\R^d$, which we dub the Gaussian case, as eigenfunctions are always of Gaussian type.

\begin{theorem}\label{Th:Gaussian}
Let $A\in S_d^{++}(\R)$ be a symmetric positive definite matrix and $c\in L^\infty_{\mathrm{per}}((0,T))$. Let $\bar\lambda(c)\,, \bar c\geq 0$ be the first eigenvalue associated with the problem 
\begin{equation}\label{Eq:EigenGaussian}
\begin{cases}
\partial_t w_0-\Delta w_0=\overline\lambda(c)w_0-\frac12c(t)\langle x,Ax\rangle w_0&\text{ in }[0,T]\times \R^d\,, 
\\ w_0(T,\cdot)=w_0(0,\cdot)\,, 
\\ w_0\neq 0
\\ 0\leq w_0\leq w_0(0,0)=1.
\end{cases}\end{equation} Then there holds \[\overline\lambda(c)\geq \overline\lambda(c^\#),\] with equality if and only if $c=c^\#$ up to translation. 
\end{theorem}

We could not locate the proof of existence of a solution to \eqref{Eq:EigenGaussian} in the literature; we provide it. Based on this case, we conjecture that optimisers should always be monotone decreasing but we are not able to prove it at the moment. Note however that this result provides some justification for our conjecture that in general rearranging $c$ should be beneficial for the eigenvalue. Another intuitive justification behind this conjecture is to be found in the monotonicity in frequency of $\lambda$. To be more precise, consider a profile $c$ that is not constant in time and, for any $k\in \N$, let $c_k(t):=c(kt)$. A theorem of Liu, Lou, Peng \& Zhou \cite{zbMATH07122717} asserts that $(\lambda(c_k V))_{k\in \N}$ is an increasing sequence for $V\in L^\infty (\TT)$. Consequently, oscillations are detrimental to spectral optimisation, whence a natural conjecture is that the optimiser should be ``as monotonic" as possible. Nevertheless, while our results lends credence to the link between monotonicity in frequency and monotonicity of optimisers in this case, we will see that, unlike monotonicity in space, monotonicity in time can not a priori be related to an isoperimetric type-property, see Theorem \ref{Th:SymmetryBreaking}.

\subsubsection{Monotonicity in time for two scaling limits of the equation}
 We are able to show that $c_{\mathrm{min}}$, the solution of \eqref{Pv:2} is monotonic in time or, put otherwise, that $c_{\mathrm{min}}=c_{\mathrm{min}}^\#$, in two scaling limits of the equation.
 
 \textbf{First scaling limit: large heat operator} We first replace \eqref{Eq:Main1} with 
 \begin{equation}\label{Eq:Main2}
 \begin{cases}
 \mu\partial_t u_{c,\mu}-\mu \Delta u_{c,\mu}=cVu_{c,\mu}+\lambda(c,\mu)u_{c,\mu}&\text{ in }\TT\,, 
 \\ u_{c,\mu}(0,\cdot)=u_{c,\mu}(T,\cdot)\,, 
 \\ u_{c,\mu}\geq 0\,, \neq 0,\end{cases}\end{equation}
 and we are interested in 
 \begin{equation}\label{Eq:PvLargeDiffusivity}
 \min_{c\in \mathcal C}\lambda(c,\mu).\end{equation}
 \begin{theorem}\label{Th:LargeDiffusivity}
Let $V\in L^\infty(\T)$ be fixed  and, for any $\mu>0$, $c_{\mathrm{min}}^\mu$ be a solution of \eqref{Eq:PvLargeDiffusivity}.  Assume that $V$ is not constant. Then there exists $\mu_0>0$ such that, for any $\mu\geq \mu_0$, up to a translation:
\[ c_{\mathrm{min}}^\mu=(c_{\mathrm{min}}^\mu)^\#=\bar c.\]
\end{theorem}

\textbf{Second scaling limit:  singular perturbations} 
In the next result we assume the following:
\begin{equation}\label{Hyp:VNnDeg}\tag{$\bold{H}_V$}
\text{$V\in \mathscr C^2(\mathbb{T}^d)$, $x^*=0$ is the only local maximiser of $V$ in $\mathbb{T}^d$ and $\n^2 V(0)\in S_d^{--}(\R)$.}
\end{equation}
Here, $ S_d^{--}(\R)$ denotes the class of symmetric negative definite matrix. For any $\e>0$, let $u_{c,\e}$ be the eigenfunction of
\begin{equation}\label{Eq:Main2}
 \begin{cases}
\e\partial_t u_{c,\e}-\e^2 \Delta u_{c,\e}=cVu_{c,\e}+\lambda_\e(c)u_{c,\e}&\text{ in }\TTo\,, 
 \\ u_{c,\e}(0,\cdot)=u_{c,\e}(T,\cdot)\,, 
 \\ u_{c,\e}\geq 0\,, \neq 0.\end{cases}\end{equation}
We consider
 \begin{equation}\label{Eq:PvLowDiffusivity}
 \min_{c\in \mathcal C}\lambda_\e(c).\end{equation}

\begin{theorem}\label{Th:SmallDiffusivity}Assume $\bar c\geq 0$. 
Let $V$ be fixed and satisfy \eqref{Hyp:VNnDeg} and, for any $\mu>0$, $c_{\mathrm{min}}^\e$ be a solution of \begin{equation}\label{Eq:Pv3}\min_{c\in \mathcal C}\lambda_\e(c).\end{equation}
 Up to a translation, \[ c_{\mathrm{min}}^\e\underset{\e\to0^*}\rightarrow\bar c=\bar c^\#\text{ in }L^p((0,T))\text{ for any }p\in [1;+\infty).\]\end{theorem}

\subsection{Some comments on our proofs -- Can our results be linked to time-periodic Talenti inequalities?}
Both Theorems \ref{Th:SmallDiffusivity} and \ref{Th:LargeDiffusivity} ultimately rely on elliptic Talenti inequalities; we show that the optimisation problems can be recast as an energy minimisation for an elliptic problem. Nevertheless, the two limit problems obtained are different in nature: in Theorem \ref{Th:LargeDiffusivity}, the limit problem is  a linear heat equation that can be dealt with using Fourier series, and the standard Talenti inequalities. In Theorem \ref{Th:SmallDiffusivity}, the main tool is the Hopf-Cole transform, which allows to relate \eqref{Pv:2} to a  problem set in all of $\R^d$ using the theory of singular perturbations; we build on some geometric estimates of \cite{zbMATH05041278}. As is often the case in singular perturbation problems, assumptions such  as \eqref{Hyp:VNnDeg} are required. It is worth observing that in proving Theorem \ref{Th:SmallDiffusivity} we will use Theorem \ref{Th:Gaussian}, which also relies on elliptic Talenti inequalities.  This begs the question of the validity of Talenti inequalities for time rearrangements of parabolic equations.

To be more specific, Theorem \ref{Th:LargeDiffusivity} relies on the following fact: we let, for any $c\in \mathcal C$, $\p_c$ be the unique solution of 
\begin{equation}\label{Eq:Main3}
\begin{cases}
\partial_t\p_c-\Delta\p_c=cV-\left(\fint_0^T c\cdot\fint_\T V\right)\,, 
\\ \p_c(T,\cdot)=\p_c(0,\cdot).\end{cases}\end{equation} 
Then we prove (see the proof of Proposition \ref{Pr:LimitProblem}) that 
\[ \fiint_\TT \p_c^2\leq \fiint_\TT \p_{\bar c}^2,\]using a Fourier series decomposition and applying elliptic Talenti inequalities.
A natural question is whether this inequality  holds for norms other than the $L^2$ norm, in which case this might give hope to derive a general monotonicity result using an isoperimetric inequality approach. We investigate the optimisation problem
\begin{equation}\label{Eq:PvHamiltonian}
\max_{c\in \mathcal C}J_k(c):=\iint_{(0,T)\times \T} \p_c^{2k}.\end{equation}  We will in fact show that $J_k(c)\leq J_k(c^\#)$ for all $c$ and $V$ if, and only if, $k=1$.

Before stating our precise result, let us briefly comment on Talenti inequalities in the elliptic case: the original elliptic Talenti inequality \cite{zbMATH03531830} reads as follows: letting, for any $f$, $f^*$ be its symmetric decreasing rearrangement and $\Omega$ be the unit ball, there holds, for any increasing function $J$, 
\[ \int_\O J(u_f)\leq \int_\O J(u_{f^*})\] where $u_f$ solves $-\Delta u_f=f\,,  u_f\in W^{1,2}_0(\O)$. Vazquez \cite{zbMATH03789116} extended the approach of Talenti to heat equations set in the ball and proved that, retaining the notation $f^*$ for the rearrangement in space of a function $f=f(t,x)$, the following comparison holds: for any increasing convex function $J$, there holds 
\[ \iint_{(0,T)\times \Omega} J(u_f)\leq \iint_{(0,T)\times \O} J(u_f^*)\] where $\partial_t u_f-\Delta u_f=f\,, u(0,\cdot)=0\,, u(t,\cdot)\in W^{1,2}_0(\O)$. In fact, using the domination relation defined in \eqref{Eq:Domination}, he proves that
\[ u_f(t,\cdot)\preceq u_{f^*}(t,\cdot)\] for any $t$.  Alvino, Lions \& Trombetti proposed systematic approaches to this question \cite{Alvino1986,Alvino1990,alvino1991,AlvinoTrombettiLions,Alvino1989,Alvino2010} and established various comparison results, albeit always comparing the solution of a parabolic equation with the solution of a similar equation whose coefficients are suitably symmetrised in space. The question of symmetry in time, and of whether it was possible to have such strong results (namely, that there should be a univocal notion of time symmetrisation that would provide comparison principles) was partially settled in \cite{zbMATH07517798} where the second author shows that it is in general hopeless. Nevertheless, \cite{zbMATH07517798} deals with initial value problems rather than with time-periodic solutions, and so this analysis can not be satisfactory for our needs. In this sense, Theorem \ref{Th:SymmetryBreaking} provides a conclusive (negative) answer and highlights the exceptionality of the $L^2$ norm in parabolic equations.

\bigskip

Our main theorem is the following:
\begin{theorem}\label{Th:SymmetryBreaking}We let $T=2\pi$, $V(x)=\cos(x)$ and  $\bar c=\cos(\pi+t)$. $\bar c$ solves \eqref{Eq:PvHamiltonian} if, and only if, $k=1$.\end{theorem}
\begin{remark}\label{Re:CounterExample}
We can not prove theoretically a symmetry breaking result for ``classical" constraints on $c$ (see Remark \ref{Re:Classical}). It seems from numerical simulation that $\bar c$ might actually be a maximiser. We refer to section \ref{Se:Numerics}.\end{remark}

\section{Proof of Theorem \ref{Th:Symmetry1}}
The proof of Theorem \ref{Th:Symmetry1} relies on Proposition \ref{Pr:Variational}. Observe that a version of \eqref{Eq:Variational} was previously obtained in \cite[Proposition 2.6]{zbMATH05834183}, and uses the Heinze change of variables \cite{Heinze}. 
\begin{proof}[Proof of Proposition \ref{Pr:Variational}]
Let $u_m$ be the solution of \eqref{Eq:Main1} and $v_m$ be the principal eigenfunction of the adjoint operator, that is:
\begin{equation}\label{Eq:Main1Adjoint}\begin{cases}
-\partial_t v_m-\Delta v_m=\lambda(m)v_m+mv_m&\text{ in }\TT\,, 
\\ v_m(0,\cdot)=v_m(T,\cdot),
\\ v_m\geq 0\,, \neq 0.\end{cases}\end{equation} We work with the normalisation 
\begin{equation}\label{Eq:Normalisation} \iint_\TT u_m^2=\iint_\TT u_mv_m=1.\end{equation}
Define 
$\alpha_m:=\sqrt{u_mv_m}\,, \beta_m:=\frac12\left(\ln(u_m)-\ln(v_m)\right).$
We use the Hopf-Cole transform by writing 
$ u_m=e^{-\p_m+a_m}\,, v_m=e^{-\psi_m+b_m},$ where we assume that 
$\min_\TT \p_m=\min_\TT \psi_m=0$ and $a_m\,, b_m$ are chosen to satisfy \eqref{Eq:Normalisation}.
Then $(\p_m\,, \psi_m)$ satisfy 
\begin{equation}\label{Eq:HopfCole}
\begin{cases}
\partial_t \p_m-\Delta \p_m+|\n \p_m|^2=-\lambda(m)-m&\text{ in }\TT\,, 
\\ -\partial_t \psi_m-\Delta \psi_m+|\n \psi_m|^2=-\lambda(m)-m&\text{ in }\TT\,, 
\\ \p_m(0,\cdot)-\p_m(T,\cdot)=\psi_m(T,\cdot)-\psi_m(0,\cdot)=0.\end{cases}\end{equation} With these notations, we have
$ \alpha_m=e^{-\frac{\p_m+\psi_m}2}\,, \beta_m=\frac{\psi_m-\p_m}2.$
Substracting the two equations of \eqref{Eq:HopfCole} we obtain 
\[\partial_t \frac{\p_m+\psi_m}2-\Delta\frac{\p_m-\psi_m}2+\left\langle \n\frac{\p_m-\psi_m}2,\n(\p_m+\psi_m)\right\rangle=0.\]
This rewrites
\begin{equation}\label{Eq:Emil}- \frac12\partial_t\alpha_m^2-\n\cdot(\alpha_m^2\n\beta_m)=0.\end{equation}Similarly, summing the equations in \eqref{Eq:HopfCole} and dividing by 2 we deduce that
\begin{multline*}\partial_t\frac{\p_m-\psi_m}2-\Delta\frac{\p_m+\psi_m}2+\left|\n \frac{\p_m+\psi_m}2\right|^2\\+\frac{|\n\p_m|^2+|\n \psi_m|^2-2\langle\n\p_m,\n\psi_m\rangle}4=-\lambda(m)-m\end{multline*}or, using a Hopf-Cole transform, 
\[-\Delta\alpha_m=\left(\lambda(m)+m+\partial_t\beta_m+|\n\beta_m|^2\right)\alpha_m.\]
We deduce the following from \eqref{Eq:Emil}: for any $t\in[0,T]$, $\beta_m$ is the unique minimiser of 
\[ W^{1,2}(\T)\ni\beta\mapsto\int_\T \alpha_m^2|\n\beta|^2-\frac12\int_\T \beta\partial_t\alpha_m^2,\]whence $\beta_m$ is the unique minimiser of 
\[ L^2(0,T;W^{1,2}(\T))\ni\beta\mapsto E(\beta,\alpha_m):=\frac12\iint_\TT \alpha_m^2|\n\beta|^2-\frac12\iint_\TT \beta\partial_t\alpha_m^2,\]
To alleviate notations, define 
\[ A(\TT):=L^2(0,T;W^{1,2}(\T))\cap W^{1,2}(0,T;L^2(\T)).\]
Since $\beta_m\in W^{1,2}(0,T;L^2(\T))$ by construction, this shows that $\beta_m$ solves \[ \min_{\beta \in A(\TT)}\frac12\iint_\TT\left(\partial_t\beta+|\n\beta|^2\right)\alpha_m^2.\]
Similarly, as $\alpha_m>0$ by construction, $\alpha_m$ is the unique minimiser of 
\[  \min_{\alpha \in A(\TT)\,, \iint_\TT \alpha^2=1}\iint_\TT|\n \alpha|^2-\iint_\TT (m+\partial_t\beta_m+|\n\beta_m|^2)\alpha^2.\] Consequently, $(\alpha_m,\beta_m)$ is a solution of the saddle point problem
\begin{multline*}
\min_{\alpha A(\TT)\,, \iint_\TT \alpha^2=1}\\\max_{\beta \in L^2(0,T;W^{1,2}(\T))\cap W^{1,2}(0,T;L^2(\T))}\iint_\TT|\n \alpha|^2-\iint_\TT (m+\partial_t\beta_m+|\n\beta_m|^2)\alpha^2.
 \end{multline*}
Conversely, if $(\tilde\alpha,\tilde\beta)$ is a solution of this saddle point problem with associated value $\Lambda$, direct computations show that, setting $\gamma:=-\ln(\alpha)$, we have 
\[ \partial_t(\gamma+\tilde\beta)-\Delta(\gamma+\tilde\beta)+|\n(\gamma+\tilde\beta)|^2=-\Lambda-m\] whence $e^{-(\gamma+\tilde\beta)}> 0$ solves \eqref{Eq:Main1} with eigenvalue $\Lambda$. The uniqueness of solutions to \eqref{Eq:Main1} allows to conclude. 
\end{proof}

\begin{proof}[Proof of Theorem \ref{Th:Symmetry1}]
We just prove the result for \eqref{Pv:2} as the proof for \eqref{Pv:1} is similar.  Introduce the notation $B(\TT):=L^2(0,T;L^2(\T))$. Let $c_{\mathrm{min}}$ be an optimiser for \eqref{Pv:2}. Up to a translation, we can assume that 
\[ \int_{\frac{T}2}^T c_{\mathrm{min}}=\int_{\frac{T}2}^T \bar c.\] Recall (see \eqref{Eq:Adm2}) that $\bar c^\#=\bar c$. From Proposition \ref{Pr:Variational} we know that 
\begin{multline*} \lambda(c_{\mathrm{min}})=\min_{\alpha\in A(\TT)\,, \iint_\TT \alpha^2=1}\left(\iint_\TT|\n_x\alpha|^2-\iint_\TT c_{\mathrm{min}}V\alpha^2\right.\\-\left.\min_{\beta\in B(\TT)}\iint_\TT\left(-\beta\partial_t\alpha^2+|\n\beta|^2\alpha^2\right)\right).\end{multline*}
Let $(\alpha^*,\beta^*)$ be the unique solution of this saddle-point problem and let $(I_1\,, I_2)$ be defined, respectively, as 
\begin{multline*}I_1:=\frac{\iint_{[0;\frac{T}2]\times \T}|\n_x\alpha^*|^2-\iint_{[0;\frac{T}2]\times \T} c_{\mathrm{min}}V(\alpha^*)^2-\min_{\beta}\iint_{[0;\frac{T}2]\times \T}\left(\partial_t\beta(\alpha^*)^2+|\n\beta|^2(\alpha^*)^2\right)}{\iint_{[0;\frac{T}2]\times \T}(\alpha^*)^2}\\
I_2:=\frac{\iint_{\left[\frac{T}2;T\right]\times \T}|\n_x\alpha^*|^2-\iint_{\left[\frac{T}2;T\right]\times \T} c_{\mathrm{min}}V(\alpha^*)^2-\min_{\beta}\iint_{\left[\frac{T}2;T\right]\times \T}\left(\partial_t\beta(\alpha^*)^2+|\n\beta|^2(\alpha^*)^2\right)}{\iint_{\left[\frac{T}2;T\right]\times \T}(\alpha^*)^2}.
 \end{multline*}It follows from \eqref{Eq:Variational} and the inequality $\frac{a+b}{c+d}\geq \min(\frac{a}c,\frac{b}d)$ that 
 \[ \lambda(c_{\mathrm{min}})\geq \min(I_1,I_2).\] Assume without loss of generality that $I_2\leq I_1$. In that case, define \[ \tilde c_{\mathrm{min}}(t)=\begin{cases}c_{\mathrm{min}}(t)\text{ if }t\in\left[\frac{T}2;T\right]\,, 
 \\ c_{\mathrm{min}}\left({T}-t\right)\text{ else}\end{cases}\] and \[\tilde\alpha^*(t,\cdot):=\begin{cases}\alpha^*(t,\cdot)\text{ if }t\in\left[\frac{T}2;T\right]\,,
 \\ \alpha^*(T-t,\cdot)\text{ else}\,,\end{cases} \tilde\beta^*(t,\cdot):=\begin{cases}\beta^*(t,\cdot)\text{ if }t\in \left[\frac{T}2;T\right]\,, 
 \\ -\beta^*(T-t,\cdot)\text{ else}.\end{cases}\]
 By construction, we have 
 \[ \min_{\beta\in B(\TT)}\iint_\TT -\beta\partial_t(\tilde\alpha^*)^2+|\n \beta|^2(\tilde\alpha^*)^2=\iint_\TT -\tilde\beta^*\partial_t(\tilde\alpha^*)^2+|\n \tilde\beta^*|^2(\tilde\alpha^*)^2\] and 
 \[ I_2=\frac{\iint_\TT |\n\tilde\alpha^*|^2-\iint_\TT \tilde c_{\mathrm{min}}V(\tilde \alpha^*)^2-\iint_\TT -\tilde\beta^*\partial_t(\tilde\alpha^*)^2+|\n \tilde\beta^*|^2(\tilde\alpha^*)^2}{\iint_\TT(\tilde\alpha^*)^2}.\] so that 
 \[\lambda(c_{\mathrm{min}})\geq I_2\geq \lambda(\tilde c_{\mathrm{min}}).\] This establishes that 
 \[ \lambda(c_{\mathrm{min}})\geq \lambda(\tilde c_{\mathrm{min}}).\] Two points remain to be checked: first, that $\tilde c_{\mathrm{min}}\in \mathcal C$ and, second, that if $\lambda(c_{\mathrm{min}})=\lambda(\tilde c_{\mathrm{min}})$, then $c_{\mathrm{min}}=\tilde c_{\mathrm{min}}$. Regarding the first point, it is a direct consequence of the construction of $\tilde c_{\mathrm{min}}$. Indeed, let $s\in(0,T/2)$ and let $E_s^*$ be optimal for the problem
 \[ \max_{E\subset \left[\frac{T}2,T\right]\,, |E|=s}\int_{\frac{T}2}^T\mathds 1_E c_{\mathrm{min}}.\] Let $\tilde E_s$ be the set obtained by symmetrising $E_s^*$ with respect to $t=\frac{T}2$; clearly $\tilde E_s$ is a solution of 
 \[ \max_{E\subset[0,T]\,, |E|=s}\int_0^T \mathds 1_E \tilde c_{\mathrm{min}}.\]
 However, by the Hardy-Littlewood inequality for the decreasing rearrangement, we know that 
 \[ \int_{\frac{T}2}^T \mathds 1_E c_{\mathrm{min}}\leq \int_{\frac{T}2}^T \mathds 1_{E^\#} (c_{\mathrm{min}})^\#\leq \int_{\frac{T}2}^T\mathds 1_{E^\#} \bar c\] by definition of $\mathcal C$. We thus conclude that 
 \[  \max_{E\subset[0,T]\,, |E|=s}\int_0^T \mathds 1_E \tilde c_{\mathrm{min}}\leq \int_{\frac{T}2}^T\mathds 1_{\tilde E_s^*}\bar c\] and, as $\int_0^T \tilde c_{\mathrm{min}}=\int_0^T \bar c$, consequently, $\tilde c_{\mathrm{min}}\preceq \bar c$. Now, assume that 
 \[ \lambda(c_{\mathrm{min}})=\lambda(\tilde c_{\mathrm{min}}).\] We deduce from \eqref{Eq:Variational} that 
 \[ \tilde\alpha^*=\alpha_{\tilde c_{\mathrm{min}}}\,, \tilde\beta^*=\beta_{c_{\mathrm{min}}}.\]
 We already observed, in the proof of Proposition \ref{Pr:Variational}, that 
 \[ \ln(\alpha_m)+\beta_m=\p_m.\] Let $\p^*:=-\ln(u_{c_{\mathrm{min}}})$, so that, defining 
 \[ \tilde \p^*(t,\cdot):=\begin{cases}
 \p^*(t,\cdot)\text{ if }t\geq \frac{T}2\,, 
 \\ \p^*(T-t,\cdot)\text{ else,}
 \end{cases}\tilde \psi^*(t,\cdot):=\begin{cases}
 \psi^*(t,\cdot)\text{ if }t\geq \frac{T}2\,, 
 \\ \psi^*(T-t,\cdot)\text{ else,}
 \end{cases}\]it follows that 
 \[ \ln(\tilde\alpha^*)+\beta^*=\p_{\tilde c_{\mathrm{min}}}=\tilde\p^*\text{ and, similarly, }\psi_{\tilde c_{\mathrm{min}}}=\tilde\psi^* .\]
 As $\tilde c_{\mathrm{min}}$ is symmetric, we obtain 
 \[ \p_{\tilde c_{\mathrm{min}}}(t,\cdot)=\psi_{\tilde c_{\mathrm{min}}}(T-t,\cdot),\] which in turn implies that 
 \[\tilde \p^*(t,\cdot)=\tilde \psi^*(T-t,\cdot).\] Turning back to the definitions of $\tilde\p^*\,, \tilde\psi^*$, we finally conclude that, for any $t\in\left[\frac{T}2;T\right]$, 
 \[ \p^*(t,\cdot)=\psi^*(T-t,\cdot)\,, \p^*(T-t,\cdot)=\psi^*(t,\cdot),\] from which we deduce, from \eqref{Eq:HopfCole}, that $ c_{\mathrm{min}}=\tilde c_{\mathrm{min}}$.
 \end{proof}

\section{Proof of Theorem \ref{Th:Gaussian}}
We being with a study of the eigenpair given by \eqref{Eq:EigenGaussian}.
\subsection{Study of \eqref{Eq:EigenGaussian}}
\subsubsection{Preliminary on periodic solutions of non-linear ODEs}
In this subsection, given $f\in L^\infty_{\mathrm{per}}((0,T))$, we are interested in the existence and uniqueness of a solution $\gamma_f$ to 
\begin{equation}\label{Eq:HJBODE}
\begin{cases}
\gamma_f'+\gamma_f^2=f&\text{ in }(0,T)\,, 
\\ \gamma_f(0)=\gamma_f(T).\end{cases}\end{equation} 
We introduce the following auxiliary eigenvalue problem: 
 for any $\alpha\in \R$, let $\theta_f(\alpha)$ be the first eigenvalue associated with the problem
\begin{equation}\label{Eq:AuxiliaryEigenvalue}\begin{cases}
-y_{\alpha,f}''- 2\alpha y_{\alpha,f}'- \alpha^2y_{\alpha,f}=-fy_{\alpha,f}+\theta_f(\alpha)y_{\alpha,f}&\text{ in }[0,T]\,, 
\\ y_{\alpha,f}(0)=y_{\alpha,f}(T)\,, 
\\ y_{\alpha,f}\geq 0.\end{cases}\end{equation}In this section, for the sake of notational compactness, we use the notation 
\[\langle f\rangle_{[0,T]}=\frac1T\int_0^T f.\]
The main result here is the following:
\begin{theorem}\label{Pr:Uniqueness}For any $f\geq 0\,, f\not\equiv 0$, there exist exactly two solutions $\gamma_f^1\,, \gamma_f^2$ to \eqref{Eq:HJBODE}. Up to relabelling, we have 
\[ \gamma_f^1> 0\,, \gamma_f^2<0\text{ and }\langle \gamma_f^1\rangle_{[0,T]}=-\langle \gamma_f^2\rangle_{[0,T]}.\] Finally, with the notations of \eqref{Eq:AuxiliaryEigenvalue}, 
\[ \theta_f\left(\langle \gamma_f^1\rangle_{[0,T]}\right)=\theta_f\left(\langle\gamma_f^2\rangle_{[0,T]}\right)=0.\]\end{theorem}

The proof of this theorem is lengthy and relies on several intermediate results. We begin with the following proposition. 
 \begin{proposition}\label{Pr:CNS}
For any $f\in L^\infty_{\mathrm{per}}((0,T))$  there exists a solution $\gamma_f$ to \eqref{Eq:HJBODE} if, and only if, there exists $\alpha_f>0$ such that 
\[\theta_f(\alpha_f)=0.\] Moreover, in this case, one has  $\alpha_f=\langle \gamma_f\rangle_{[0,T]}$. \end{proposition}

\begin{proof}[Proof of Proposition \ref{Pr:CNS}]

 First of all, assume that there exists a solution $\gamma_f$ to \eqref{Eq:HJBODE} and define 
$ y_f:t\mapsto e^{ \int_0^t (\gamma_f-\langle \gamma_f\rangle)}.$ $y_f$ is periodic, and direct computations show that
$-y_f''-2\langle \gamma_f\rangle_{[0,T]} y_f'- \langle\gamma_f\rangle_{[0,T]}^2 y_f=-  fy_f. $ As $y_f>0$ this rewrites
$\theta_f(\langle \gamma_f\rangle_{[0,T]})=0.$ 
Conversely, assume that for a given $\alpha$ we have $\theta_f(\alpha)=0$. In that case, define 
$ \gamma_f:=\alpha+{y_f'}/{y_f}.$
Direct computations show
$\gamma_f'+\gamma_f^2=f.$
\end{proof}
We continue with a basic estimate on \eqref{Eq:AuxiliaryEigenvalue}\begin{lemma}\label{Le:Estimate}
For any $f\in L^\infty_{\mathrm{per}}((0,T))\,, f \geq 0\,, f\not\equiv0$, there holds 
\[ \theta_f(\alpha)\underset{\alpha\to \pm\infty}\rightarrow -\infty\,, \theta_f(0)>0.\] Furthermore, there exists a unique $\alpha>0$ such that $\theta_f(\alpha)=0$ and a unique $\alpha'<0$ such that $\theta_f(\alpha')=0$.\end{lemma}

\begin{proof}[Proof of Lemma \ref{Le:Estimate}]
For $\alpha=0$ the eigenvalue $\theta_f(0)$ has the variational formulation 
\[ \theta_f(0)=\min_{y\in W^{1,2}_{\mathrm{per}}((0,T))\,, \int_0^T y^2=1}\left(\int_0^T (y')^2+\int_0^T y^2 f\right).\] $W^{1,2}_{\mathrm{per}}((0,T))$ stands for the set of $T$-periodic $W^{1,2}$ functions.
 As $f\geq 0\,, \not\equiv 0$ we deduce that $\theta_f(0)>0$.
Furthermore, a direct computation proves that 
$ \frac{\partial \theta_f(\alpha)}{\partial \alpha}=-2\alpha.$ The conclusion  follows.\end{proof}
We now establish the following lemma:
\begin{lemma}\label{Le:AverageSign} Let $f\geq 0\,, f\not\equiv 0$ and assume $\gamma_1\,, \gamma_2$ are two distinct solutions of \eqref{Eq:HJBODE}. Then there holds (up to relabelling)
$0<\langle \gamma_1\rangle_{[0,T]}=-\langle\gamma_2\rangle_{[0,T]}.$
\end{lemma}
\begin{proof}[Proof of Lemma \ref{Le:AverageSign}]
Define $\tilde \gamma:=\gamma_1-\gamma_2$. Then $\tilde\gamma$ solves
$ \tilde\gamma'+\tilde\gamma(\gamma_1+\gamma_2)=0.$ Upon integration, this implies 
$ \tilde \gamma(T)=\tilde\gamma(0)e^{-T\langle\gamma_2\rangle_{[0,T]}-T\langle \gamma_1\rangle_{[0,T]}},$ whence, as $\tilde\gamma\neq 0$, 
$ \langle\gamma_1\rangle_{[0,T]}=-\langle\gamma_2\rangle_{[0,T]}.$ As $\theta_f(0)>0$, $\langle\gamma_1\rangle_{[0,T]} \neq 0$ and  the conclusion follows.

\end{proof}
\begin{corollary}\label{Co:Uniqueness}
There exists at most one non-negative and one non-positive solution of \eqref{Eq:HJBODE}.
\end{corollary}
\begin{lemma}\label{Le:Sign}
Let $f\geq 0\,, \not\equiv 0$ and $\gamma$ be a solution of \eqref{Eq:HJBODE} with $\langle \gamma\rangle_{[0,T]}>0$.Then $\gamma>0$ in $(0,T)$. Similarly, if $\langle\gamma\rangle_{[0,T]}<0$, then $\gamma< 0$ in $(0,T)$.
\end{lemma}
\begin{proof}[Proof of Lemma \ref{Le:Sign}]
Assume that $\gamma$ solves \eqref{Eq:HJBODE} and satisfies $\langle \gamma\rangle_{[0,T]}>0$. Up to a translation, we can assume that $\gamma(0)>0$. Since $f\geq 0$, we deduce that 
$\gamma'+\gamma^2\geq 0$ whence, from the comparison principle for ODEs, we deduce
$\gamma(t)\geq \frac1{t+\frac1{\gamma(0)}}$ and the conclusion follows. If, on the other hand, $\langle \gamma\rangle_{[0,T]}<0$, it suffices to observe that $\tilde \gamma:=-\gamma(T-t)$ solves 
$\tilde\gamma'+\tilde\gamma^2=f(T-\cdot)\,, \langle \tilde\gamma\rangle_{[0,T]}>0.$ We are reduced to the previous case.
\end{proof}

\begin{proof}[Proof of Theorem \ref{Pr:Uniqueness}]
From Lemma \ref{Le:Estimate}, let $\alpha_2<0<\alpha_1$ be such that $\theta_f(\alpha_i)=0\,, i=1,2$, and let $\gamma_f^i\ (i=1,2)$ be two solutions of \eqref{Eq:HJBODE} with $\langle \gamma_f^i\rangle_{[0,T]}=\alpha_i$. From Lemma \ref{Le:AverageSign} $\gamma_f^i$ is the unique solution of \eqref{Eq:HJBODE} with $\langle\gamma_f^1\rangle_{[0,T]}>0$ and $\langle \gamma_f^2\rangle_{[0,T]}<0$. From Lemma \ref{Le:Sign}, we have $\gamma_f^1>0\,, \gamma_f^2<0$ in $(0,T)$, which concludes the proof. \end{proof}

\subsubsection{Analysis of \eqref{Eq:EigenGaussian}}\label{sec:w0}
We begin with the following proposition:
\begin{proposition}\label{Pr:UniqueConfining}
If $c\geq 0$, there exists a unique ($\overline\lambda(c),w_0)$ of \eqref{Eq:EigenGaussian}.
\end{proposition}
After building a particular solution to \eqref{Eq:EigenGaussian}, the difficult part in this proposition is the uniqueness, which ultimately relies on the Donsker-Varadhan formula. Observe that Lions implicitly uses this approach to derive uniqueness results for spectral problems in \cite[Lecture 7]{Lions}.
We begin with the construction of a particular solution to \eqref{Eq:EigenGaussian}.
\begin{lemma}\label{Le:Construction} Let $P$ be an orthogonal matrix such that 
\[ A=P \mathrm{diag}(\mu_1,\dots,\mu_d)P^T,\] where $0<\mu_1\leq \dots\leq \mu_d$ are the eigenvalues of $A$. An explicit solution of \eqref{Eq:EigenGaussian} is given by
\[ \overline\lambda(c)=\frac1T\int_0^T \mathrm{Tr}(B_c)\,, \quad w_0(t,x)=e^{\beta(t)-\frac{1}2\langle x,B_c(t)x\rangle}\] with 
\[ B_c(t)=P\mathrm{diag}(\xi_1(t),\dots,\xi_d(t))P^T\text{ where, for any $i$, }
\forall i\in \{1,\dots,d\}\, \begin{cases}
\frac{\xi_i'(t)}2+\xi_i(t)^2=\frac{\mu_i}2 c(t)\,, 
\\ \xi_i(0)=\xi_i(T)\,, 
\\ \xi_i>0
\end{cases}\]
and 
\[ \beta(t)=\int_0^t(\overline\lambda(c)-\mathrm{Tr}(B_c(t)))dt.\]
\end{lemma}
\begin{proof}[Proof of Lemma \ref{Le:Construction}]
We first look for a solution $(\overline\lambda(c)\,, w_0)$ of \eqref{Eq:EigenGaussian} with $w_0$ writing 
$w_0=e^{\beta(t)-\frac{1}2\langle x,B_c(t)x\rangle}$ for $\gamma$ a function to be determined and $B_c$ a symmetric positive definite matrix-valued function. Direct computations show that, for any such $w_0$ and for any $\overline\lambda$,  $(\overline\lambda,w_0)$ solves \eqref{Eq:EigenGaussian} if, and only if, 
\begin{equation}\label{Eq:LimitHopfCole}
\begin{cases}
\beta'(t)-\frac{1}2\langle x,B_c'(t)x\rangle+\mathrm{Tr}(B_c(t))-\langle x,B_c(t)^2x\rangle=\overline\lambda-\frac{c(t)}2\langle x,Ax\rangle\,, 
\\ \beta\,, B_c\text{ are  $T$-periodic}\,, B_c(t)\in S_d^{++}(\R).
\end{cases}
\end{equation}
In particular, if we can find a solution $B_c$ of the differential equation
\begin{equation}\label{Eq:Ricatti} \frac{B_c'(t)}2+B_c(t)^2=\frac{c(t)}2 A\,, \quad B_c(T)=B_c(0)\,, \quad B_c(\cdot)\in S_d^{++}(\R)\end{equation} then, defining \[\beta:t\mapsto \int_0^t  (\langle \mathrm{Tr}(B_c)\rangle_{[0,T]}-\mathrm{Tr}(B_c))\,, \quad \overline\lambda:=\fint_0^T\mathrm{Tr}(B_c)\] it follows that $( \overline \lambda,w_0)$ is a solution of \eqref{Eq:EigenGaussian}. Let us prove that a solution to \eqref{Eq:Ricatti} exists. We look for $B_c$ in the form 
$B_c=P\mathrm{diag}(\xi_1(t),\dots,\xi_d(t))P^T$ so that \eqref{Eq:Ricatti} rewrites
\begin{equation}\label{Eq:HJBODERicatti}
\forall i\in \{1,\dots,d\}\, \begin{cases}
\frac{\xi_i'(t)}2+\xi_i(t)^2=\frac{\mu_i}2 c(t)\,, 
\\ \xi_i(0)=\xi_i(T)\,, 
\\ \xi_i>0.
\end{cases}\end{equation} From Theorem \ref{Pr:Uniqueness}, $\xi_i$ exists, and is uniquely defined, whence the conclusion.
\end{proof}
\begin{proof}[Proof of Proposition \ref{Pr:UniqueConfining}]
We let $(\overline\lambda_0,w_0)$ be the solution of \eqref{Eq:EigenGaussian} provided by Lemma \ref{Le:Construction}. First of all observe that, if $(\overline \lambda,w)$ is a solution of \eqref{Eq:EigenGaussian} then, since $\Vert w_0\Vert_{L^\infty}\leq Ae^{-A'\Vert x\Vert^2}$ by construction, it follows from integration by parts that $\overline\lambda=\overline\lambda_0$.
Second, let $\overline\eta_0$ be defined as a solution of 
\begin{equation}\begin{cases}
-\partial_t\eta_0-\Delta\eta_0-\nabla\cdot\left(\frac{\n_x w_0}{w_0}\eta_0\right)=0&\text{ in }(0,T)\times \R^d\,, 
\\ \forall t\in (0,T)\,, \int_{\R^d}\eta_0(t,\cdot)=1\,, 
\\ \eta_0\geq 0\,, \eta_0(T,\cdot)=\eta_0(0,\cdot).\end{cases}\end{equation}That such an $\eta_0$ exists follows from the same construction as in Lemma \ref{Le:Construction}: define $\overline z_0$ as the solution of 
\begin{equation}
\begin{cases}
-\partial_t z_0-\Delta z_0=\overline\lambda_0z_0-\frac12c(t)\langle x,Ax\rangle z_0&\text{ in }[0,T]\times \R^d\,, 
\\ z_0(T,\cdot)=z_0(0,\cdot)\,, 
\\ z_0(0,0)=1=\Vert z_0\Vert_{L^\infty}\,,
\\ z_0\geq 0,
\end{cases}\end{equation} which can be done exactly in the same way as in Lemma \ref{Le:Construction}, and set $\eta_0:=w_0z_0$. Observe that $\Vert \eta_0(t,x)\Vert\lesssim Ce^{-A\Vert x\Vert^2}$ for some constant $A$. Now, consider $\p_0:=-\ln(w_0)$ and define the functional 
\[ \mathscr E:\p\mapsto \iint_{(0,T)\times \R^d}\left(\partial_t\p-\Delta \p+\Vert\n\p\Vert^2-\frac{c}2\langle x,Ax\rangle\right)d\eta_0.\] Since $\mathscr E$ is strictly convex in $\p$ and since $\p_0$ is one of its critical points we deduce that
\[ \forall \p\,, \mathscr E(\p)\geq \mathscr E(\p_0)=\overline\lambda_0.\] Furthermore, observe that, for any $\p$, if we set $\overline\p:=\p-\p_0$,
\begin{align*}
\mathscr E(\p)-\mathscr E(\p_0)&=\iint_{(0,T)\times \R^d}\Vert \n\overline\p\Vert^2d\eta_0.
\end{align*}
In particular, let $(\overline\lambda_0,w)$ be a solution of \eqref{Eq:EigenGaussian}. Then, setting $\p=-\ln(w)$ we have 
$ 0=\mathscr E(\p)-\mathscr E(\p_0)=\iint_{(0,T)\times \R^d}\Vert \n\overline\p\Vert^2d\eta_0.$ We deduce that there exists a constant $a$ such that 
$\p=\p_0+a.$ Since $\p(0,0)=\p_0(0,0)=1$ we have $a=0$ and the conclusion follows.
\end{proof}

\subsection{Proof of Theorem \ref{Th:Gaussian}}
\begin{proof}[Proof of Theorem \ref{Th:Gaussian}]
We know from Lemma \ref{Le:Construction} that 
\[ \overline\lambda(c)=\sum_{i=1}^d\langle \xi_{i,c}\rangle_{[0,T]}\] where, for every $i\in\{1,\dots,d\}$, 
\[\begin{cases}
\frac{\xi_{i,c}'(t)}2+\xi_{i,c}(t)^2=\frac{\mu_i}2 c(t)\,, 
\\ \xi_{i,c}(0)=\xi_{i,c}(T)\,, 
\\ \xi_{i,c}>0.
\end{cases}\]Recall that the $\mu_i$ are the eigenvalues of $A=-\n^2 V(0)$. If we can prove that 
\begin{multline}\label{Eq:MonRea}
\forall i\in \{1,\dots,d\}\,,\\\langle \xi_{i,c}\rangle_{[0,T]} \geq \langle \xi_{i, c^\#}\rangle_{[0,T]}\text{ with equality if, and only if, up to translation: }c= c^\#\end{multline} then the conclusion follows. We thus prove \eqref{Eq:MonRea}. With the notations of Theorem \ref{Pr:Uniqueness} we know that 
$\theta_{\mu_i c/2}(\langle\xi_{i,c}\rangle_{[0,T]})=0$ and that, for any given $c$, the map $\theta_{\mu_ic/2}$ is strictly decreasing on $(0;+\infty)$. Furthermore we know from \cite[Theorem 1.2]{zbMATH05834183} that, for any $f$ such that, for any $\tau$, $f(\cdot+\tau)\not\equiv f^\#$, 
for any $ \alpha>0\,, \theta_f(\alpha)>\theta_{f^\#}(\alpha).$ Consequently,  
$ \theta_{\mu_i c^\#/2}(\langle \xi_{i,c}\rangle_{[0,T]})<0$ and thus \eqref{Eq:MonRea} holds. 
\end{proof}

\section{Proof of Theorem \ref{Th:LargeDiffusivity}}

\subsection{Asymptotic behaviour of $(\lambda(c,\mu),u_{c,\mu})$ and limit problem}
For any $m\in L^\infty(\TT)$, we let $\bar\p_m$ be the unique solution of 
\begin{equation}\label{Eq:Acampora}\begin{cases}
\partial_t\bar \p_m-\Delta \bar\p_m=m-\fiint_\TT m&\text{ in } \TT\,, 
\\\bar \p_m(0,\cdot)=\bar\p_m(T,\cdot)\,, 
\\ \fiint_\TT\bar \p_m=0,\end{cases}\end{equation} and we define 
\[ \Lambda(m):=\fiint_\TT \left|\n\bar \p_m\right|^2.\]
The main result of this section is the following:
\begin{proposition}\label{Pr:DAEigenvalue}
There exists a constant $M$ that only depends on $\Vert m\Vert_{L^\infty(\TT)}\,, T$ and $d$ such that
\begin{equation}\label{Eq:DAEigenvalue}
\left| \lambda(m,\mu)+\fiint_\TT m-\frac{\Lambda(m)}{\mu}\right|\leq M\frac{\Lambda(m)}{\mu^{\frac32}}.\end{equation}
\end{proposition}
\begin{remark}[On the exponent $3/2$]
We could  improve the exponent to 2 at the cost of greater technicality. Nevertheless, this suboptimal exponent is sufficient for our purposes.\end{remark}
\begin{proof}[Proof of Proposition \ref{Pr:DAEigenvalue}]
We work with a fixed $m$ so that we write $u_\mu$ for the solution of 
\begin{equation}\label{Eq:umu} \mu\left(\partial_t u_\mu-\Delta u_\mu\right)=\lambda(m,\mu)u_\mu+mu_\mu\end{equation} endowed with periodic boundary conditions and normalised with 
$ \fiint_\TT u_\mu^2=1.$ 
\textbf{Uniform bounds on $\lambda(m,\mu)$}
First of all, observe that, by writing 
$u_{\mu}:=e^{-\p_\mu}$, then $\p_\mu$ solves
\[ \mu\partial_t \p_\mu-\mu\Delta \p_\mu+\mu|\n\p_\mu|^2=-\lambda(m,\mu)-m,\] endowed with time-space periodic boundary conditions. Integrating this equation on $\TT$, we deduce that there holds
$ \lambda(m)=-\fiint_\TT m-\mu\fiint_\T |\n \p_\mu|^2\leq-\fiint_\TT m\leq \Vert m\Vert_{L^\infty(\TT)}$. Second, taking a point of minimum of $\p_\mu$ in $\TT$ we obtain 
$ \lambda(m)\geq -\Vert m\Vert_{L^\infty(\TT)}$. We  deduce
\begin{equation}\label{Eq:EigenvBoundLargeDiffusivity}
\forall \mu>0\,,|\lambda(m,\mu)|\leq \Vert m\Vert_{L^\infty(\TT)}.
\end{equation}
\textbf{Uniform convergence of $u_\mu$} 
\begin{lemma}\label{Le:Carazzato}There exists $M$ that only depends on $\Vert m\Vert_{L^\infty(\TT)}$ such that for all $\mu$ large enough
\begin{equation}\label{Eq:Carazzato}
\Vert u_\mu-1\Vert_{L^\infty(\TT)}\leq \frac{M}\mu.
\end{equation}\end{lemma}
\begin{proof}[Proof of Lemma \ref{Le:Carazzato}]
We first prove that there exists $M_0$ such that, for any $\mu$ large enough
\begin{equation}\label{Eq:BoundIntegral}
\sup_{t\in [0,T]}\int_\T u_\mu(t,\cdot)\leq M_0.
\end{equation}
\begin{proof}[Proof of \eqref{Eq:BoundIntegral}]
Integrate the equation on $u_\mu$ in space,  pass to the absolute value, integrate in time and use the Cauchy-Schwarz inequality and \eqref{Eq:EigenvBoundLargeDiffusivity} to obtain\[ \mu\int_0^T \left|\frac{d}{dt} \fint_\T u_\mu\right|\leq 2\Vert m\Vert_{L^\infty(\TT)}.\] Thus $I:t\mapsto \fint_\T \mu u_\mu$ is uniformly bounded in $BV(0,T)$, whence  
\begin{equation}\label{Eq:Bd1}
\sup_{t\in [0,T]}\left| \fint_\T u_\mu(t,\cdot)-\fiint_\TT u_\mu \right|\leq \frac{M_1}\mu.
\end{equation}
Using the Cauchy-Schwarz inequality again, we deduce that there exists $M_0$ such that \eqref{Eq:BoundIntegral} holds.
\end{proof}

 Observe that \eqref{Eq:BoundIntegral} implies, once again integrating \eqref{Eq:umu} in $\T$, that there exists $M>0$ such that:
\begin{equation}\label{Eq:Carazatto2'} \sup_{t\in [0,T]} \left|\frac{d}{dt}\fint_\T u_\mu\right|\leq \frac{M}\mu.\end{equation} Let us now prove
\begin{equation}\label{Eq:Lunardi}\text{for any $p\in [1;\infty)$ there holds
 $\sup_{t\in [0,T]}\mu\Vert \n_xu_\mu\Vert_{L^p(\T)}\leq M_p<\infty$.}\end{equation}

 Consider the function $\hat u_\mu=u_\mu-\fint_\T u_\mu$. From \eqref{Eq:BoundIntegral} and  \eqref{Eq:Carazatto2'}, there holds
 \[ \partial_t(\mu\hat u_\mu)-\Delta (\mu\hat u_\mu)-m\hat u_\mu-\lambda \hat u_\mu\in L^\infty(\TT).\] From standard parabolic regularity and a bootstrap argument, \eqref{Eq:Lunardi} holds.

In particular, choosing $p$ large enough, we conclude  from the Sobolev embedding $W^{1,p}(\T)\hookrightarrow \mathscr C^0(\T)$ that
\begin{equation}\label{Eq} \sup_{t\in[0,T]}\mu\left\Vert u_\mu(t,\cdot)-\fint_\T u_\mu(t,\cdot)\right\Vert_{L^\infty(\T)}\leq M.\end{equation} From \eqref{Eq:BoundIntegral} and up to modifying $M$ we derive
\begin{equation}\label{Eq:BoundIntegral2}
\Vert u_\mu\Vert_{L^\infty(\TT)}\leq M.\end{equation}
Finally, let us prove that 
\begin{equation}\label{Eq:FinalIntegral}
\sup_{t\in[0,T]} \left|\fint_\T u_\mu(t,\cdot)-1\right|\leq \frac{M}\mu.
\end{equation}
As, by parabolic regularity, 
\[ \fiint_\TT (\partial_t u_\mu)^2\leq \frac{M}{\mu^2}\]the estimate \eqref{Eq:BoundIntegral2} implies that the map
$I_2:t\mapsto \fint_\T u_\mu^2$ satisfies 
\[ \int_0^T (I_2'(t))^2dt\leq \frac{M}{\mu^2}.\] Since
$ \fiint_\TT u_\mu^2=1$  we deduce from the Sobolev embedding $W^{1,2}(0,T)\hookrightarrow \mathscr C^0(0,T)$  that 
\[ \sup_{t\in [0,T]}\left|\fint_\T u_\mu^2-1\right|\leq \frac{M}\mu.\] 
However,  by the Poincar\'e inequality \begin{equation}\label{Eq:Jensen}
\left| \fint_\T u_\mu^2-\left(\fint_\T u_\mu\right)^2\right|\leq \fint_\T |\n_xu_\mu|^2\leq \frac{M}{\mu^2}.
\end{equation}  This implies \eqref{Eq:FinalIntegral}.
Going back to \eqref{Eq} we finally obtain
\[ \sup_{t\in[0,T]} \left\Vert u_\mu(t,\cdot)-1\right\Vert_{L^\infty(\T)}\leq \frac{M}\mu.\]\end{proof}
\textbf{Uniform convergence of the invariant measure associated with the system}

Introduce, for any $\mu>0$, the solution $\eta_\mu$ of 
\begin{equation}\label{Eq:InvariantMeasureLD}\begin{cases}
\mu\left(-\partial_t\eta_\mu-\Delta\eta_\mu-2\nabla\cdot\left(\eta_\mu\frac{\n_xu_\mu}{u_\mu}\right)\right)=0&\text{ in }\TT\,, 
\\ \eta_\mu\geq 0\,, \fiint_\TT \eta_\mu=1.\end{cases}\end{equation}
Let  $v_\mu$ be the solution of the adjoint problem
\begin{equation}\begin{cases}
\mu\left(-\partial_t v_\mu-\Delta v_\mu\right)=\lambda(m,\mu)v_\mu+mv_{\mu}&\text{ in }\TT\,, 
\\ v_{\mu}\geq 0\,, \fiint_\TT v_\mu u_\mu=1.\end{cases}\end{equation}
Then direct computations show that \[\eta_\mu=u_\mu v_\mu.\] Adapting \emph{mutatis mutandis} the proof of \eqref{Eq:Carazzato}, there holds
\[
\Vert v_\mu-1\Vert_{L^\infty(\TT)}\leq \frac{M}\mu.\] Consequently we obtain 
\begin{equation}\label{Eq:Carazatto2}
\Vert \eta_\mu-1\Vert_{L^\infty(\TT)}\leq \frac{M}\mu.
\end{equation}

\textbf{Using the Donskers-Varadhan formula}
Use again the Hopf-Cole transform by writing $u_\mu=e^{-\p_\mu}$, where $\p_\mu$ solves
\[ \mu\left(\partial_t\p_\mu-\Delta\p_\mu+|\n\p_\mu|^2\right)=-\lambda(m,\mu)-m.\] Then it appears that $\lambda(m,\mu)$  admits the variational formulation 
\begin{multline*} -\lambda(m,\mu)=\\\max_{\eta\in \mathbb P(\TT)}\min_{\phi\in L^2(0,T;W^{1,2}(\O))}\fiint_\TT \left(\mu\left(\partial_t\phi-\Delta\phi+|\n\phi|^2\right)+m\right)d\eta\end{multline*}
where $\mathbb P(\TT)$ is the set of probability measures on $\TT$.
Indeed, we can check immediately that $(\p_\mu,\eta_\mu)$ is a solution of this saddle point problem and the uniqueness of a solution follows from the same arguments as in the proof of Theorem \ref{Th:Symmetry1}. We also refer to \cite{arXiv:2409.08740}. In particular, as $(\p_\mu,\eta_\mu)$ is a solution of this saddle point problem, we deduce, using $(\overline\eta\equiv {\mathrm{Vol}(\TT)},\p_\mu)$ and $(\eta_\mu,\frac1{\mu}\overline\p_m)$ (recall that $\overline \p_m$ is defined by \eqref{Eq:Acampora}) as test couples, that we have the estimate
\[\mu\fiint_\TT |\n\p_\mu|^2\leq-\lambda(m,\mu)-\fiint_\TT m \leq \frac1\mu\fiint_\TT |\n\overline\p_m|^2\eta_\mu.
\]Using \eqref{Eq:Carazatto2} this implies that
\begin{equation}\label{Eq:Cito} -\lambda(m,\mu)-\fiint_\TT m -\frac{\Lambda(m)}\mu\leq \frac1{\mu^2}\fiint_\TT|\n\overline \p_m|^2\leq M\frac{\Lambda(m)}{\mu^2}.\end{equation}  To conclude the proof, we thus need to show that 
\begin{equation}\label{Eq:Barbato}\left|\mu\fiint_\TT|\n \p_\mu|^2-\frac{\Lambda(m)}\mu\right|\leq M\frac{\Lambda(m)}{\mu^{\frac32}}.
\end{equation}
Observe that, introducing $\overline\p_\mu:=\mu\p_\mu$, \eqref{Eq:Barbato} is equivalent to proving that 
\begin{equation}\label{Eq:Barbato2} \left|\fiint_\TT |\n\overline\p_\mu|^2-\fiint_\TT|\n\overline\p_m|^2\right|\leq M\frac{\Lambda(m)}{\sqrt{\mu}}.\end{equation} A first observation is that $\overline\p_\mu=-\mu\ln(u_\mu)$ so that \eqref{Eq:Carazzato} implies
\begin{equation}\label{Eq:Masiello} \Vert \overline\p_\mu\Vert_{L^\infty(\TT)}\leq M.\end{equation}
Define $z_\mu:=\overline\p_\mu-\overline\p_m$, which solves
\[\partial_t z_\mu-\Delta z_\mu+\frac1\mu|\n \overline\p_\mu|^2=-\lambda(m,\mu)-\fiint_\TT m.\] Up to replacing $z_\mu$ with $z_\mu+A$ where $A$ is independent of $\mu$  we can without loss of generality assume that 
\[ 0\leq z_\mu\leq M\] for some constant $M$ independent of $\mu$ thanks to \eqref{Eq:Masiello}. Multiplying  by $z_\mu$ and integrating by parts gives
\begin{multline*}\fiint_\T |\n z_\mu|^2\leq \fiint_\TT |\n z_\mu|^2+\frac1\mu\fiint_\TT z_\mu|\n\overline\p_\mu|^2\\\leq \left(-\lambda(m,\mu)-\fiint_\TT m\right)\left(\int_0^T \left\Vert z_\mu-\fint_\T z_\mu\right\Vert_{L^2(\T)}+\Vert z_\mu\Vert_{L^\infty(\TT)}\right).
\end{multline*}
The Jensen inequality, the Poincar\'e inequality and  \eqref{Eq:Cito} yield that, by setting $X_\mu:=\left(\fiint_\TT |\n z_\mu|^2\right)^{\frac12}$, we have 
\[ X_\mu^2\leq \frac{\Lambda(m)}{\mu}X_\mu+\frac{M\Lambda(m)}\mu.\] We deduce that for $\mu$ large enough there holds (for a constant still denoted by $M$)
\[ \left|X_\mu\right|\leq M\sqrt{\frac{\Lambda(m)}\mu}.\] 
Thus
\begin{multline}\label{Eq:Barbato4}
\left(\fiint_\TT |\n \overline\p_\mu|^2\right)^{\frac12}\leq  M\sqrt{\frac{\Lambda(m)}\mu}+\left(\fiint_\T|\n \overline\p_m|^2\right)^{\frac12}\\\leq M\left(\sqrt{\frac{\Lambda(m)}{\mu}}+\sqrt{\Lambda(m)}\right)\leq M\sqrt{\Lambda(m)}.
\end{multline}
Going back to \eqref{Eq:Barbato2},  
\begin{align*}
\left|\fiint_\TT |\n \overline\p_\mu|^2-\fiint_\TT |\n \overline\p_m|^2 \right|&=\left| \fiint_\TT \langle \n z_\mu,\n(\overline\p_\mu+\overline \p_m)\rangle\right|
\\&\leq \left(\fiint_\TT|\n z_\mu|^2\right)^{\frac12}\left(\fiint_\TT|\n \overline \p_\mu+\n\overline\p_m|^2\right)^{\frac12}
\\&\leq M\sqrt{\frac{\Lambda(m)}{\mu}}\sqrt{\Lambda(m)}
\\&\leq M\frac{\Lambda(m)}{\sqrt{\mu}}.
\end{align*}
This concludes the proof of \eqref{Eq:Barbato2} and, consequently, the proof of Proposition \ref{Pr:DAEigenvalue}.
\end{proof}
\subsection{Analysis of the asymptotic problem}
In this section, we go back to our initial problem, so that the potential $m$ writes as $cV$ for a fixed potential $V$. With a slight abuse of notation, we denote by $\lambda(c,\mu)\,, \Lambda(c)\,, \overline \p_c$ the quantities $\lambda(cV,\mu)\,, \Lambda(cV)\,, \overline \p_{cV}$. As we know from Proposition \ref{Pr:DAEigenvalue} that 
\[ \lambda(c,\mu)=-\fiint_\TT m-\frac{\Lambda(c)}{\mu}+\Lambda(c)\underset{\mu \to\infty}o\left(\frac1{\mu}\right)\] we first investigate the maximisation problem
\begin{equation}\label{Eq:Lecce}
\max_{c\in \mathcal C}\Lambda(c).\end{equation} Recall that 
\[ \Lambda(c)=\fiint_\TT |\n\overline\p_c|^2\text{ with }\begin{cases}\partial_t\overline\p_c-\Delta\overline\p_c=c(t)V(x)-\fint_0^T c\cdot\fint_\T V&\text{ in }\TT\,, 
\\ \overline\p_c(0,\cdot)=\overline\p_c(T,\cdot).\end{cases}\]
The main proposition is the following:
\begin{proposition}\label{Pr:LimitProblem}
Assume $V$ is not constant. For any $c\in \mathcal C$ (see Eq. \eqref{Eq:Adm2}), there holds
\[ \Lambda(c)\leq \Lambda(\overline c)\text{ with equality if, and only if, there exists $\tau_0$ such that }c(\tau_0+\cdot)=\overline c.\]
\end{proposition}
\begin{proof}[Proof of Proposition \ref{Pr:LimitProblem}]
Decompose $V$ in the basis $(\lambda_k,\psi_k)_{k\in \N}$ of Laplace eigenfunctions of $\T$ as 
\[V(x)=\sum_{k=0}^\infty a_k\psi_k.\]
The function $\overline\p_c$ writes 
\[ \overline\p_c=\sum_{k=0}^\infty a_k \alpha_{k,c}(t)\psi_k(x)\text{ where }\begin{cases}\alpha_{0,c}:t\mapsto \int_0^t(c-\langle c\rangle)\,, 
\\ \forall k\geq 1\,, \alpha_{k,c}'(t)+\lambda_k\alpha_{k,c}(t)=c(t)&\text{ in }[0,T]\,, 
\\ \alpha_{k,c}(0)=\alpha_{k,c}(T).\end{cases} \]
In particular, 
\[ \Lambda(c)=\sum_{k=1}^\infty \lambda_ka_k^2\fint_0^T \alpha_{k,c}^2(t)dt.\] We introduce, for any $k\geq 1$, 
\[ \Lambda_k(c):=\fint_0^T \alpha_{k,c}^2(t)dt\] and we investigate the optimisation problem 
\[ \max_{c\in\mathcal C}\Lambda_k(c).\] Let us show that, for any $k\in \N$, 
\begin{equation}\label{Eq:Intermediate} \Lambda_k(c)\leq \Lambda_k(\overline c)\text{ with equality if, and only if, there exists $\tau_0$ such that }c(\tau_0+\cdot)=\overline c.\end{equation}
\begin{proof}[Proof of \eqref{Eq:Intermediate}]
We rewrite $\Lambda_k$ as the Dirichlet energy of a certain elliptic problem. For any $k\in \N$, let $\beta_{k,c}$ be the solution of 
\begin{equation}\label{Eq:Betak}
\begin{cases}
-\beta_{k,c}'(t)+\lambda_k\beta_{k,c}(t)=\alpha_{k,c}(t)&\text{ in }[0,T]\,, 
\\ \beta_{k,c}(0)=\beta_{k,c}(T).\end{cases}\end{equation}
Multiplying \eqref{Eq:Betak} by $\alpha_{k,c}$ and integrating by parts we obtain 
\[ \Lambda_k(c)=\fint_0^T \alpha_{k,c}^2(t)=\fint_0^T \alpha_{k,c}(-\beta_{k,c}'+\lambda_k\beta_{k,c})=\fint_0^T \beta_{k,c}(\alpha_{k,c}'+\lambda_k\alpha_{k,c})=\fint_0^T\beta_{k,c}c.\]
Now, observe that 
\[ -\beta_{k,c}''+\lambda_k^2\beta_{k,c}=\left(\frac{d}{dt}+\lambda_k\right)\left(-\frac{d}{dt}+\lambda_k\right)\beta_{k,c}=c.\] Consequently, $\beta_{k,c}$ solves 
\[ \min_{\beta \text{ $T$-periodic, }\beta\in W^{1,2}(0,T)}\left(E(\beta,c):=\frac12\fint_0^T(\beta')^2+\frac{\lambda_k^2}2\fint_0^T\beta^2-\fint_0^T \beta c\right)\] and 
\[ \fint_0^T \beta_{k,c} c=-2E(\beta_{k,c},c).\]
Thus, \eqref{Eq:Intermediate} is equivalent to 
\begin{equation}\label{Eq:Intermediate2}
\min_{c} E(\beta_{k,c},c)=\min_{\beta,c}E(\beta,c).\end{equation} Introduce the increasing periodic rearrangement $f_\#$ of a function $f$ (it is the same notion as the one given in Definition \ref{De:Rearrangement} with the requirement that $f^\#$ be decreasing replaced with the requirement that it should be increasing), and that the Poly\'a-Szeg\"{o}  and Hardy-Littlewood inequalities are also valid for the periodic rearrangement: for any $f\in W^{1,2}(0,T)$ and any $f\,, g \in W^{1,2}(0,T)\,, f\,, g \geq 0$, there holds 
\begin{equation}\label{Eq:PS}
\fint_0^T (f')^2\geq \fint_0^T \left(\left(f_\#\right)'\right)^2\,, \fint_0^T fg\geq \fint_0^T f_\#g^\#.\end{equation} In particular, this implies, for any $(\beta,c)$
\[2 E(\beta,c)\geq 2E(\beta_\#,c^\#)\geq -\Lambda_k( c^\#).\] 

This gives 
\[ \forall c\in\mathcal C\,, \Lambda_k(c)\leq \Lambda_k( c^\#).\] Now observe that  
\[ \Lambda_k(c)=\Lambda_k(c^\#),\] implies equality in \eqref{Eq:PS} for $f=\beta_{k,c}$ and that $\beta_{k,c}^\#=\beta_{k,\overline c}$. Since the Poly\'a-Szeg\"o inequality is rigid, this implies that there exists $\tau_0>0$ such that 
\[\beta_{k,c}(\tau_0+\cdot)=\beta_{k,c}^\#.\]  Using the equation satisfied by $\beta_{k,c}$, this proves that $c(\tau_0+\cdot)= c^\#$, thereby yielding \eqref{Eq:Intermediate}.\end{proof}
\noindent\textbf{Back to the proof of Proposition \ref{Pr:LimitProblem}}
We know that 
\[ \Lambda(c)=\sum_{k=1}^\infty\lambda_ka_k^2\Lambda_k(c).\] From \eqref{Eq:Intermediate} we deduce that 
\[ \Lambda(c)\geq \Lambda( c^\#).\] If we have equality, then there holds, for any $k\in \N$, 
\[ a_k^2 \Lambda_k(c)=a_k^2\Lambda_k( c^\#).\] Since $V$ is not constant, there exists $k\geq 1$ such that $a_k\neq 0$, and, for this $k$, \eqref{Eq:Intermediate} implies that, up to a translation by $\tau_0$ (which, without loss of generality, we take to be zero), $c= c^\#$.
\end{proof}

\section{Proof of Theorem \ref{Th:SmallDiffusivity}}
Recall that we work with the following eigenvalue problem (where, up to a translation we assume that the eigenfunction reaches its maximum at $(t^*,x^*)=(0,0)$)
\[\begin{cases}
\e \partial_t u_{\e,V}-\e^2\Delta u_{\e,V}=\lambda_\e(V)u_{\e,V}+c(t)V(x)u_{\e,V}&\text{ in }\TT\,, 
\\ u_{\e,V}(0,0)=1=\Vert u_{\e,V}\Vert_{L^\infty(\TT)}\,, 
\\ u_{\e,V}(T,\cdot)=u_{\e,V}(0,\cdot).\end{cases}
\]
We assume  that $\overline c\geq 0$ so that, for any $c\in \mathcal C$, we have $c\geq 0$.
Our goal is to investigate the asymptotic behaviour of $\lambda_\e(V)$ as $\e\to 0$ and, more specifically, to prove the following proposition:
\begin{proposition}\label{Pr:BlowUp}
Let $A:=-\n^2V(0)\in S_d^{++}(\R)$ and, for any $c\in \mathcal C$, let $\overline\lambda(c)$ be the principal parabolic, periodic in time eigenvalue of 
$\partial_t-\Delta+\frac{c}2\langle \cdot,A\cdot\rangle$ in $\R^d$ (see \eqref{Eq:EigenGaussian}). Then
\[\frac{\lambda_\e(c)+\langle c\rangle V(0)}\e\underset{\e\to 0}\rightarrow \overline\lambda(c).\]
\end{proposition}
We refer to Proposition \ref{Pr:UniqueConfining} for the well-posedness of \eqref{Eq:EigenGaussian}. We single Proposition \ref{Pr:BlowUp} out as it is in itself interesting; however, as we will be working with a sequence of minimisers $\{c_\e\}_{\e\to 0}$ we also need a uniform version of this proposition (the proof of which is a minor adaptation of the proof of Proposition \ref{Pr:BlowUp}, see Remark \ref{Re:CvMin}).
 \begin{proposition}\label{Pr:CvMin}
With the same notations as in Proposition \ref{Pr:BlowUp}, let $\{c_\e\}_{\e\to 0}$ be a sequence of elements in $\mathcal C$ that weakly converges to some $c_\infty$. Then
\[\frac{\lambda_\e(c_\e)+V(0)\fint_0^T c_\e}\e\underset{\e\to 0}\rightarrow \overline\lambda(c_\infty)\] where $\overline\lambda(\cdot)$ is defined as the first eigenvalue of \eqref{Eq:EigenGaussian}.

\end{proposition}
To prove Proposition \ref{Pr:BlowUp}, we rely on a blow-up result for eigenvalues of elliptic operators \cite{zbMATH05041278}.
\begin{proposition}\label{Pr:Holcman}\cite[Proposition 2]{zbMATH05041278}
Let $\Phi\in \mathscr C^2(\T)$ be such that any local minimiser is non-degenerate and, for any $\e>0$, let $\bar\lambda_\e(\Phi)$ be the first eigenvalue of the operator $-\e^2\Delta +\Phi$ on $\T$. There exists a constant $\Lambda$ that only depends on $\Vert \Phi\Vert_{\mathscr C^2}$ and $\T$ such that there holds 
\[ \min_\T \Phi\leq \bar \lambda_\e(\Phi)\leq \min_\T \Phi+\Lambda \e.\]
\end{proposition}
\begin{proof}[Proof of Proposition \ref{Pr:BlowUp}]
Let us first show that we can assume, without loss of generality, that $V(0)=\max_\T V=0$. Using the Hopf-Cole transform $\p_{\e,V}:=-\ln(u_{\e,V})$, the function $\p_{\e,V}$ satisfies
\[\e \partial_t \p_{\e,V}-\e^2\Delta \p_{\e,V}+\e^2|\n \p_{\e,V}|^2=-\lambda_\e(V)-c(t)V(x).\] Let $\gamma_\e$ be the unique solution of 
\[ \begin{cases}
\e \gamma_\e'(t)=c(t)V(0)-\fint_0^T c\cdot V(0)\,,
\\ \gamma_\e(T)=\gamma_\e(0).
\end{cases}\]Setting $\tilde \p_{\e,V}:=\p_{\e,V}+\gamma_\e$, the function $\tilde \p_{\e,V}$ solves the Hamilton-Jacobi equation 
\[
\e\partial_t \tilde\p_{\e,V}-\e^2\Delta\tilde\p_{\e,V}+\e^2|\n \tilde\p_{\e,V}|^2=-\lambda_\e(V)-\fint_0^T c \cdot V(0)-c(t)(V(x)-V(0))\] and up to replacing $\lambda_\e(V)$ with $\lambda_\e(V)+\fint_0^T c\cdot V(0)$ we can thus assume that:
\begin{equation}\label{Eq:Morse} V(x)\leq V(0)=0\,, \n^2 V(0)\in S_d^{++}(\R).\end{equation}

We now prove the following: if \eqref{Eq:Morse} holds, there exists a constant $C$ such that
\begin{equation}\label{Eq:Holcman}
\forall \e>0\text{ small enough, } 0\leq \frac{\lambda_\e(V)}{\e}\leq C.
\end{equation}
The fact that $\lambda_\e(V)\geq 0$ follows from the maximum principle: at a point $(t^*,x^*)$ of maximum of $u_{\e,V}>0$ we obtain $\lambda_\e(V)+c(t^*)V(x^*)\geq 0$ and the conclusion follows since $V\leq 0$. In order to obtain the upper-bound on $\lambda_\e(V)/\e$ we use the fact that $V\leq 0$ so that 
$ c(t)V(x)\geq \Vert c\Vert_{L^\infty}V(x).$ As a consequence 
$ \lambda_\e(c)\leq \lambda_\e(\Vert c\Vert_{L^\infty}).$ Without loss of generality, let us assume that $\Vert c\Vert_{L^\infty}=1$. Thus if we define $\bar\lambda_\e(V)$ as the first eigenvalue of $-\e^2\Delta -V$ we obtain 
$ 0\leq \lambda_\e(c)\leq \bar\lambda_\e(V).$ As $V\leq V(0)=0$ and as $0$ is a non-degenerate critical point we deduce from Proposition \ref{Pr:Holcman} that
$ 0\leq \lambda_\e(V)\leq \Lambda \e$ for some constant $\Lambda$ that only depends on $\Vert V\Vert_{\mathscr C^2}$. \eqref{Eq:Holcman} follows.

\textbf{Blow-up procedure}
We  introduce the rescaled function 
$w_\e:(t,x)\mapsto u_{\e,V}(t,\sqrt{\e}x).$ Letting $\O\subset ]-\pi;\pi[^d$ be a neighbourhood of 0, $w_\e$ solves:
\[\begin{cases} \partial_t w_\e-\Delta w_\e=\frac1\e\lambda_\e(c)w_\e+c\frac{V(\sqrt{\e}\cdot)}\e w_\e\text{ in }(0,T)\times \frac1\e\O\,, 
\\ w_\e(T,\cdot)=w_\e(0,\cdot)\,, \\w_\e(0,0)=1.\end{cases}\]
Observe that locally uniformly 
\[ \frac{V(\sqrt{\e}\cdot)}\e\underset{\e \to 0}\rightarrow \left(x\mapsto\frac12 \langle x,\n^2V(0)x\rangle\right). \]
Recall that $A:=-\n^2 V(0)\in S_d^{++}(\R).$
Let $\overline \lambda(c)$ be a closure point of the sequence $\frac{\lambda_\e(c)}\e$ as $\e\to 0$. As $0\leq w_\e\leq 1$, proceeding as in \cite[Chapter 2]{zbMATH00049232} we deduce from parabolic estimates that $w_\e$ converges locally uniformly as $\e\to 0$ to a bounded function $w_0\in \mathscr C^{1,2}(\R^d)$ that solves
\begin{equation}\label{Eq:LimitEigenProblem}
\begin{cases}
\partial_t w_0-\Delta w_0=\overline\lambda(c)w_0-\frac12c(t)\langle x,Ax\rangle w_0&\text{ in }[0,T]\times \R^d\,, 
\\ w_0(T,\cdot)=w_0(0,\cdot)\,, 
\\ w_0(0,0)=1\,,
\\ w_0\geq 0.
\end{cases}\end{equation}
As $w_0$ is bounded, non-zero and non-negative, the conclusion follows.\end{proof}
\begin{remark}\label{Re:CvMin}
Observe that, at this stage, if we work with a weakly converging sequence $\{c_\e\}_{\e\to 0}$ rather than with a fixed $c$, the potential $(c_\e V(\sqrt{\e}\cdot)/\e)$ converges weakly to $-c_\infty\langle Ax,x\rangle$ and the same argument allows to conclude and to prove Proposition \ref{Pr:CvMin}.
\end{remark}

We can finally conclude the proof of the main theorem.
\begin{proof}[Proof of Theorem \ref{Th:SmallDiffusivity}]
We let, for any $\e>0$, $c^\e_{\min}$ be a solution of \eqref{Eq:PvLowDiffusivity}. Up to a subsequence, we know that $c^\e_{\min}\underset{\e\to0}\rightharpoonup c_0$ for some $c_0\in \mathcal C$. We deduce from Proposition \ref{Pr:BlowUp} that
\[ \frac{\lambda_\e(c^\e_{\min})+\langle c^\e_{\min}\rangle V(0)}\e\underset{\e\to 0}\rightharpoonup \overline \lambda(c_0).\] Passing to the limit in 
\[ \lambda_\e(c^\e_{\min})\leq \lambda_\e(c)\text{ for all }c\in \mathcal C\] we deduce that 
\[ \forall c\in \mathcal C\,, \overline\lambda(c_0)\leq \overline\lambda(c).\] From Theorem \ref{Th:Gaussian} we deduce that $c_0=\bar c$ and thus we have that $\{c^\e_{\min}\}_{\e\to 0}$ converges weakly in $L^\infty-*$ to $\bar c$ as $\e\to 0$. As $\bar c$ is an extreme point of $\mathcal C$ (see Remark \ref{Re:Topological}), this convergence is strong in all $L^p((0,T))$.
\end{proof}

\section{Proof of Theorem \ref{Th:SymmetryBreaking}}
We let $H:p\mapsto p^{2k}$. The proof of Theorem \ref{Th:SymmetryBreaking} boils down to proving that $\bar c$ does not satisfy first order optimality conditions for \eqref{Eq:PvHamiltonian} if $k\neq 1$. We thus first investigate these optimality conditions.

\textbf{Optimality conditions for \eqref{Eq:PvHamiltonian}}
Introduce, for any $c\in \mathcal C$, the solution $\psi_c$ of 
\begin{equation}\label{Eq:AdjointHamiltonian}
\begin{cases}
-\partial_t\psi_c-\Delta\psi_c=H'(\p_c)-\fiint_{\TT} H'(\p_c)\,, 
\\ \psi_c(T,\cdot)=\psi_c(0,\cdot)\,, 
\\ \fiint_\TT \psi_c=0.
\end{cases}\end{equation}The first order optimality conditions read:
\begin{lemma}\label{Le:OptimalityConditions}
Let $c_{\mathrm{min}}$ be a solution of \eqref{Eq:PvHamiltonian} and define 
\[ \Psi_{c_{\mathrm{min}}}:t\mapsto \int_\T \psi_{c_{\mathrm{min}}}(t,x)m(x)dt.\] Then
\[ \forall c \in \mathcal C\,, \int_0^T \Psi_{c_{\mathrm{min}}}c_{\mathrm{min}}\geq \int_0^T \Psi_{c_{\mathrm{min}}}c.\]In particular, if $\overline c$ is optimal, $\Psi_{\bar c}$ is symmetric with respect to $\frac{T}2=\pi$.
\end{lemma}
\begin{proof}[Proof of Lemma \ref{Le:OptimalityConditions}]
The Gateaux differentiability of  $J_H(c):=\iint_{\TT} H(\p_c)$ with respect to $c$ is a consequence of parabolic regularity and, letting $\dot\p_c$ denote the Gateaux derivative of the solution mapping $\mathcal C\ni c\mapsto \p_c$ at $c$ in a given direction $h\in L^\infty_{\mathrm{per}}((0,T))$ with $\langle h\rangle=0$, $\dot\p_c$ satisfies 
\begin{equation}\label{Eq:GateauxSolution}
\begin{cases}
\partial_t\dot\p_c-\Delta\dot\p_c=h(t)m(x)\in \TTo\,, 
\\ \dot\p_c(0,\cdot)=\dot\p_c(T,\cdot)\,, 
\\ \langle \dot\p_c\rangle=0.
\end{cases}
\end{equation}
Similarly, the Gateaux derivative of the map $J_H$ admits the expression 
\[ \dot J_H(c)[h]=\iint_\TTo \dot\p_cH'(\p_c). \]
Multiplying \eqref{Eq:GateauxSolution} by the solution $\psi_c$ of \eqref{Eq:AdjointHamiltonian} and integrating by parts, this gives
\[ \dot J_H(c)[h]=\iint_\TTo h(t)m(x)\psi_c(t,x)dtdx=\int_0^T h(t)\Psi_{c_{\mathrm{min}}(t)}dt.\]In particular, if $c_{\mathrm{min}}$ is a solution of \eqref{Eq:PvHamiltonian}, we deduce that, if we set $\Psi_{c_{\mathrm{min}}}:t\mapsto \int_\T m(x)\psi_{c_{\mathrm{min}}}(t,x)dx$, there holds
\begin{equation}\label{Eq:OC1}
\forall c \in \mathcal C\,, \int_0^T \Psi_{c_{\mathrm{min}}}(t) c_{\mathrm{min}}(t)dt\geq\int_{0}^T\Psi_{c_{\mathrm{min}}}(t) c(t)dt.
\end{equation}
Now, let us assume that $\overline c$ is optimal and let us write, for the sake of convenience, $\Psi$ for $\Psi_{\overline c}$. We know that for any $c\in \mathcal C$ we have 
\[ \int_0^T \Psi \bar c\geq \int_0^T \Psi c.\] Now, assume by contradiction that $\Psi$ is not symmetric. In order to alleviate notations, for any subset $E\subset(0,T/2)$, we let $\tilde E$ be the symmetric of $E$ with respect to $T/2$ and $\tilde\Psi$ be defined as 
\[ \tilde \Psi:t\mapsto \begin{cases} \Psi(T/2-t)\text{ if }t\in(0,T/2)\,,\\ \Psi(t)\text{ if } t\in (T/2;T).\end{cases} \] Since $\Psi$ is not symmetric, there exists an interval $I=(a,b)\subset (0,T/2)$ such that
\[\sup_{(a,b)}\tilde \Psi<\inf_{(a,b)}\Psi \text{ or }\inf_{(a,b)}\tilde\Psi>\sup_{(a,b)}\Psi\text{ in }(a,b).\] Since both cases are treated similarly, assume that 
\[\sup_{(a,b)}\tilde \Psi<\inf_{(a,b)}\Psi.\] The interval $\tilde I$ is explicitly given by $\tilde I=(T-b,T-a)$. 
We first assume that $\Psi>0$ in $(a,b)$.
Observe that $\bar c$ is decreasing. We let $d:=(a+b)/2$ and we define $c$ as 
\[ c:t\mapsto\begin{cases}
\bar c(t)\text{ if } t\in (0,T)\setminus (I\cup \tilde I)\,,
\\ \bar c(t)\text{ if }t\in ((a+b)/2,b)\cup(T-(a+b)/2,T-a)\,,
\\ \bar c(t+T-b-a)\text{ if }t \in (a,(a+b)/2)\,, 
\\ \bar c (t+(a+b)-T)\text{ if }t\in (T-b,T-(a+b)/2).
\end{cases}\]
Observe that as $\bar c$ is strictly increasing in $(0,T/2)$ there exists a constant $\delta>0$ such that for any $t\in((a+b)/2;b)$ \[ \bar c(t)-\bar c(t-(a+b)/2)\geq \delta,\] whence a direct comparison shows 
\[\int_0^T \Psi(\bar c-c)<0,\] a contradiction.\end{proof}
We now prove the following proposition (recall that $m(x)=\cos(x)$ and $\bar c(t)=\cos(\pi+t)$):
\begin{proposition}\label{Pr:NonOptimality}
$\bar c$ does not satisfy first-order optimality conditions for \eqref{Eq:PvHamiltonian}.\end{proposition}

\begin{proof}[Proof of Proposition \ref{Pr:NonOptimality}]
We can solve \eqref{Eq:Main3} explicitly. As $\fiint_{\TT} \bar c m=0$, $\p_{\bar c}$ solves
\[\partial_t\p_{\bar c}-\Delta\p_{\bar c}=\cos(\pi+t)\cos(x).\] Thus, 
\[ \p_{\bar c}=\omega(t)\cos(x)\text{ with }\begin{cases}\omega'(t)+\omega(t)=\cos(\pi+t)\,, 
\\ \omega(2\pi)=\omega(0).\end{cases}\]The function $\omega$ admits the explicit expression
\[\omega(t)=\frac{\cos(t+\pi)+\sin(t+\pi)}2.\]
Consequently the function $\psi=\psi_{c_{\mathrm{min}}}$ solves
\[ -\partial_t\psi-\Delta\psi=H'\left(\omega(t)\cos(x)\right)=(2k)\omega(t)^{2k-1}\cos(x)^{2k-1}.\] Since 
\[ \frac1\pi\int_{-\pi}^\pi\cos(x)^{2k-1}\cos(x)dx=\frac{1}{2^{2k-1}}\] it follows that the function $\Psi=\int_{-\pi}^\pi \psi\cos(\cdot)$ solves
\[ -\Psi'+\Psi=\frac{2k\pi}{2^{2k-1}}\omega^{2k-1}=a_k\omega^{2k-1}\text{ with $a_k=\frac{2k\pi}{2^{2k-1}}$}.\]

Applying the operator $(d/dt+\mathrm{Id})$ to the previous equation we obtain 
\[\frac1{a_k}( -\Psi''+\Psi)=\omega(t)^{2k-2}\left((2k-1)\omega'(t)+\omega(t)\right).\] If $\Psi$ were symmetric with respect to $t=\pi$, then so would the map
\[F:t\mapsto \omega(t)^{2k-2}\left((2k-1)\omega'(t)+\omega(t)\right).\]
However, observe that 
\[ F'(t)=(2k-1)(2k-2)\omega'(t)^2\omega(t)^{2k-3} +(2k-1)\omega''(t)\omega(t)^{2k-2}+(2k-1)\omega'(t)\omega(t)^{2k-2}.
\]
As 
\[ \omega(\pi)=\omega'(\pi)=-\omega''(\pi)=\frac12\]we conclude that 
\[F'(\pi)=(2k-1)(2k-2)\frac1{2^{2k-1}}-\frac{2k-1}{2^{2k-1}}+(2k-1)\frac{1}{2^{2k-1}}=\frac{(2k-1)(2k-2)}{2^{2k-1}}.
\]In particular, $F'(\pi)=0$ if and only if $k=1$.
\end{proof}

\section{Numerical simulations}\label{Se:Numerics}

In the following we illustrate the theoretical results from previous sections through a series of numerical simulations. These help us formulate several conjectures in the conclusion of the article.

\subsection{General structure of the section}
In this section, we work in the one-dimensional (in space) setting:  we identify $\mathbb T$ with the interval $[0,1]$. We proceed in steps of increasing complexity: after explaining how to handle, at a numerical level, admissible classes of the form \eqref{Eq:Adm2} (which is non standard when compared to the existing literature, see for instance \cite{zbMATH07605268} and the references therein), we study, successively, the following problems (the order of the problems corresponds to their intrinsic mathematical and computational difficulties):
\begin{enumerate}
\item First, we investigate Talenti-like problems, that is, we focus on \eqref{Eq:PvHamiltonian}. We give a numerical illustration of the symmetry breaking phenomenon of Theorem \ref{Th:SymmetryBreaking}.
\item Second, we focus on \eqref{Pv:2}, and our numerical simulations back up Theorems \ref{Th:LargeDiffusivity}--\ref{Th:SmallDiffusivity}: it seems likely that, in general, an optimiser $c_{\min}$ of \eqref{Pv:2} satisfies $c_{\min}=c_{\min}^\#$.
\item Finally, we turn to \eqref{Pv:1} where we also obtain, numerically, the monotonicity in time of optimisers.
\end{enumerate}

\subsection{Generalities on the discretisation of the equations and of the constraints}
\subsubsection{Discretisation of the PDE and of the objective functional}
We discretise the time interval $[0,T]$ using a uniform partition
\[ 0=t_0<t_1<...<t_N = T, t_i = i/N.\]
Similarly, the space interval $[0,1]$ is discretised using a uniform partition of $[0,1]$
\[ 0=x_0<x_1<...<x_M = 1, x_i = i/M.\] 
The time and space steps are denoted by $(\Delta t) = T/N, (\Delta x) = 1/M$ respectively. The periodicity conditions in time and space are implemented by identifying $t_{i+N}=t_i$ and $x_{i+M} = x_i$. Therefore, indices are extended by periodicity in time and space whenever necessary.

Given a function $u:[0,T]\times \Bbb T \to \Bbb R$  consider the matrix $U \in \Bbb{R}^{M\times N}$ approximating the values of $u$ at the nodes of a finite difference grid on $[0,T]\times \Bbb T$ with nodes given above by $u(t_j,x_i) \approx U_{i,j}$. Denote by $U_j$, $j=0,...,N$, the vector approximating $U(t_j,\cdot)$ at the nodes of the space grid. 

The Laplace operator $(-\Delta)$  is approximated using second order finite differences, through the matrix
	\[ A = \frac{1}{\Delta x^2}\begin{pmatrix}
	2 & -1 & ... & ...  & -1 \\
	-1 & 2 & -1 & ... & 0 \\
	... & ... & ... & ... & ... \\
	-1 & ... & ... & -1 & 2
	\end{pmatrix}.\]Thus, for a fixed $0\leq j \leq N-1$
 \[ -\left( \frac{U_{i+1,j}-2U_{i,j}+U_{i-1,j}}{\Delta x^2} \right)_{i=0}^{M-1} =  A U_j.\]
 The time derivative is approximated using first order finite differences. The parabolic equation is discretised implicitly for the diffusion term, and explicitly for the remaining terms. 

 {\bf The direct problem.} We first study the discretisation of problem \eqref{Eq:Main3}. Without loss of generality, assume that $\int_{\Bbb{T}}V(x)dx = 0$. The discrete approximation for \eqref{Eq:Main3} is given by
\[ \frac{U_{n+1}-U_n}{(\Delta t)}+AU_{n+1}=c(t_i)V, \text{ for every } 0 \leq n \leq N-1
 \]
 where we denote we a slight abuse of notation $V$ the discretisation of $V$ at space points $x_i$. Re-factoring gives
 \[ (I+(\Delta t) A)U_{n+1}-U_n=(\Delta t) c_i(t)V, \text{ for every } 0 \leq n \leq N-1.\]
 Taking into account the periodicity in time we find that the discrete problem is equivalent to
  \begin{equation}\label{eq:discrete-main3}
 \mathcal A \bo U = \bo R,\end{equation}
	where $\bo U \in \Bbb{R}^{M\times N}$ is the vector obtained by concatenating $\bo U = (U_0,...,U_{N-1})$ and $B = (I+(\Delta t)A)$, $\mathcal A$ is the block matrix
	\begin{equation}
	    \label{eq:matrix-A-block}
	 \mathcal A = \begin{pmatrix}
	-I & B & 0 & ... & 0 \\
	0 & -I & B & ... & 0 \\
	... & ... & ... &... &... \\
	B & 0 & ... & ... & -I
	\end{pmatrix} \end{equation}
	and the right hand side is given by
	\[ \bo R = \begin{pmatrix}
	c_0 V \\ c_1 V \\\vdots \\ c_{N-1} V
	\end{pmatrix}.\]

Note that the system $\mathcal A \bo U = \bo R$ discretises simultaneously the time and space aspects of the equation. This allows us to treat the periodicity in time in a straightforward way, without resorting to iterative approximations. Up to a constant factor, the discrete version of \eqref{Eq:PvHamiltonian} is
\[ \max_{c} \|\bo U\|_{\ell^p}\]
for some $p \geq 2$. In order to solve \eqref{Eq:PvHamiltonian}, we need to discretise the adjoint, which corresponds to another parabolic equation. To avoid complications we differentiate the discrete system directly.

It is standard to compute derivatives for $J(\bo U) = \sum_{i,j} j(U_{i,j})$ under the constraint \eqref{eq:discrete-main3}. Indeed, if $h \in \Bbb{R}^N$ is a perturbation for $c$ then 
\[ J'(\bo U(c))(h)=\sum_{i,j}j'(U_{i,j})U_{i,j}(c)'(h),\]
where $\bo U'(h)$ satisfies
\[ \mathcal A \bo U'(h) = \begin{pmatrix}
	h_0 V \\ h_1 V \\\vdots \\ h_{N-1} V 
	\end{pmatrix}.\]
We consider the solution $\bo P$ to the adjoint system
 \[ \mathcal A^T \bo P = 
     j'(\bo U)
,\]
 which implies that
 \[ J'(\bo U(c))(h) =\bo P\cdot  \begin{pmatrix}
	h_0 V \\ h_1 V \\\vdots \\ h_{N-1} V
	\end{pmatrix}.\]

This expression yields
 \begin{equation}\label{eq:deriv-c}
  \frac{\partial J(c)}{\partial c_i}=p_i \cdot V.
  \end{equation}

The two types of constraints considered for $c$ are introduced below. Each one of them is preserved during the gradient descent/ascent algorithm using projections, which will be detailed below.

\subsubsection{Discretisation of the classical constraints} We first study constraints that we call classical (see Remark \ref{Re:Classical}), that is, corresponding to  $a\leq c\leq b$, $\fint_0^T c(t)dt = c_0 \in [a,b]$. To construct the projection on this constraint, observe that the mapping 
 \[ \tau \mapsto \fint_0^T \max\{a,\min\{c(t)+\tau,b\}\}dt \]
 is increasing in $\tau$. Therefore, using a dichotomy algorithm we find a threshold $\tau_0$ such that 
\[\fint_0^T \max\{a,\min\{c(t)+{\tau_0},b\}\}dt=c_0\]
where the equality is understood up to a fixed tolerance chosen beforehand. The projection operator is
\begin{equation}\label{eq:projection-algo}
    P(c) = \max\{a,\min\{c(t)+\tau_0,b\}\}.
\end{equation} 
While the projection operator \eqref{eq:projection-algo} is defined in the continuous context, the discrete version is obtained in a  straightforward way by replacing the integral with the average ($a$ and $b$ can also vary with $t$; in this case replace $a,b$ with $a(t)$ and $b(t)$ in the truncation operator above).

\bo{Numerical simulations.} We solve \eqref{Eq:PvHamiltonian} numerically with the aid of a gradient ascent algorithm with projection. The vector $c$ is updated using the usual formula $c^{k+1}=P(c^k+\gamma_k\nabla J(c^k))$ ($P$ is the projection operator defined in \eqref{eq:projection-algo}) where $\gamma_k>0$ is an ascent step chosen such that the objective function increases: $J(c^{k+1})>J(c^k)$. The step $\gamma_k$ is updated at every iteration: if the objective function increases, the step is slightly increased; if the objective function does not increase, the iteration is repeated with a smaller step. Simulations are ran for various bounds and average constraints for $c$. The variable $c$ is initialised randomly, then projected onto the constraints using \eqref{eq:projection-algo}. In the numerical computations we choose $V(x)=\cos x$. 
Interestingly enough, at a numerical level and for such constraints, we do not observe the symmetry breaking phenomenon of Theorem \ref{Th:SymmetryBreaking}: $c$ is equal to its maximal value on a connected interval and is equal to its minimal value on the complement. Simulations are performed using $100$ discretisation points for the spatial variable and $200$ discretisation points for the time variable. The final time is $T=2\pi$ and the periodic $x$ variable belongs to $[-\pi,\pi]$. Numerical simulations generally converge in under $100$ iterations for the parameters recalled above. Examples are shown in Figure \ref{fig:direct-classical}.

\begin{figure}
    \centering
    \begin{tabular}{cc}
           \includegraphics[width=0.45\linewidth]{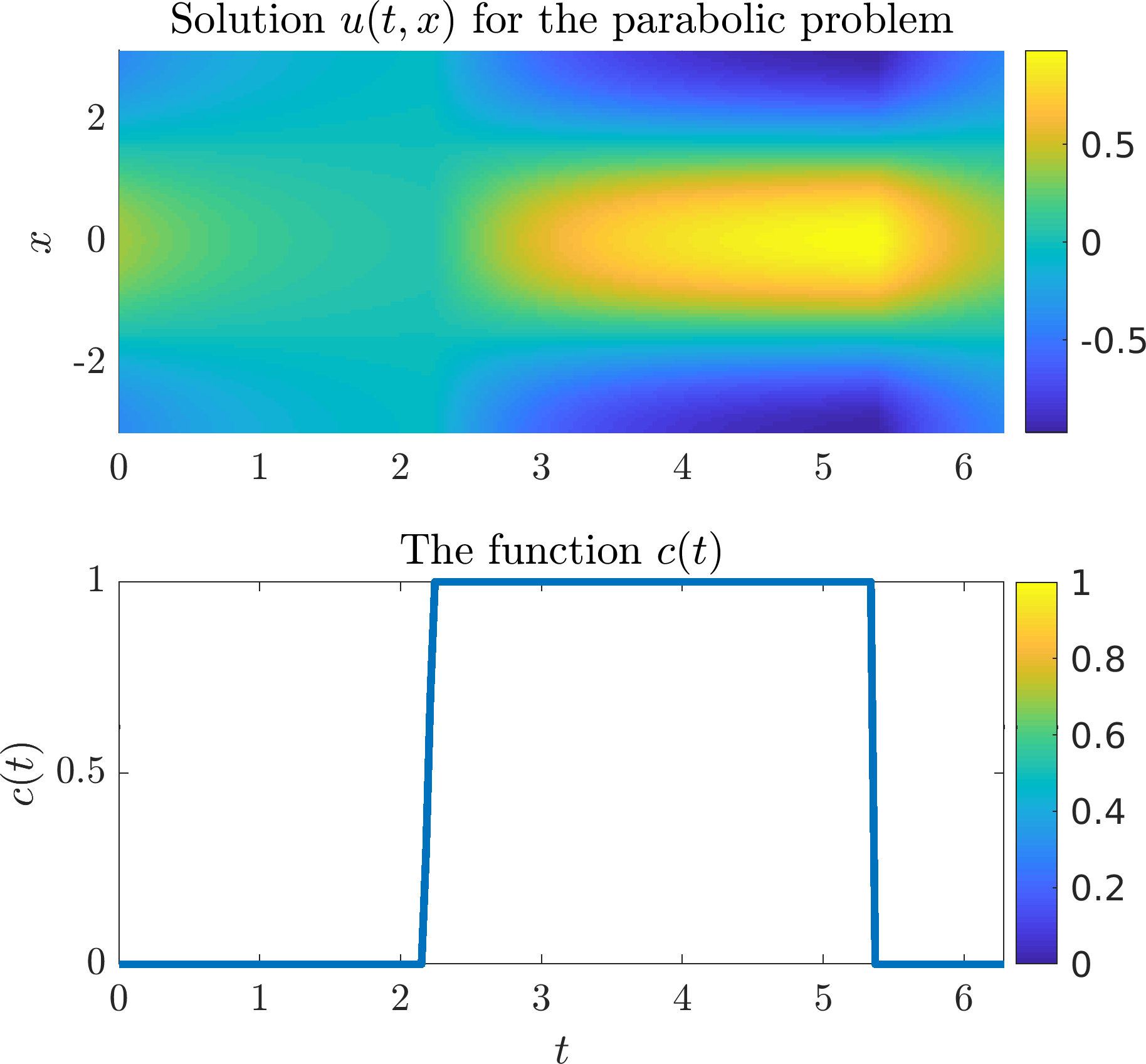}
  &     \includegraphics[width=0.45\linewidth]{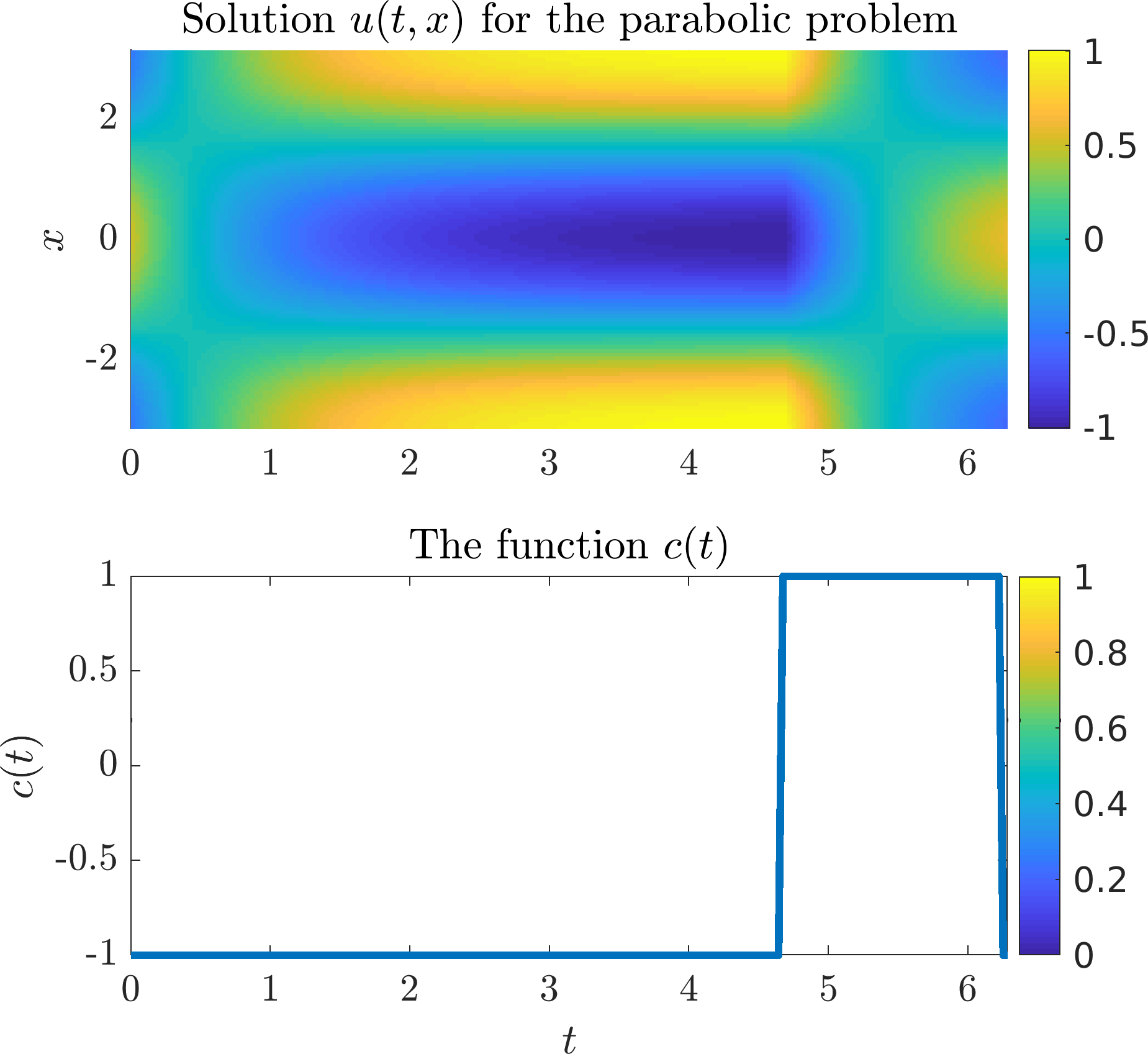}
 \\
     $0 \leq c\leq 1$, $\fint c = 0.5$, $p=10$    & 
     $-1 \leq c\leq 1$, $\fint c = -0.5$, $p=2$ 
    \end{tabular}
    \caption{Numerical solutions for problem \eqref{Eq:PvHamiltonian} for various parameters, using bound and average constraints on the function $c$.}
    \label{fig:direct-classical}
\end{figure}

\subsubsection{Discretisation of rearrangement constraints}
\label{sec:rearrangement}
In this case we handle the case of functions $c$ having a fixed rearrangement $c^{\#} = \bar c$. Let $c_{\min}, c_{\max}$ be the minimum and maximum values of $c$ on $[0,T]$. Consider the domain $[0,T] \times [c_{\min},c_{\max}]$ divided into the discrete grid given by $t_0,...,t_N$, the time discretisation instances and $c_{\min}=v_0,...,v_K=c_{\max}$, an equi-distributed discretisation of the image space of $c$ on $[0,T]$. We set $\Delta c:=(c_{max}-c_{min})/K$. Let $F_{i,j}$ be the area of the graph of $c$ contained in the cell $[t_j,t_{j+1}]\times[v_i,v_{i+1}]$ for $0 \leq i \leq K-1$, $0\leq j \leq N-1$. See Figure \ref{fig:rearrangement-cos} for an illustration containing the proportions $F_{i,j}$ for the cosine function and $K=N=10$.
\begin{figure}
    \centering
    \includegraphics[width=0.7\linewidth]{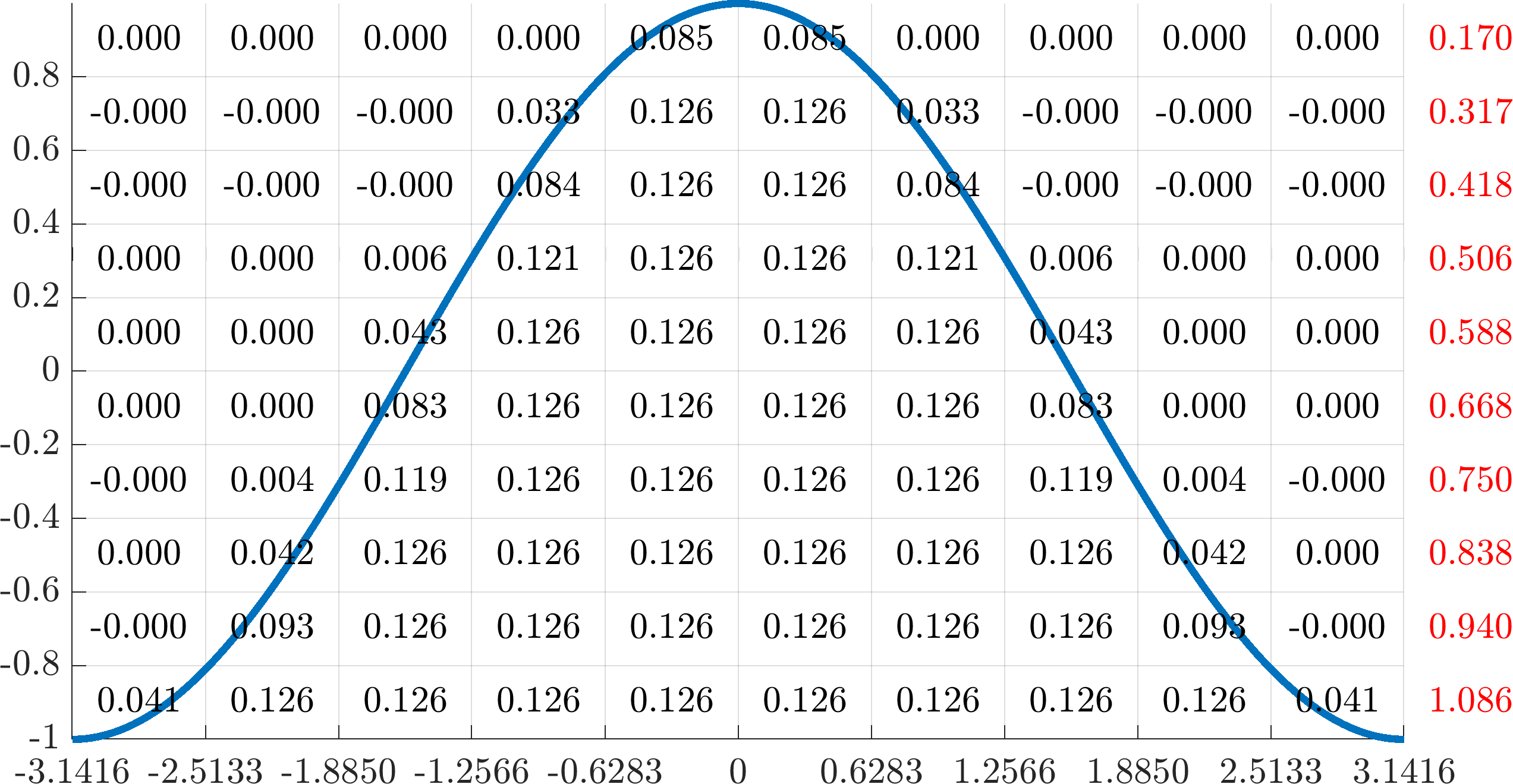}
    \caption{The region $[-\pi,\pi]\times [-1,1]$ containing the graph of the cosine function is divided into $N\times K$ rectangles. The covered area in each small rectangle is indicated in the figure. The area of a small rectangle is $(\Delta t)\times (\Delta c) \approx 0.126$. { A function having the same discrete rearrangement will be characterised by a $K\times N$ matrix with entries decreasing on columns, having fixed sum on lines (rightmost column above).}}
    \label{fig:rearrangement-cos}
\end{figure}

The numerical discretisation of the rearrangement constraint we propose is inspired from two natural observations.\begin{itemize}
    \item If $c$  is a function then $F_{i,j}\geq F_{i+1,j}$: the "quantity" of $c$ contained in a given interval, on a given slice is larger than what lies in a superior slice. 
    \item If $c$ has a prescribed rearrangement then $\sum_{j=0}^{N-1}F_{i,j}$ is fixed. This follows from the Cavalieri principle.\end{itemize}

We propose the following discrete rearrangement constraint. If $\overline c$ is a symmetric decreasing function with respect to $T/2$ then for any $i\in \{0,\dots,K-1\}$ consider 
\[q_i = \left|\{({t},y): {t} \in [0,T], y \in [v_i,v_{i+1}], y \leq c(t)\}\right|.\]
The sequence $q_i$ is decreasing.

Consider a matrix $F_{i,j}$ of size $(K\times N)$ and entries verifying for all $0\leq j \leq N-1$, $0 \leq i \leq K-1$ (extended by periodicity):
\[ F_{i,j} \geq F_{i+1,j}, \quad \sum_{j=0}^{N-1} F_{i,j}=q_i.\]
We reconstruct the discrete function $c$ by summing on the columns of $F$:
\begin{equation}\label{eq:sum-F} c_j = c_{\min}+\frac{1}{\Delta x}\sum_{i=0}^{K-1}F_{i,j}.
\end{equation}
In Figure \ref{fig:rearrangement-arb} an example of a different matrix $F$ corresponding to a discrete rearrangement of the cosine function is given. Together with the values, the cumulative sum on columns computed with \eqref{eq:sum-F} is also illustrated. 

\begin{figure}
    \centering
    \includegraphics[width=0.7\linewidth]{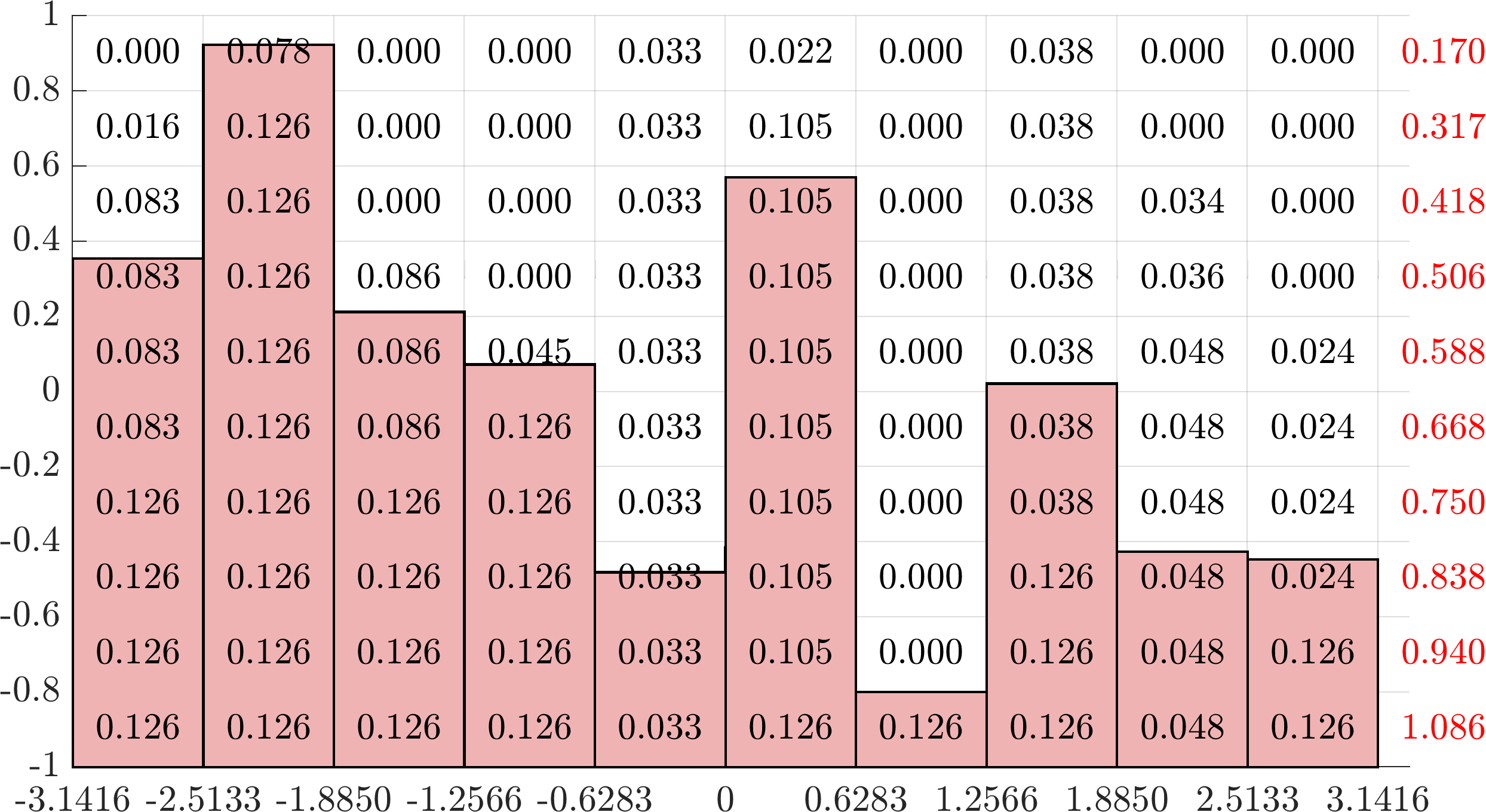}
    \caption{An example of a discrete function defined on the $N\times K$ rectangular grid having the required properties to be a discrete rearrangement of the cosine function: sum of values of lines equals to the same quantity for the cosine, values are decreasing on vertical columns. The area of a small rectangle is $(\Delta t)\times (\Delta c) \approx 0.126$.}
    \label{fig:rearrangement-arb}
\end{figure}

The cumulative sum is not the most representative illustration, since one discrete interval $[t_i,t_{i+1}]$ may not be fully covered at some of the intermediary levels. Another way of interpreting the discrete rearrangement proposed is shown in Figure \ref{fig:rearrangement-rectangles} where in each cell a rectangle corresponding to the fraction covered by the discrete function in the current cell is shown. Of course, when $N$ and $K$ are large enough, the description of the rearrangement becomes more and more precise. In the numerical simulations we use $N\geq 100$ and $K=100$ vertical slices to characterize the discrete rearrangement. Note that in order to characterize a one dimensional function with prescribed rearrangement, a two dimensional array is required. 

Similar to our previous computations, we obtain 
\[ \frac{\partial J}{\partial F_{i,j}} = \frac{\partial J}{\partial c_j} \frac{\partial c_j}{\partial F_{i,j}} = \frac{1}{\Delta t} p_i\cdot V.  \]

\begin{figure}
    \centering
    \includegraphics[width=0.7\linewidth]{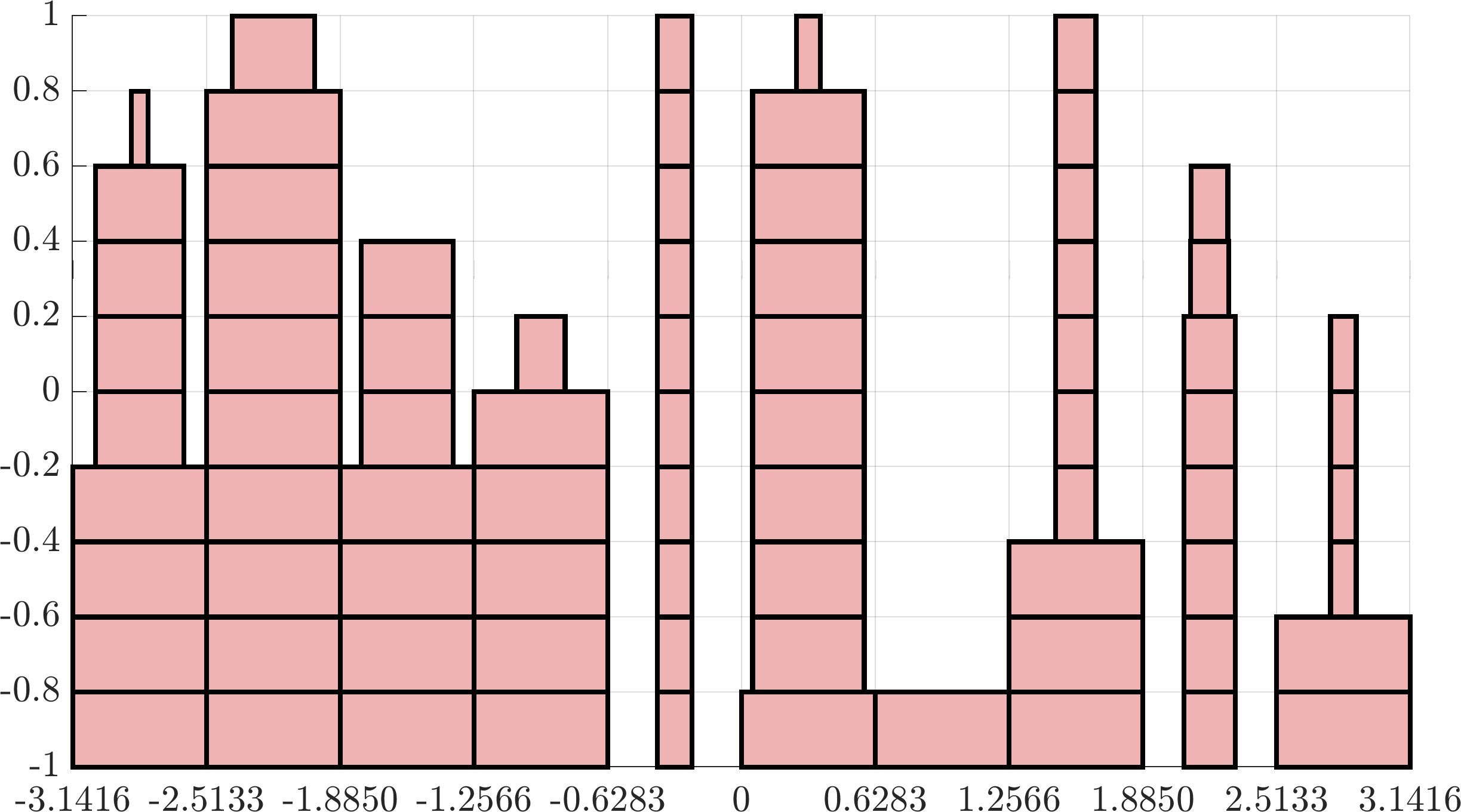}
    \caption{An alternative illustration for the discrete rearrangement of the cosine illustrated in Figure \ref{fig:rearrangement-arb}. In each one of the $N\times K$ cells a rectangle is plotted having area proportional to the area covered by the rearrangement in that cell.}
    \label{fig:rearrangement-rectangles}
\end{figure}

This constraint is also preserved in a gradient algorithm using projections. Indeed, note that by construction $F_{i,j}$ is an area proportion in a rectangle of area $(\Delta t)\times (\Delta c)$, where $\Delta c = \frac{c_{\max}-c_{\min}}{K}$. Therefore $0 \leq F_{i,j}\leq (\Delta t)\times (\Delta c)$. This leads to the following iterative algorithm:

\begin{itemize}
    \item \bo{First slice:} only bound and average constraints. Find $F_{0,j}$, $0 \leq j \leq N-1$ such that
    \[ 0 \leq F_{0,j} \leq (\Delta t)\times (\Delta c), \sum_{j=0}^{N-1} F_{i,j}=q_0.\]
    To achieve this, simply apply the projection algorithm stated in \eqref{eq:projection-algo}.
    \item \bo{Subsequent slices}: fixed average, smaller than previous slice. Given $i\geq 1$, find $F_{i,j}$, $0 \leq j \leq N-1$ such that
    \[ 0 \leq F_{i,j} \leq F_{i-1,j}, \sum_{j=0}^{N-1} F_{i,j}=q_i.\]
    To achieve this, apply the projection algorithm stated in \eqref{eq:projection-algo}, with variable bounds.
    \item Recover the discrete function $c$ using \eqref{eq:sum-F}.
\end{itemize}

\bo{Numerical Simulations.} Like in the previous case, we take $T=2\pi$ and $V(x) = \cos(x)$ for $x \in [-\pi,\pi]$. Assume that $c:[0,T] \to [-1,1]$ has the same symmetric rearrangement as $\cos(t)$. The rearrangement constraint is implemented using $K=100$ vertical slices, using the same notations indicated above. Two simulations are shown where \eqref{Eq:PvHamiltonian} is maximised for $p=2$ and $p=6$. In the first case, for $p=2$ the optimal $c$ is identically equal to $\cos$. For $p=6$ the symmetry is lost, as indicated by the theoretical results of Theorem \ref{Th:SymmetryBreaking}. Results are illustrated in Figure \ref{fig:rearrangement-classical}.

\begin{figure}
    \centering
    \begin{tabular}{cc}
         \includegraphics[width=0.49\linewidth]{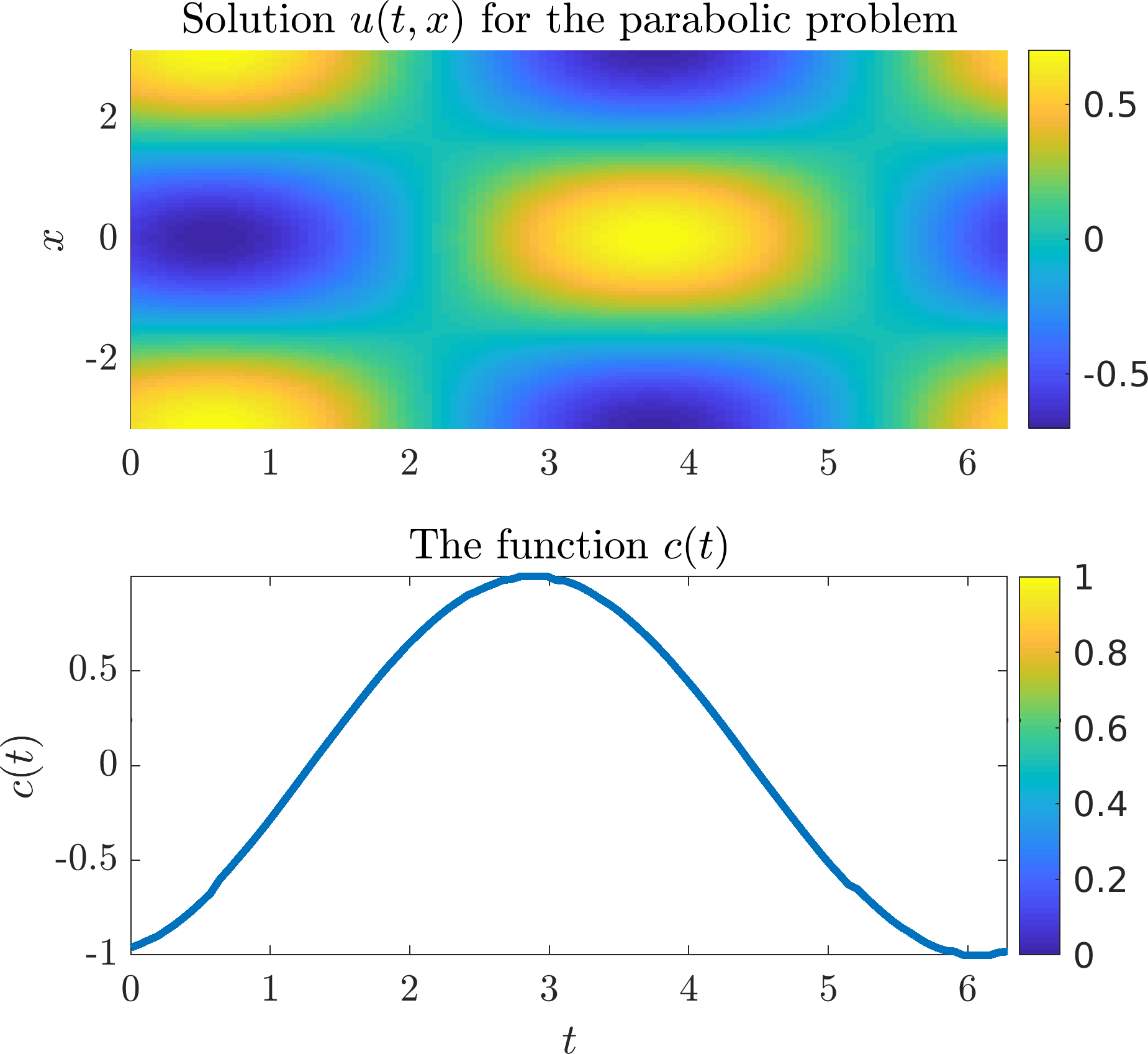}
    & \includegraphics[width=0.49\linewidth]{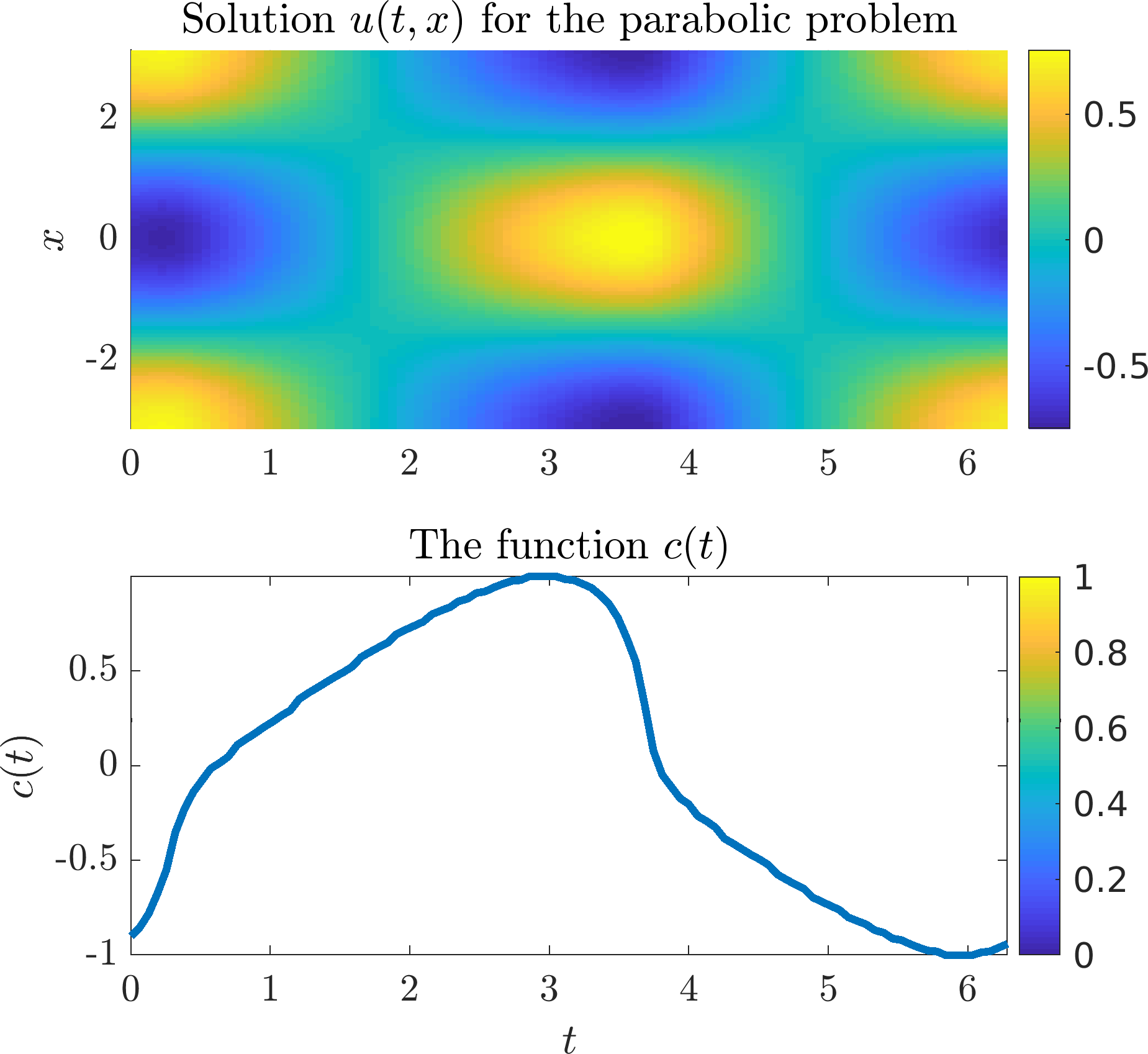} \\
     $p=2$    &  $p=6$
    \end{tabular}
    \caption{Numerical optimisation for the functional \eqref{Eq:PvHamiltonian} for functions $c$ having a prescribed rearrangement. Numerical results show that there is a different behaviour of the numerical optimiser with respect to the parameter $p$. For $p=2$ the solution is symmetric, while for $p=6$ the symmetry is lost. This confirms the results of Theorem \ref{Th:SymmetryBreaking}.}
    \label{fig:rearrangement-classical}
\end{figure}

\subsection{The eigenvalue problem}The discretisation of the eigenvalue problem \eqref{Eq:Main1} follows the same lines as the one used for the direct problem. The resulting discrete problem has the form
\begin{equation}\label{eq:discrete-eig}
(\mathcal A -(\Delta t)\diag(m))\bo U = (\Delta t)\lambda(m) \bo U,
\end{equation}
where $\mathcal A$ is the same matrix as in \eqref{eq:discrete-main3}, \eqref{eq:matrix-A-block} and $\diag(m)$ denotes the diagonal matrix with entries given by $m_{i,j}\approx m(x_i,t_j)$. We are interested in the principal eigenvalue for problem \eqref{eq:discrete-eig} for which we have $\bo U>0$ pointwise. 
More precisely, the elements on the diagonal of $\diag(m)$ are, in order, given by the columns of the matrix $m_{i,j}$: $(m_{0,j})_{j=0}^{M-1}$, $(m_{1,j})_{j=0}^{M-1}$, ..., $(m_{N-1,j})_{j=0}^{M-1}$.

 We observe that adding a constant $\delta m$ to $m$ and setting $\tilde m=m+\delta m$ simply shifts the eigenvalue: $\lambda(\tilde m)=\lambda(m)-\delta m$. Since the matrix $\mathcal A$ already has negative eigenvalues, to guarantee that the eigenvalue of the smallest absolute value found numerically (using standard libraries like \texttt{eigs} in Matlab) has a constant sign eigenvector, we shift $m$ with a constant such that all its components are strictly positive and, change $\lambda(m)$ accordingly to find the discrete eigenvalue \eqref{eq:discrete-eig}. 
 
Our objective is to optimise the eigenvalue $\lambda(m)$ with respect to $m$. Like in the previous case, to avoid technical difficulties, we differentiate the discrete problem \eqref{eq:discrete-eig} in order to optimise $\lambda(m)$ numerically. Considering $h \in \Bbb{R}^{M\times N}$ a perturbation of $m$ and differentiating \eqref{eq:discrete-eig} gives
\begin{equation}\label{eq:discrete-eig-deriv}
(\mathcal A -(\Delta t)\diag(m))\bo U'-(\Delta t)\diag(h)\bo U = (\Delta t)\lambda'(m)(h) \bo U+(\Delta t)\lambda(m)\bo U'.
\end{equation}
The adjoint problem for \eqref{eq:discrete-eig} is the same eigenvalue problem, but for the transpose matrix (remember that $\mathcal A$ defined in \eqref{eq:matrix-A-block} is not symmetric)
\begin{equation*}\label{eq:discrete-eig-adj}
(\mathcal A^T -(\Delta t)\diag(m))\bo V =(\Delta t) \lambda(m) \bo V,
\end{equation*}
in other words, $\bo V$ is a left eigenvector for the same eigenvalue $\lambda(m)$, which also satisfies $\bo V>0$ pointwise. Multiplying \eqref{eq:discrete-eig-deriv} with $\bo V^T$ and simplifying leads to an explicit formula for the derivative of $\lambda(m)$
\begin{equation}\label{eq:deriv-lambda}
\lambda'(m)(h) = -\frac{1}{\bo V^T \bo U} \bo V^T(\diag(h) \bo U).
\end{equation}
Since $\bo U, \bo V>0$ the scalar product $\bo V^T\bo U$ is strictly positive. More precisely we find that
\[ \frac{\partial \lambda(m)}{\partial m_{i,j}} = -\frac{1}{\bo V^T\bo U}U_{i,j}V_{i,j}.\]

\bo{Numerical simulations.} \bo{(a) Separated time/space potentials:} We choose $T=2\pi$, $100$ discretisation points for the space variable and $200$ discretisation points for the time variable. For potentials of the form $c(t)m(x)$ we choose $m(x) = \cos(x)$ and $c$ verifying either classical constraints (bounds, fixed average) or rearrangement constraints. The eigenvalues are computed in Matlab using the \texttt{eigs} routine. We refer to Figure \ref{fig:rearrangement-1D}.
\begin{figure}
    \centering
    \begin{tabular}{cc}
           \includegraphics[width=0.48\linewidth]{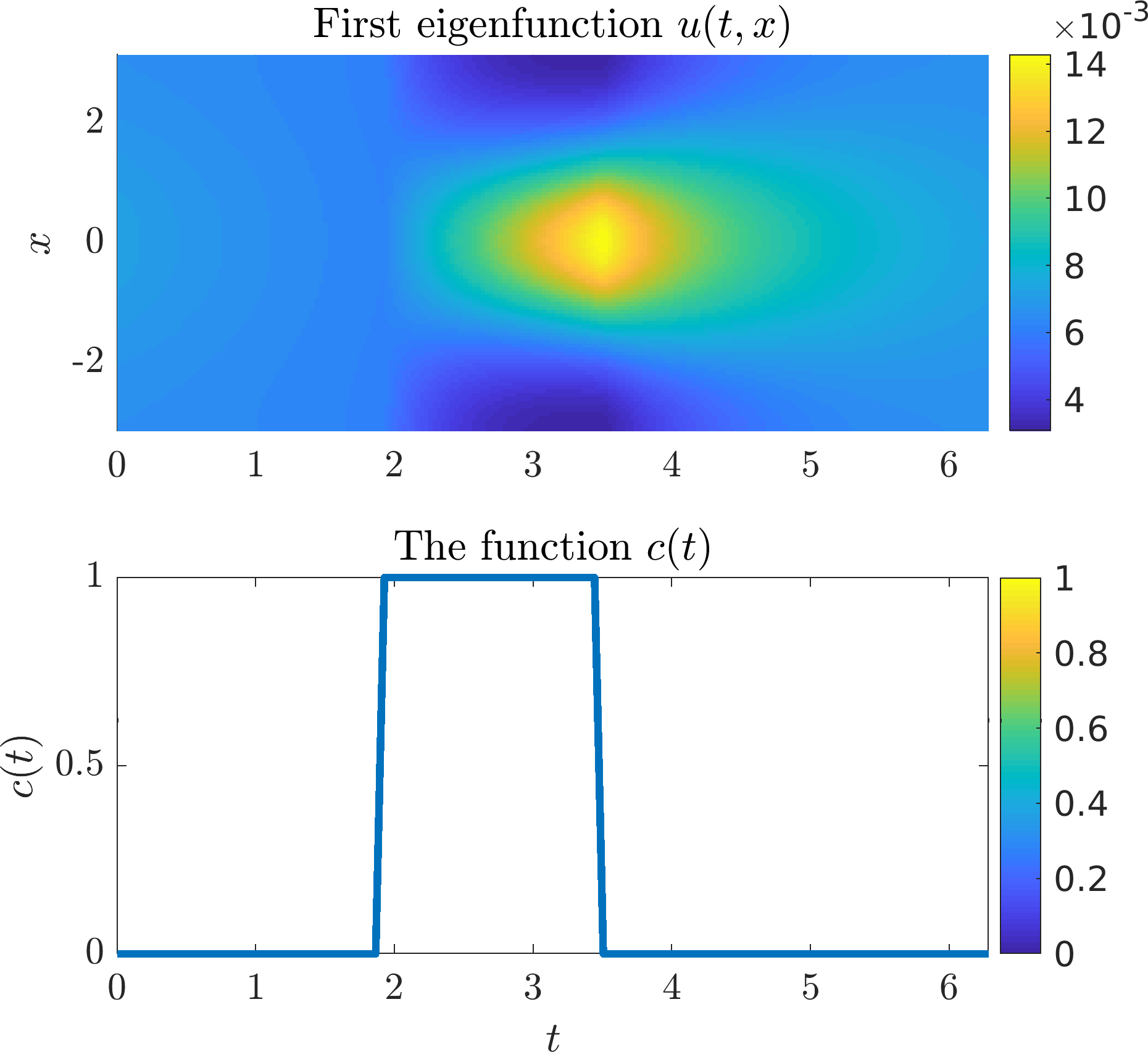}
  &     \includegraphics[width=0.48\linewidth]{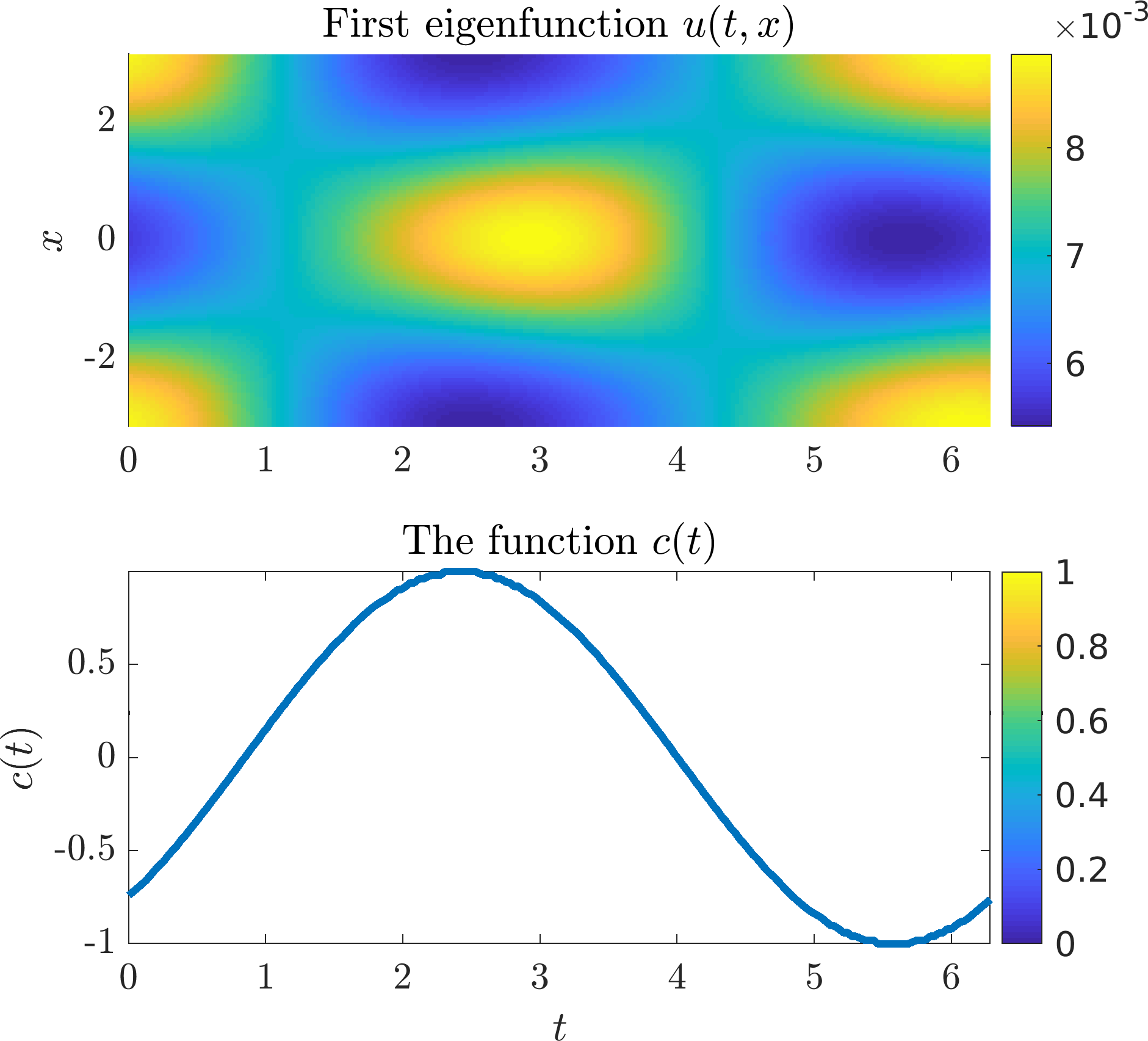}
 \\
     $0 \leq c\leq 1$, $\fint c = 0.5$    & 
     $c^{\#} = \cos(t)$ 
    \end{tabular}
    \caption{Numerical minimisationation of the principal eigenvalue in \eqref{Eq:Main1} with respect to potentials of the form $c(t)m(x)$ with $m(x)=\cos(x)$ and $c(t)$ verifying classical or rearrangement constraints.}
    \label{fig:rearrangement-1D}
\end{figure}
%
%\bo{(b) General 2D potentials under classical constraints.} The numerical framework also allows us to optimise with respect to general potentials $m(t,x)$ verifying constraints of the form
%\[ -k \leq m \leq 1, \fiint m = m_0.\]
%A gradient descent algorithm with projection gives the result shown in Figure \ref{fig:eigenvalue-general}. For all tested average and bound constraints the result corresponds to an interval in space times the whole time interval: $[a,b]\times [0,T]$.
%
%\begin{figure}
%    \centering
%    \begin{tabular}{cc}
%    \includegraphics[width=0.5\linewidth]{pics/Eigenvalue_Classical_Full.png} &
%    \includegraphics[width=0.5\linewidth]{pics/Eigenvalue_Classical_Full2.png} \\
%    $-1\leq m \leq 1,\ \fiint m = 0$ &
%    $-0.5\leq m \leq 1,\ \fiint m = 0$ 
%    \end{tabular}
%    \caption{optimisation for general potential $m(t,x)$ under bound and average constraints. The numerical result is of the form $[0,T]\times [a,b]$, where the interval $[a,b]$ has length depending only on the bounds and the average constraint.}
%    \label{fig:eigenvalue-general}
%\end{figure}

\bo{(b) Two dimensional potentials with rearrangement constraints.} The discrete framework presented in Section \ref{sec:rearrangement} for the rearrangement constraint can be extended \emph{ad literam} in higher dimensions. 

To parametrize potentials $m: [0,T] \times \Bbb T \to \Bbb{R}$ which have a fixed rearrangement consider the following:
\begin{itemize}
    \item Divide the region $[0,T]\times \Bbb{T}\times [\min m,\max m]$ into $N \times M\times K$ boxes labeled $B_{i,j,k}$ using equidistributed points along each dimension. Denote by $h_0,...,h_K$ a discretisation of $[\min m,\max m]$
    \item Consider $F_{i,j,k}$ the proportion of the volume of the graph of the potential $m$ contained in $B_{i,j,k}$.
    \item Take $F_{i,j,k}$ as optimisation variables assuming that $k \mapsto F_{i,j,k}$ is decreasing and $\sum_{i,j} F_{i,j,k}=q_k$ is fixed and equal to the integral of $m$ on $\{(t,x,z) : z \in [h_i,h_{i+1}]\}$.
\end{itemize}

Given $F_{i,j,k}$, satisfying the constraints above, re-construct $m$ using the formula
\begin{equation}\label{eq:sum-F2D} m_{i,j} = m_{\min}+\frac{1}{\Delta t\cdot \Delta x}\sum_{k=0}^{K-1}F_{i,j,k}.
\end{equation}

The eigenvalue can become degenerate if one is not careful with the time scaling and diffusivity of the equation, in the sense that the gradient will be extremely low. To overcome this numerical hurdle, we consider the eigenvalues of the operator $\alpha \partial_t -\beta \Delta - m I$. This amounts to a scaling of the time-space domain. Simulations presented below support the conjecture that the optimal potential under a rearrangement constraint is symmetric in time and in space. Simulations are made for $(t,x) \in [0,T]\times [-\pi,\pi]$. The optimisation algorithm is a gradient descent algorithm with projection on the discrete rearrangement constraint. The initialisation is chosen randomly. Precise details are given for each computation. 

\bo{Computation 1.} Parameters: $\alpha=0.1$, $\beta = 10$, $m^\#(t,x) = \cos(x)\cos(2\pi t)$ on $[0,1]\times [-\pi,\pi]$. The discretisation uses $201$ time steps, $201$ space steps and $200$ slices for the rearrangement constraint. The initialisation is chosen at random and projected onto the rearrangement constraint. The gradient descent algorithm is ran for $1000$ iterations. 

\begin{figure}
    \centering
    \includegraphics[width=0.48\linewidth]{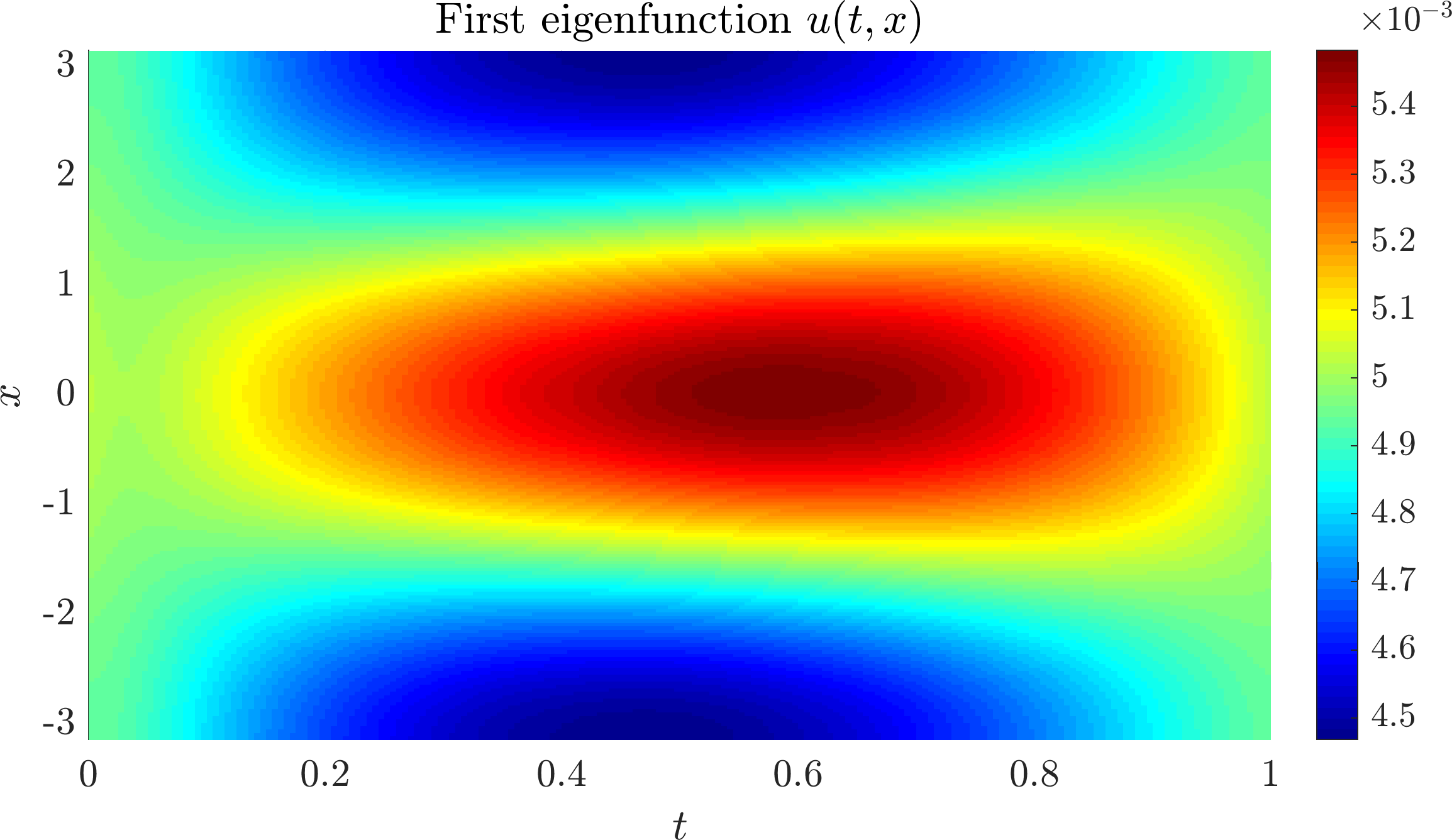}
    \includegraphics[width=0.48\linewidth]{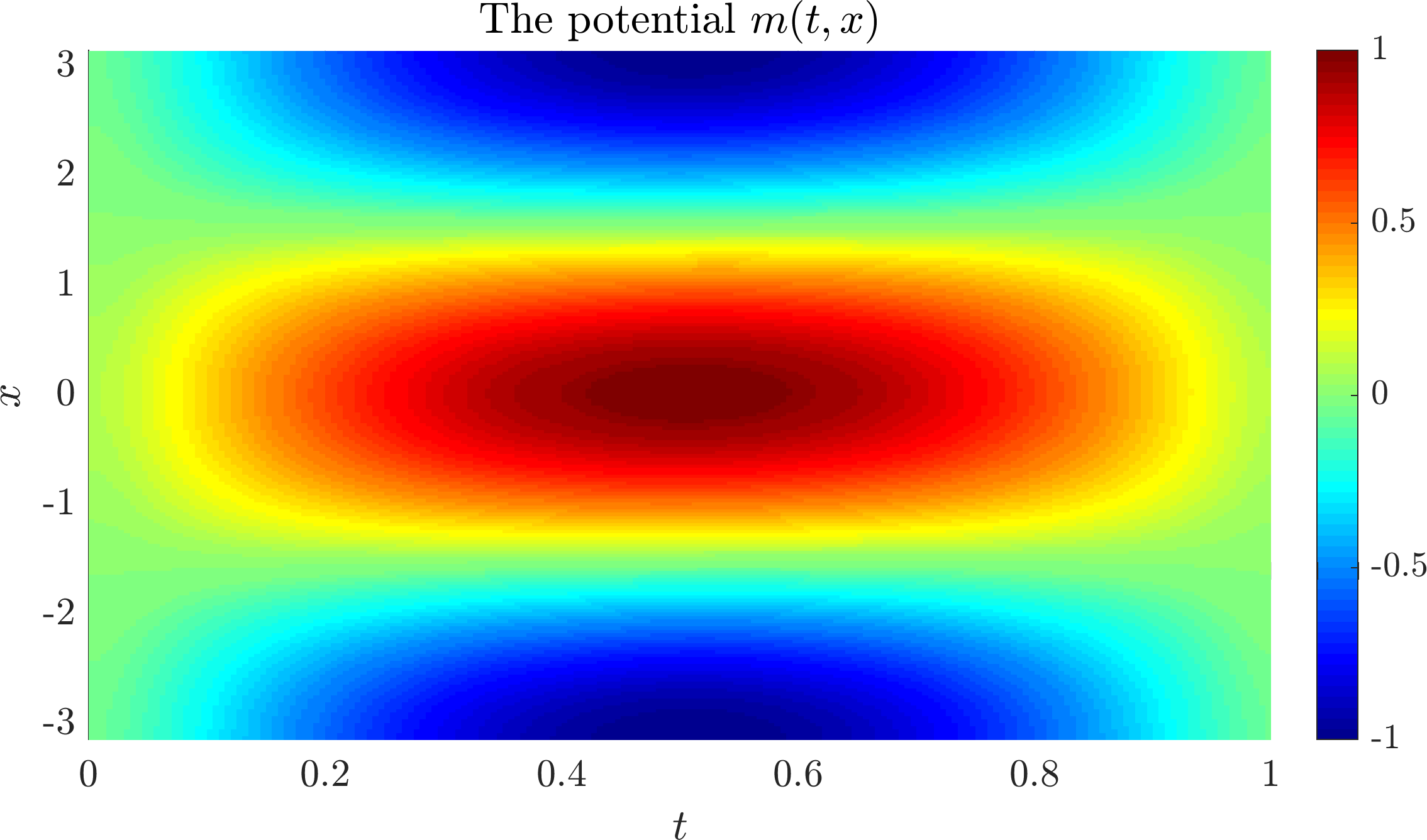}

     \includegraphics[width=0.48\linewidth]{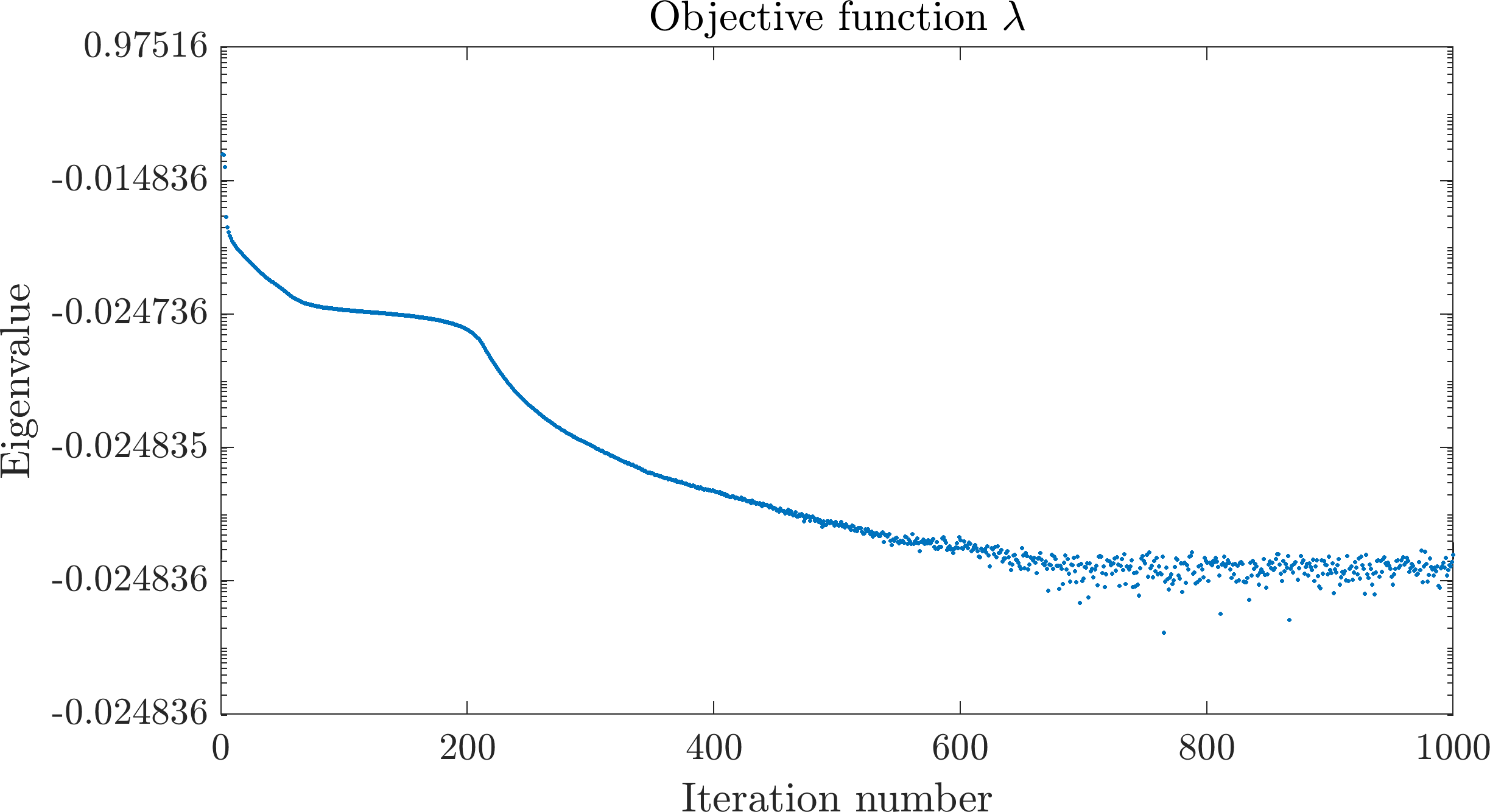}
     \includegraphics[width=0.48\linewidth]{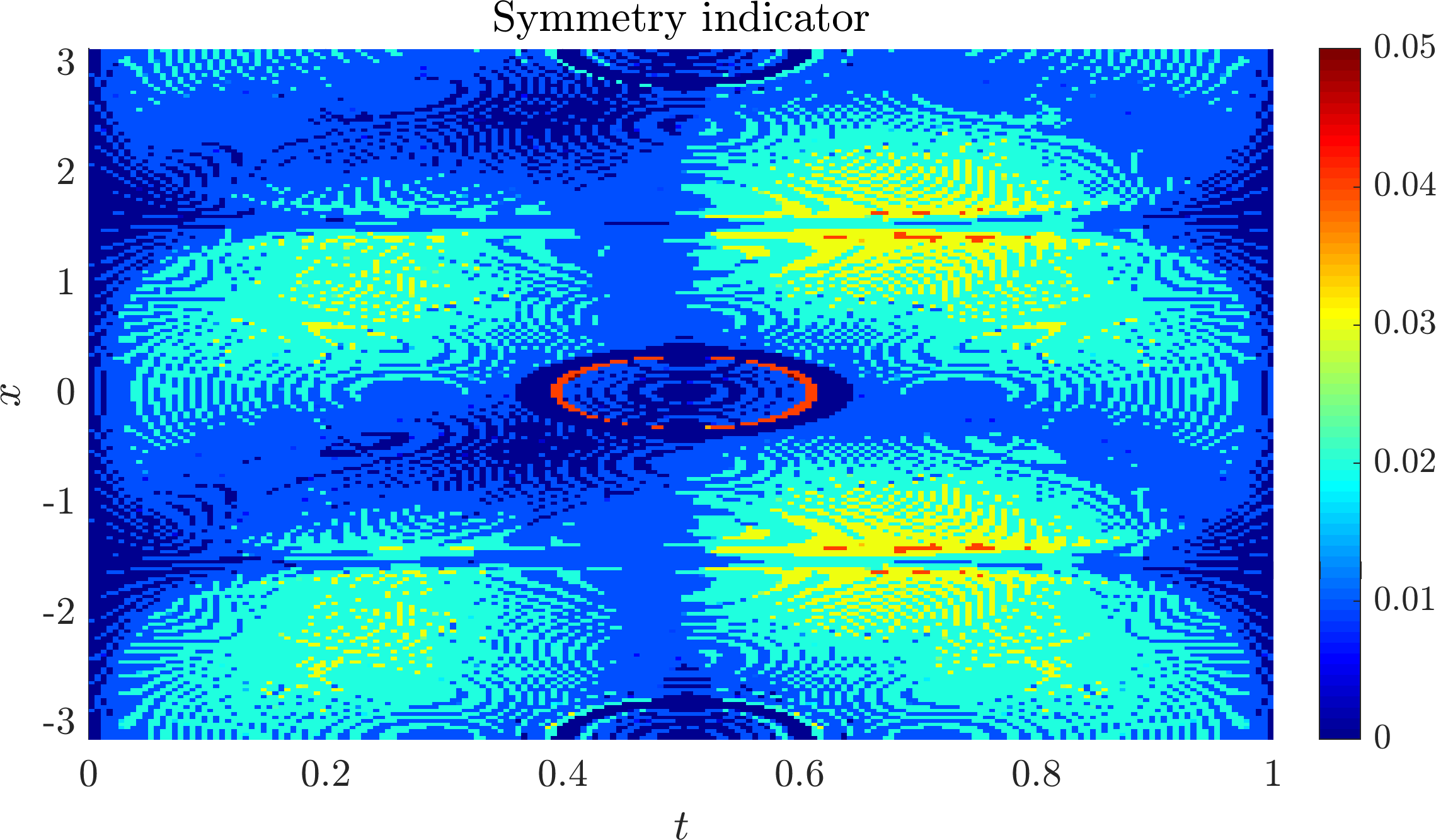}
    \caption{Rearrangement constraint $m^\#=\cos(x)\cos(2\pi t)$, $\alpha=0.1$, $\beta = 10$. The eigenfunction, optimal potential and objective function are plotted. The fourth image plots $\max\{|m(t,x)-m(T/2-t,x)|, |m(t,x)-m(t,-x)|\}$: small values indicate closeness to a symmetric potential.}
    \label{fig:rearrangement2D_1}
\end{figure}

Results are illustrated in Figure \ref{fig:rearrangement2D_1} where the optimal potential, corresponding eigenvalue and the evolution of the eigenvalue are shown. To investigate the symmetry of the numerical minimiser we shift the discrete solution such that the maximal value is at the center of the grid. The optimal potential is compared with the symmetrisation with respect to the vertical and horizontal direction. The maximal difference between $m$ and its symmetrisation is plotted. The values obtained are small enough to conjecture that the minimiser $m(t,x)$ is symmetric decreasing in time and space. The precise formula is: $\max\{|m(t,x)-m(T/2-t,x)|, |m(t,x)-m(t,-x)|\}$.

\bo{Computation 2.} In a second computation we consider a rearrangement constraint where the graph of $m^\#$ is a cone: functions taking values in $[0,1]$ with slice areas equal to $(1-h)^2 \pi \delta h$, where $\delta h$ is $1/K$ and $K=200$ is the number of horizontal slices used in the rearrangement constraint. The scaling parameters are $\alpha=0.1$, $\beta=1$ and the number of iterations is fixed to $250$. The numerical result obtained is shown in Figure \ref{fig:rearrangement2D_2} and is close to being symmetric decreasing in time and space.

\begin{figure}
    \centering
    \includegraphics[width=0.48\linewidth]{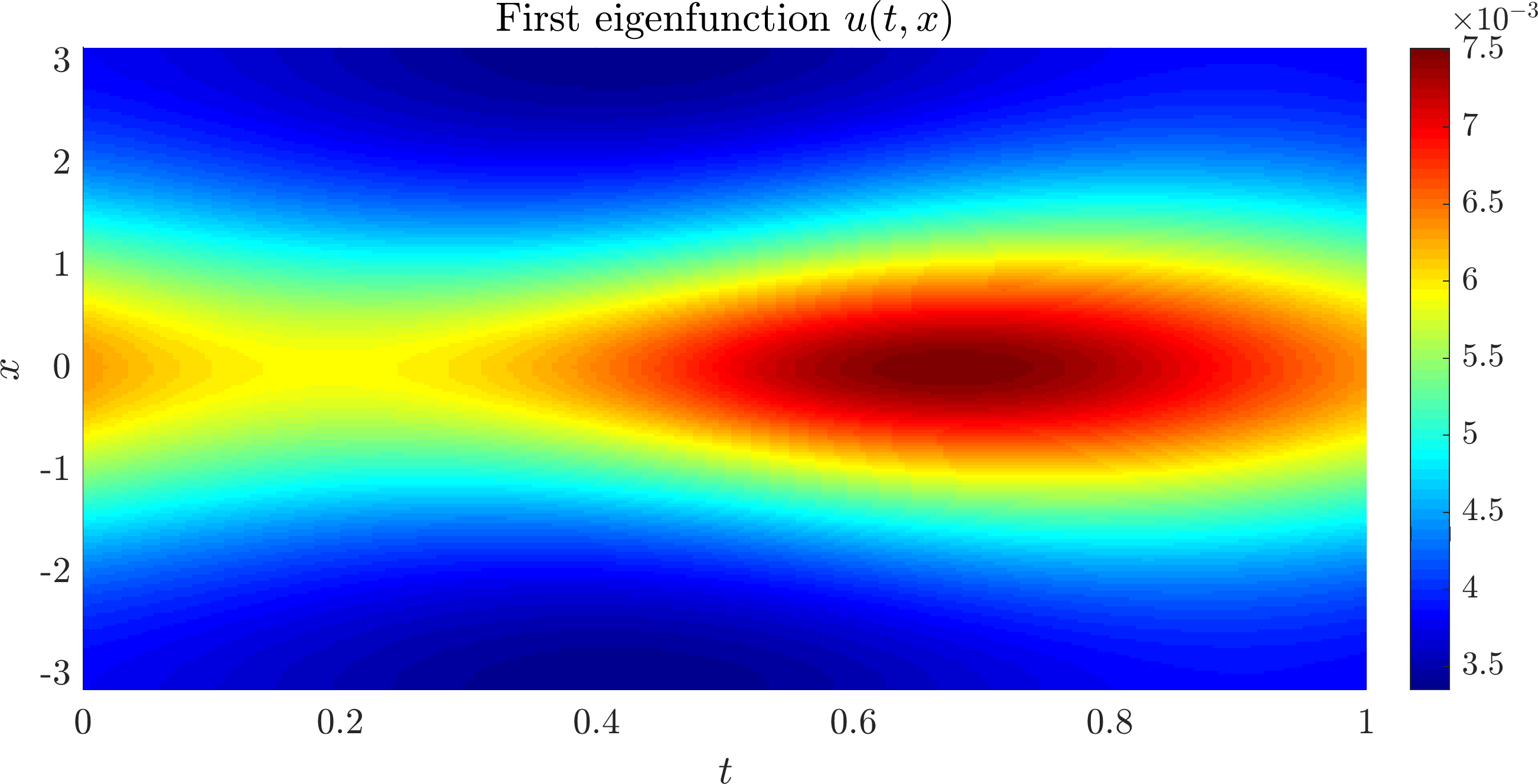}
    \includegraphics[width=0.48\linewidth]{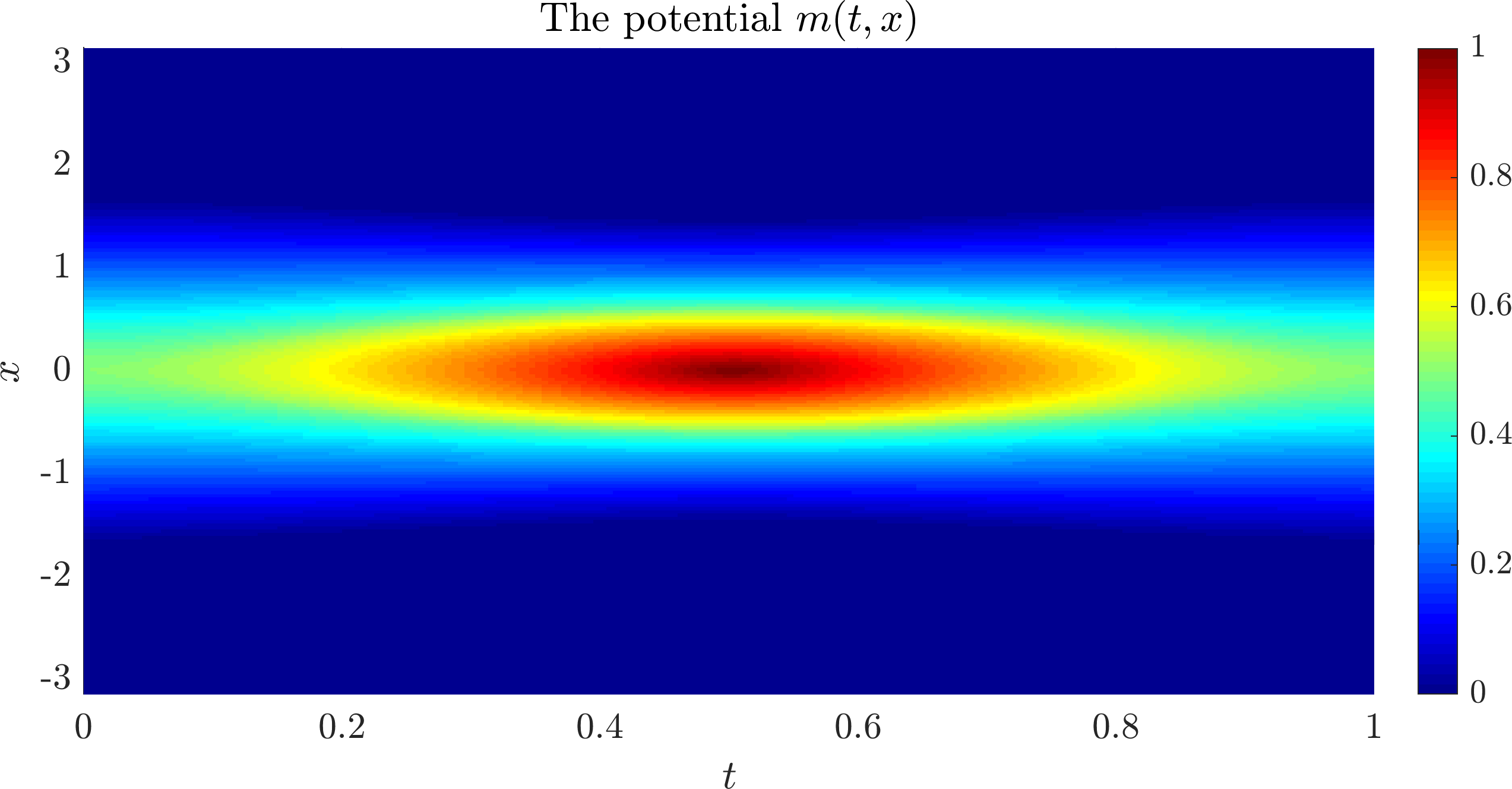}

     \includegraphics[width=0.48\linewidth]{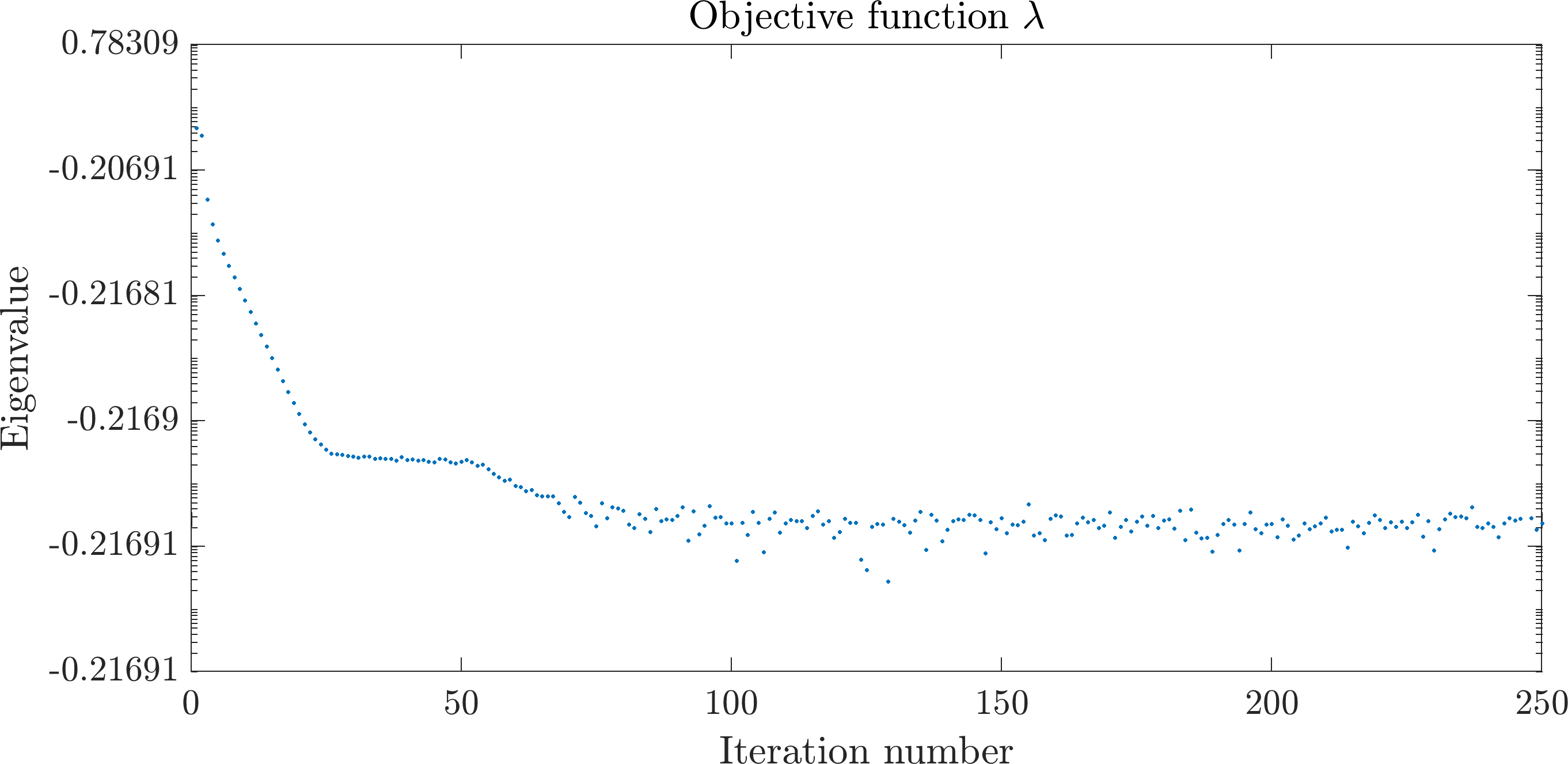}
     \includegraphics[width=0.48\linewidth]{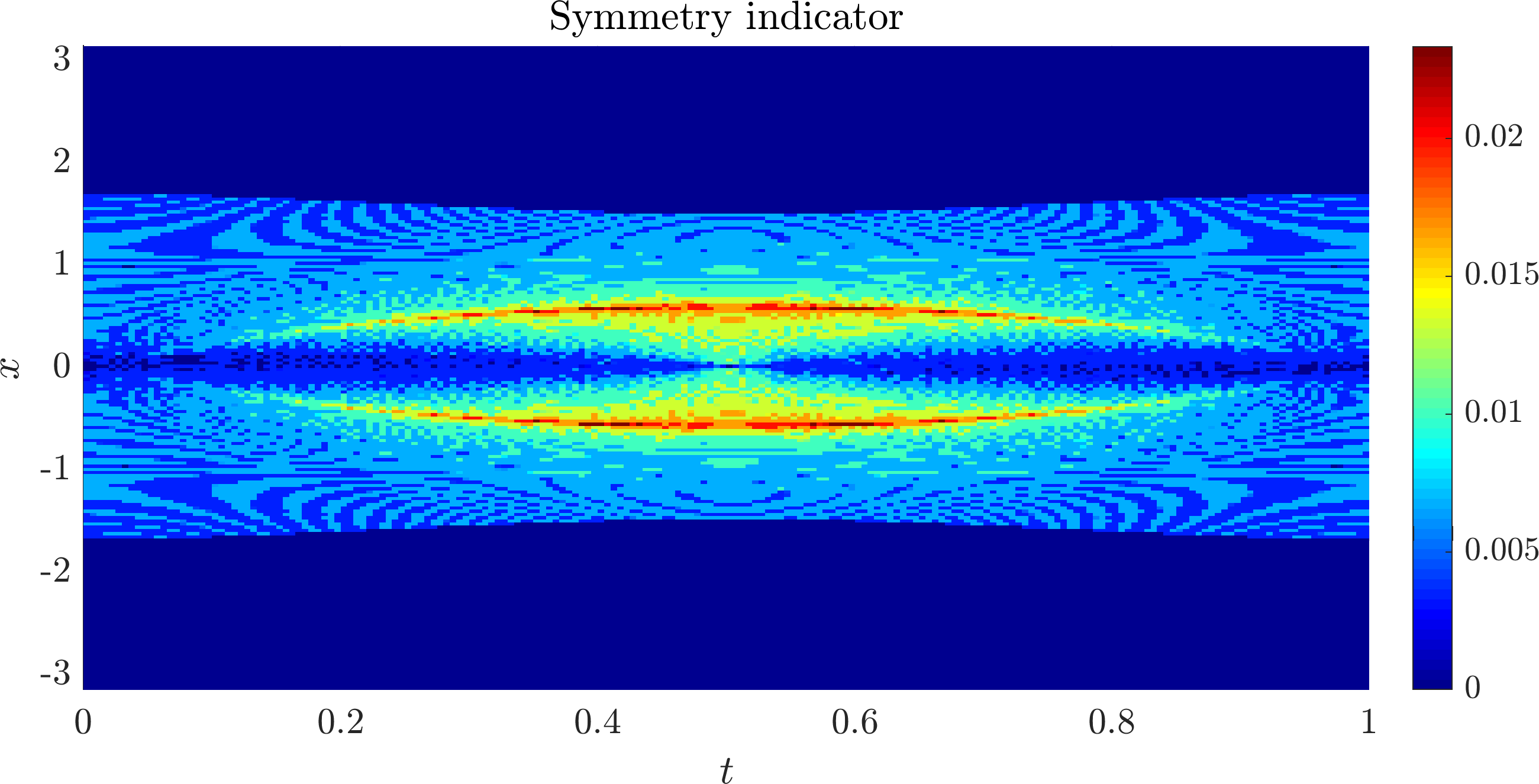}
    \caption{Rearrangement constraint: $m^\#$ is a cone, $\alpha=0.1$, $\beta = 1$. The eigenfunction, optimal potential and objective function are plotted. The fourth image plots $\max\{|m(t,x)-m(T/2-t,x)|, |m(t,x)-m(t,-x)|\}$: small values indicate closeness to a symmetric potential.}
    \label{fig:rearrangement2D_2}
\end{figure}

\section{Conclusion}
In this article, we were able to derive symmetry and monotonicity properties of optimisers of space-time periodic eigenvalues in various particular frameworks. Based on our results and on our numerical simulations, we conjecture that the optimiser should always be symmetric decreasing in time and space, but the tools required to prove it seem out of reach at the moment.

\bibliographystyle{abbrv}
\bibliography{BiblioHJB}

\begin{thebibliography}{10}

\bibitem{Alvino1986}
A.~Alvino, P.~Lions, and G.~Trombetti.
\newblock A remark on comparison results via symmetrization.
\newblock {\em Proceedings of the Royal Society of Edinburgh: Section A
  Mathematics}, 102(1-2):37--48, 1986.

\bibitem{Alvino1990}
A.~Alvino, P.-L. Lions, and G.~Trombetti.
\newblock Comparison results for elliptic and parabolic equations via {S}chwarz
  symmetrization.
\newblock {\em Annales de l'Institut Henri Poincare (C) Non Linear Analysis},
  7(2):37--65, Mar. 1990.

\bibitem{alvino1991}
A.~Alvino, P.-L. Lions, and G.~Trombetti.
\newblock Comparison results for elliptic and parabolic equations via
  symmetrization: a new approach.
\newblock {\em Differential Integral Equations}, 4(1):25--50, 1991.

\bibitem{AlvinoTrombettiLions}
A.~Alvino, P.-L. Lions, and G.~Trombetti.
\newblock Comparison results for elliptic and parabolic equations via
  symmetrization: a new approach.
\newblock {\em Differential Integral Equations}, 4(1):25--50, 1991.

\bibitem{Alvino1989}
A.~Alvino, G.~Trombetti, and P.~Lions.
\newblock On optimization problems with prescribed rearrangements.
\newblock {\em Nonlinear Analysis: Theory, Methods {\&} Applications},
  13(2):185--220, Feb. 1989.

\bibitem{Alvino2010}
A.~Alvino, R.~Volpicelli, and B.~Volzone.
\newblock Comparison results for solutions of nonlinear parabolic equations.
\newblock {\em Complex Variables and Elliptic Equations}, 55(5-6):431--443,
  Apr. 2010.

\bibitem{zbMATH03897446}
A.~Beltramo and P.~Hess.
\newblock On the principal eigenvalue of a periodic-parabolic operator.
\newblock {\em Commun. Partial Differ. Equations}, 9:919--941, 1984.

\bibitem{zbMATH02194918}
H.~Berestycki, F.~Hamel, and L.~Roques.
\newblock Analysis of the periodically fragmented environment model. {I}:
  {Species} persistence.
\newblock {\em J. Math. Biol.}, 51(1):75--113, 2005.

\bibitem{zbMATH02228673}
H.~Berestycki, F.~Hamel, and L.~Roques.
\newblock Analysis of the periodically fragmented environment model. {II}:
  {Biological} invasions and pulsating travelling fronts.
\newblock {\em J. Math. Pures Appl. (9)}, 84(8):1101--1146, 2005.

\bibitem{CantrellCosner}
R.~Cantrell and C.~Cosner.
\newblock Diffusive logistic equations with indefinite weights: population
  models in disrupted environments.
\newblock {\em Proceedings of the Royal Society of Edinburgh: Section A
  Mathematics}, 112(3-4):293--318, 1989.

\bibitem{zbMATH07226744}
C.~Carr{\`e}re and G.~Nadin.
\newblock Influence of mutations in phenotypically-structured populations in
  time periodic environment.
\newblock {\em Discrete Contin. Dyn. Syst., Ser. B}, 25(9):3609--3630, 2020.

\bibitem{zbMATH01028210}
D.~Daners.
\newblock Periodic-parabolic eigenvalue problems with indefinite weight
  functions.
\newblock {\em Arch. Math.}, 68(5):388--397, 1997.

\bibitem{zbMATH01525790}
D.~Daners.
\newblock Existence and perturbation of principal eigenvalues for a
  periodic-parabolic problem.
\newblock {\em Electron. J. Differ. Equ.}, 2000:51--67, 2000.

\bibitem{zbMATH01060813}
T.~Godoy, D.~E. Lami, and S.~Paczka.
\newblock The period parabolic eigenvalue problem with {{\(L^ \infty\)}}
  weight.
\newblock {\em Math. Scand.}, 81(1), 1997.

\bibitem{zbMATH02153800}
F.~Hamel, N.~Nadirashvili, and E.~Russ.
\newblock An isoperimetric inequality for the principal eigenvalue of the
  {Laplacian} with drift.
\newblock {\em C. R., Math., Acad. Sci. Paris}, 340(5):347--352, 2005.

\bibitem{zbMATH05960716}
F.~Hamel, N.~Nadirashvili, and E.~Russ.
\newblock Rearrangement inequalities and applications to isoperimetric problems
  for eigenvalues.
\newblock {\em Ann. Math. (2)}, 174(2):647--755, 2011.

\bibitem{Heinze}
S.~Heinze.
\newblock Large convection limits for {KPP} fronts.
\newblock Preprint, Max Planck Institute for Mathematics, 2005.

\bibitem{zbMATH00049232}
P.~Hess.
\newblock {\em Periodic-parabolic boundary value problems and positivity},
  volume 247 of {\em Pitman Res. Notes Math. Ser.}
\newblock Harlow: Longman Scientific \&| Technical; New York: John Wiley \&|
  Sons, Inc., 1991.

\bibitem{zbMATH05041278}
D.~Holcman and I.~Kupka.
\newblock Singular perturbation for the first eigenfunction and blow-up
  analysis.
\newblock {\em Forum Math.}, 18(3):445--518, 2006.

\bibitem{zbMATH01773170}
V.~Hutson, K.~Mischaikow, and P.~Pol{\'a}{\v{c}}ik.
\newblock The evolution of dispersal rates in a heterogeneous time-periodic
  environment.
\newblock {\em J. Math. Biol.}, 43(6):501--533, 2001.

\bibitem{zbMATH01579419}
V.~Hutson, W.~Shen, and G.~T. Vickers.
\newblock Estimates for the principal spectrum point for certain time-dependent
  parabolic operators.
\newblock {\em Proc. Am. Math. Soc.}, 129(6):1669--1679, 2001.

\bibitem{zbMATH05530397}
C.-Y. Kao, Y.~Lou, and E.~Yanagida.
\newblock Principal eigenvalue for an elliptic problem with indefinite weight
  on cylindrical domains.
\newblock {\em Math. Biosci. Eng.}, 5(2):315--335, 2008.

\bibitem{zbMATH07605268}
C.-Y. Kao and S.~A. Mohammadi.
\newblock Maximal total population of species in a diffusive logistic model.
\newblock {\em J. Math. Biol.}, 85(5):27, 2022.
\newblock Id/No 47.

\bibitem{zbMATH05042914}
S.~Kesavan.
\newblock {\em Symmetrization and applications}, volume~3 of {\em Ser. Anal.}
\newblock Hackensack, NJ: World Scientific, 2006.

\bibitem{zbMATH07668634}
K.-Y. Lam and Y.~Lou.
\newblock {\em Introduction to reaction-diffusion equations. {Theory} and
  applications to spatial ecology and evolutionary biology}.
\newblock Lect. Notes Math. Model. Life Sci. Cham: Springer, 2022.

\bibitem{zbMATH06690453}
J.~Lamboley, A.~Laurain, G.~Nadin, and Y.~Privat.
\newblock Properties of optimizers of the principal eigenvalue with indefinite
  weight and {Robin} conditions.
\newblock {\em Calc. Var. Partial Differ. Equ.}, 55(6):37, 2016.
\newblock Id/No 144.

\bibitem{Lions}
P.-L. Lions.
\newblock {Lectures at Coll\`ege de France, Th\'eorie Spectrale}.
\newblock 2020-2021.

\bibitem{zbMATH07122717}
S.~Liu, Y.~Lou, R.~Peng, and M.~Zhou.
\newblock Monotonicity of the principal eigenvalue for a linear time-periodic
  parabolic operator.
\newblock {\em Proc. Am. Math. Soc.}, 147(12):5291--5302, 2019.

\bibitem{zbMATH07517798}
I.~Mazari.
\newblock A note on the rearrangement of functions in time and on the parabolic
  {Talenti} inequality.
\newblock {\em Ann. Univ. Ferrara, Sez. VII, Sci. Mat.}, 68(1):137--145, 2022.

\bibitem{arXiv:2409.08740}
I.~Mazari-Fouquer.
\newblock Another look at qualitative properties of eigenvalues using effective
  {Hamiltonians}.
\newblock Preprint, {arXiv}:2409.08740 [math.{AP}] (2024), 2024.

\bibitem{zbMATH05610626}
G.~Nadin.
\newblock The principal eigenvalue of a space-time periodic parabolic operator.
\newblock {\em Ann. Mat. Pura Appl. (4)}, 188(2):269--295, 2009.

\bibitem{NadinPTW}
G.~Nadin.
\newblock Traveling fronts in space-time periodic media.
\newblock {\em J. Math. Pures Appl. (9)}, 92(3):232--262, 2009.

\bibitem{zbMATH05834183}
G.~Nadin.
\newblock The effect of the {Schwarz} rearrangement on the periodic principal
  eigenvalue of a nonsymmetric operator.
\newblock {\em SIAM J. Math. Anal.}, 41(6):2388--2406, 2010.

\bibitem{zbMATH05787082}
G.~Nadin.
\newblock Existence and uniqueness of the solution of a space-time periodic
  reaction-diffusion equation.
\newblock {\em J. Differ. Equations}, 249(6):1288--1304, 2010.

\bibitem{zbMATH05011372}
J.~Nolen, M.~Rudd, and J.~Xin.
\newblock Existence of {KPP} fronts in spatially-temporally periodic advection
  and variational principle for propagation speeds.
\newblock {\em Dyn. Partial Differ. Equ.}, 2(1):1--24, 2005.

\bibitem{zbMATH00764019}
R.~G. Pinsky.
\newblock Second order elliptic operators with periodic coefficients:
  {Criticality} theory, perturbations, and positive harmonic functions.
\newblock {\em J. Funct. Anal.}, 129(1):80--107, 1995.

\bibitem{zbMATH03531830}
G.~Talenti.
\newblock Elliptic equations and rearrangements.
\newblock {\em Ann. Sc. Norm. Super. Pisa, Cl. Sci., IV. Ser.}, 3:697--718,
  1976.

\bibitem{zbMATH03789116}
J.~L. Vazquez.
\newblock Symetrisation pour ${u_ t=\Delta\phi(u)}$ et applications.
\newblock {\em C. R. Acad. Sci., Paris, S{\'e}r. I}, 295:71--74, 1982.

\end{thebibliography}

\end{document}